 \RequirePackage{fix-cm}

 \documentclass[smallextended]{svjour3} 
 \smartqed  

\usepackage{amssymb,amsmath,amsfonts,latexsym} 
\usepackage{hyperref}
   
\usepackage{mathtools}
\usepackage{multirow,array}
\usepackage{graphicx}
\usepackage{subfigure}
\usepackage{epstopdf}
\usepackage{float} 
\usepackage{color, xcolor}
\usepackage{algorithm}
\usepackage{float}
\usepackage{algpseudocode}

\usepackage[nottoc]{tocbibind} 

\usepackage{booktabs}  
\usepackage{diagbox}

\usepackage{fancyhdr}
\usepackage[labelfont=bf]{caption}
\usepackage[font=it]{caption}

\usepackage{mathrsfs}

\usepackage{mathtools}
\usepackage{empheq} 
\usepackage{url}

\allowdisplaybreaks

\newtheorem{proposizione}{Proposition}

\DeclareMathOperator{\cond}{cond}

\journalname{Fract. Calc. Appl. Anal.} 

\begin{document}


\title{The univariate multinode Shepard method
 for the Caputo fractional derivatives: from
 Approximation to the solution of
 Bagley-Torvik equation}

\titlerunning{The univariate multinode Shepard method
 for the Caputo fractional derivatives}

\author{
        Francesco Dell'Accio$^1$ 
\and
        Filomena Di Tommaso$^1$ 
 \and
        Ilde Ferrara$^1$
 }

\authorrunning{F. Dell'Accio \and F. Di Tommaso \and I. Ferrara} 

\institute{Filomena Di Tommaso$^{1,*}$
\at
Department of Mathematics and Computer Science - University of Calabria, Via P. Bucci, cubo 30A -- 87036 Rende, Italy \\
\email{filomena.ditommaso@unical.it} $^*$ corresponding author 
 \and
Francesco Dell'Accio$^{1}$
\at
Department of Mathematics and Computer Science - University of Calabria, Via P. Bucci, cubo 30A -- 87036 Rende, Italy \\
\email{francesco.dellaccio@unical.it}
\and 
Ilde Ferrara$^1$
\at
Department of Mathematics and Computer Science - University of Calabria, Via P. Bucci, cubo 30A -- 87036 Rende, Italy \\
\email{ildeferrara99@gmail.com}
}

\date{Received:...  / Revised: ... / Accepted: ......}


\maketitle

\begin{abstract}
{In this paper, we approximate the fractional derivative of a given function using the univariate multinode Shepard method through the Gauss-Jacobi quadrature formula. Subsequently, the proposed method is applied to the numerical solution of boundary value problems (BVPs) and initial value problems (IVPs), specifically addressing the Bagley-Torvik equations.
Experimental results confirm the method's effectiveness, particularly in accurately approximating the Bagley-Torvik equation for both BVPs and IVPs.}

\keywords{Multinode Shepard operator \and  Caputo's fractional derivatives \and  Bagley Torvik equation \and  Boundary Value problems \and  Initial Value problems}

\subclass{65M70  \and 26A33 	\and  34A08 }

\end{abstract} 



\section{Introduction}
Fractional calculus is a very old mathematical concept, dating back to the 17th century, that deals with the integration and differentiation of arbitrary order \cite{Fractional_derivative_BTE,math7050407}. 
In the early stages of the history of integer-order calculus, several mathematicians, including L'Hospital, Leibniz, and others, began to explore the idea of fractional calculus, but the concept did not gain much attention or further development due to the lack of practical applications \cite{Fractional_derivative_book}. In recent decades, however, fractional calculus has experienced significant theoretical and practical advancements \cite{esempi_numerici}, with applications emerging in a wide range of fields such as physics and engineering \cite{Fractional_derivative_BTE}.

The numerical solution of fractional differential equations (FDEs) has become an area of growing interest, as these equations offer a more accurate explanation of many natural physical phenomena, facilitating the modeling of complex dynamic systems \cite{esempi_numerici}. In more recent years, both mathematicians and physicists have made considerable efforts to develop robust and stable numerical and analytical methods for solving fractional differential equations (FDEs) of physical relevance, driven by the increasing number of applications of FDEs \cite{Application_BTE}.
In particular, FDEs are used to characterize complex systems with memory and hereditary properties, such as anomalous diffusion, signal processing, stock market dynamics, vibration theory, electrical circuits, stability analysis, control theory, stochastic analysis, time-delay systems, and bioengineering (see \cite{article} and the references therein). They also have applications in thermodynamics, biophysics, blood flow phenomena, rheology, electrodynamics, electrochemistry, electromagnetism, continuum and statistical mechanics, and dynamical systems \cite{Fractional_derivative_BTE,Fractional_derivative_book,article2,article3}.
One of the main engineering applications of fractional calculus is in the study of the viscoelasticity of materials \cite{Fractional_derivative_legendre}, where the first significant application seems to have been made by Bagley and Torvik, in modeling the motion of a rigid plate immersed in a Newtonian fluid \cite{entropy,article4}.

As well known, fractional derivatives have various definitions, among which the most commonly used in practical applications are the Riemann-Liouville  and the Caputo fractional derivatives. The Riemann–Liouville definition is historically the first, and the theory behind it is now well established; however, there are some difficulties in applying it to real-world problems \cite{DEMIRCI20122754}. To solve these difficulties, Caputo reformulated the traditional Riemann–Liouville definition, allowing for the use of integer-order initial conditions when solving fractional-order differential equations \cite{article2}.
 
The Caputo's fractional derivative of order $\alpha >0$ of the function $f(x)$ is defined as
\begin{equation} \label{Caputo_non_esplicita}
    (D^{\alpha} f)(x)= \frac{1}{\Gamma(m - \alpha)} \int_0^x (x - t)^{m-\alpha-1} f^{(m)}(t) \, dt, \quad x>0,
\end{equation}
with $m-1<\alpha<m$, $m \in \mathbb{N}$ and 
 $\Gamma( \cdot) $ being the Gamma function.

This paper aims to approximate the fractional derivative of a given function by using the univariate multinode Shepard operator \cite{DellAccio2016,DellAccio:2018}. This operator was first introduced to solve global Birkhoff interpolation problems not otherwise solvable, even within appropriate polynomial or rational spaces. The idea is to decompose the initial problem into subproblems, each with a unique polynomial solution, and then to use multinode rational functions, based on the nodes of the decomposition, to obtain a global interpolant.

We apply the proposed fractional derivative approximation to numerically solve, via collocation, boundary value problems (BVPs) and initial value problems (IVPs), which involve the Bagley-Torvik equation \cite{BagleyTorvik1984}.

The paper is organized as follows. In Section \ref{sec:multinode} we introduce some preliminary notations and definitions regarding the multinode Shepard method as well as we compute the first and second order derivative of the multinode functions. In Section \ref{section2} we present the application of the multinode Shepard operator in approximating the fractional derivative of a given function $f$. In Section \ref{sec:Bagley-Torvik} we describe the application of the multinode Shepard operator in numerically solving the Bagley-TOrvik equation both in the case of Boundary Value Problems (BVPs) and Initial Value Problems (IVPs). In Section \ref{sec:nodes} we introduce the set of nodes used for the numerical experiments. Finally, in Section \ref{sec:numerical_experiments} we test the performance and show the accuracy of the proposed method.
\section{The multinode Shepard method}
\label{sec:multinode}
\subsection{Multinode functions}
Let $I$ be a closed interval and $X=\left\{  x_{1},x_{2},\dots,x_{n}\right\} \subset I  $ be a set of pairwise distinct nodes, such that $x_{1}<x_{2}%
<\dots<x_{n}$, $n \in \mathbb{N}$. 
Let us assume that a covering $\mathcal{F}=\{F_{1},F_{2},\dots,F_{s}\}$ of $X$ by non-empty subsets $F_{k}=\{ x_{k_1},...,x_{k_p}\}\subset {X}$ is given.
The multinode functions with respect to the covering $\mathcal{F}$ are
defined by%
\begin{equation}%
\begin{array}
[c]{cc}%
B_{\mu,k}(x)=\frac{\displaystyle \prod\limits_{i=1}^{p}\left\vert {x-x_{k_i}%
}\right\vert ^{-\mu}}{ \displaystyle \sum\limits_{\ell=1}^{s}\prod\limits_{i=1}^{p}\left\vert {x-x_{\ell_i}}\right\vert ^{-\mu}}, \quad k=1,\dots,s,  & x \in I,
\end{array}
\label{multinodebasis}%
\end{equation}
where $\mu>0$ is a parameter related to the differentiability class of $B_{\mu,k}(x)$ \cite{DellAccio:2018}. 
The multinode functions  are non-negative and form a partition of unity, and instead of being cardinal, they satisfy the following properties \cite{DellAccio:2018}:
\begin{itemize}
    \item[(i)] $B_{\mu,k}\left(  x_{j}\right)  =0$, for $\mu>0$;
\item[(ii)] $\sum\limits_{k\in \mathcal{K}_{i}}B_{\mu,k}\left(  x_{i}\right)  =1$, for $\mu>0$;
\item[(iii)] $B_{\mu,k}^{\left(  m \right)  }\left(  x_{j}\right)  =0$, for $\mu>m\geq 1$; 
\item[(iv)] $\sum\limits_{k\in \mathcal{K}_{i}}B_{\mu,k}^{\left(  m \right)  }\left(  x_{i}\right)
=0$, for $\mu>m\geq 1$; 
\end{itemize}
where%
\begin{equation} \label{K_i}
\mathcal{K}_{i}= \{k\in\{1,...,s\}:i\in \{k_1,...,k_p \} \} \neq\emptyset,
\end{equation}
is the set of indices of all subsets of $\mathcal{F}$ that contain $x_{i}$.

\subsection{First and second order derivatives of the multinode functions}
In this Section, we compute the first and second order derivatives of the multinode functions in view of the application of the multinode Shepard operator to the approximation of fractional derivatives.
\begin{proposizione} \label{prop:derivata_prima_Bmuk}
    For $\mu >1$ and even, we have 
    \begin{equation} \label{derivata_prima_B}
    B_{\mu,k}^{\prime}(x)=  \mu B_{\mu,k}(x) \left( \sum_{\ell=1}^{s} B_{\mu,\ell}(x) \sum_{i=1}^{p} \frac{1}{ x-x_{\ell_i}} -\sum_{i=1}^{p} \frac{1}{ x-x_{k_i}}\right) .
\end{equation}
\end{proposizione}
\begin{proof}
	Let be $x_j \in X$. If we multiply both the numerator and the denominator of (\ref{multinodebasis}%
	) by $\left\vert x-x_{j}\right\vert ^{\mu}$, we get
	\begin{equation} \label{B_mu_C_l}
		B_{\mu,k}\left(  x\right)  =\frac{C_{k}\left(  x\right)  }{\sum\limits_{\ell=1}%
			^{s}C_{\ell}\left(  x\right)  },
	\end{equation}
	where
	\begin{equation} \label{C_l}
		\begin{array}
			[c]{cc}%
			C_{\ell}\left(  x\right)  =\left\vert x-x_{j}\right\vert ^{\mu}\prod
			\limits_{i=1}^{p} {\left\vert {x-x_{\ell_i}}\right\vert ^{-\mu}}, &
			\ell=1,\dots,s.
		\end{array}
	\end{equation}
    Since $\mu>1$ and even, we can rewrite  \eqref{C_l} as
    \begin{equation} \label{C_k2}
    C_{\ell}(x)= \left( x-x_j\right)^{\mu}D_{\ell}(x),
\end{equation}
where
\begin{equation} \label{D_k}
    D_\ell(x)=\prod_{i=1}^{p}(x-x_{\ell_i})^{-\mu}.
\end{equation}
    By differentiating the right-hand side of \eqref{B_mu_C_l}, we get
    \begin{equation} \label{derivata_prima_B1}
    B_{\mu,k}^{\prime}(x)=\frac{C_{k}^{\prime}(x)\displaystyle \sum_{\ell=1}^{s}C_{\ell}(x)-C_{k}(x)\sum_{\ell=1}^{s}C_{\ell}^{\prime}(x)}{\left( \displaystyle \sum_{\ell=1}^{s}C_{\ell}(x) \right)^2} = \frac{C_{k}^{\prime}(x)}{\displaystyle \sum_{\ell=1}^{s}C_{\ell}(x)} -\frac{\displaystyle \sum_{\ell=1}^{s}C_{\ell}^{\prime}(x)}{\displaystyle \sum_{\ell=1}^{s}C_{\ell}(x)}B_{\mu,k}(x).
\end{equation}
The final step is the computation of $C_{\ell}^{\prime}(x)$ in \eqref{C_k2}, for which we get
\begin{equation} \label{C_k^prime}
    C_\ell^{\prime}(x)=\mu(x-x_j)^{\mu-1}D_\ell(x)+(x-x_j)^{\mu} D_\ell^{\prime}(x).
\end{equation}
To compute $D_\ell^{\prime}(x)$, we notice that, since $\mu$ is even, it is possible to compute the logarithm of both sides of \eqref{D_k}, that is
\begin{equation} \label{Log_D_k}
    \log D_\ell(x)=\log \prod_{i=1}^{p}(x-x_{\ell_i})^{-\mu}= \sum_{i=1}^{p} (-\mu) \log (x-x_{\ell_i}).
\end{equation}
By differentiating both sides of \eqref{Log_D_k}, we get 
\begin{equation*}
    \frac{D_\ell^{\prime}(x)}{D_\ell(x)}=-\mu \sum_{i=1}^{p} \frac{1}{x-x_{\ell_i}}
\end{equation*}
and then 
\begin{equation} \label{D^prime}
    D_\ell^{\prime}(x)=-\mu D_\ell(x) \sum_{i=1}^{p} \frac{1}{x-x_{\ell_i}}.
\end{equation}
By substituting \eqref{D^prime} in \eqref{C_k^prime} and by some computations, we obtain
\begin{equation}
    C_\ell^{\prime}(x)= \mu C_\ell(x) \left( \frac{1}{x-x_j} -\sum_{i=1}^{p} \frac{1}{x-x_{\ell_i}} \right). \label{C^prime_mod}
\end{equation}
Finally, by substituting \eqref{C^prime_mod} in \eqref{derivata_prima_B1} and by the partition of unity property, we get

\begin{equation*}
    \begin{split}
        B_{\mu,k}^{\prime}(x) & = \frac{\displaystyle \mu C_k(x) \left( \frac{1}{x-x_j} -\sum_{i=1}^{p} \frac{1}{x-x_{k_i}}\right) }{\displaystyle\sum_{\ell=1}^{s}C_{\ell}(x)} -\frac{\displaystyle \sum_{\ell=1}^{s}\mu C_{\ell}(x) \left( \frac{1}{x-x_j} -\sum_{i=1}^{p} \frac{1}{x-x_{\ell_i}} \right) }{\displaystyle\sum_{\ell=1}^{s}C_{\ell}(x)}B_{\mu,k}(x) =\\
        &= \mu B_{\mu,k}(x) \left( \sum_{\ell=1}^{s} B_{\mu,\ell}(x) \sum_{i=1}^{p} \frac{1}{x-x_{\ell_i}} -\sum_{i=1}^{p} \frac{1}{x-x_{k_i}}\right). \\
    \end{split}
\end{equation*}

\end{proof}

\begin{proposizione} \label{prop:derivata_seconda_Bmuk}
    For $\mu >2$ and even, we have 
\begin{equation} \label{derivata_seconda_B1}
    \begin{split}
        B_{\mu,k}^{\prime \prime}(x) & = - \mu B_{\mu,k}(x) \left(\sum_{\ell=1}^{s}B_{\mu,\ell}(x) \sum_{i=1}^{p} \frac{1}{(x-x_{\ell_i})^2}- \sum_{i=1}^{p} \frac{1}{(x-x_{k_i})^2} \right) \\
        &  + \mu B_{\mu,k}(x) \sum_{\ell=1}^{s}B_{\mu,\ell}^{\prime}(x)\sum_{i=1}^{p} \frac{1}{x-x_{\ell_i}} \\
        & + \mu B_{\mu,k}^{\prime}(x) \left( \sum_{\ell=1}^{s}B_{\mu,\ell}(x)\sum_{i=1}^{p} \frac{1}{x-x_{\ell_i}} -\sum_{i=1}^{p} \frac{1}{x-x_{k_i}} \right) .\\
    \end{split}
\end{equation}

\end{proposizione}
\begin{proof}
    Let us fix $\mu>2$ and even.
    By the equations \eqref{B_mu_C_l}, \eqref{C_k2} and \eqref{D_k}, by differentiating the right-hand side of \eqref{derivata_prima_B} and by rearranging the resulting terms, we get 
\begin{equation}  \label{derivata_seconda_B}
        B_{\mu,k}^{\prime \prime}(x) = \frac{\displaystyle C_{k}^{\prime \prime}(x)}{\displaystyle \sum_{\ell=1}^{s}C_{\ell}(x)} - \frac{\displaystyle \sum_{\ell=1}^{s}C_{\ell}^{\prime \prime}(x) }{\displaystyle \sum_{\ell=1}^{s}C_{\ell}(x)}B_{\mu k}(x) - 2 \frac{\displaystyle \sum_{\ell=1}^{s}C_{\ell}^{\prime}(x)}{\displaystyle \sum_{\ell=1}^{s}C_{\ell}(x)} B_{\mu k}^{\prime}(x). 
\end{equation}
In line with the proof of Proposition \ref{prop:derivata_prima_Bmuk}, we need to compute the second order derivative of $C_{\ell}(x)$, that is
\begin{equation} \label{C_k^''}
    C_\ell^{\prime \prime}(x)= \mu (\mu-1) (x-x_j)^{\mu-2}D_\ell(x)+2\mu(x-x_j)^{\mu-1}D_\ell^{\prime}(x)+(x-x_j)^{\mu}D_\ell^{\prime \prime}(x)
\end{equation}
and, consequently, the second order derivative of $D_\ell(x)$
\begin{equation} \label{D^second}
    D_\ell^{\prime \prime}(x)= -\mu D_\ell^{\prime}(x) \sum_{i=1}^{p} \frac{1}{x-x_{\ell_i}} +\mu D_\ell(x) \sum_{i=1}^{p} \frac{1}{(x-x_{\ell_i})^2}.
\end{equation}
By substituting \eqref{D^prime} and \eqref{D^second} in \eqref{C_k^''}, after some computation, we obtain
\begin{equation}\label{C^second_mod}
    \begin{split}
        C_\ell^{\prime \prime}(x)= \mu \left( - \frac{C_\ell(x)}{(x-x_j)^2}+ \sum_{i=1}^{p} \frac{ C_\ell(x)}{(x-x_{\ell_i})^2}  + \frac{ C_\ell^{\prime}(x) }{x-x_j} -\sum_{i=1}^{p} \frac{C_\ell^{\prime}(x) }{x-x_{\ell_i}} \right). 
    \end{split}
\end{equation}
Finally, by substituting \eqref{C^second_mod} in \eqref{derivata_seconda_B}, by the partition of unity property and by some computations, we get
\begin{equation*}
    \begin{split}
        B_{\mu,k}^{\prime \prime}(x)   &= - \mu B_{\mu,k}(x) \left(\sum_{\ell=1}^{s}B_{\mu,\ell}(x) \sum_{i=1}^{p} \frac{1}{(x-x_{\ell_i})^2}- \sum_{i=1}^{p} \frac{1}{(x-x_{k_i})^2}  \right) \\
        &+ \mu B_{\mu,k}(x) \sum_{\ell=1}^{s}B_{\mu,\ell}^{\prime}(x)\sum_{i=1}^{p} \frac{1}{x-x_{\ell_i}}  \\
        & + \mu B_{\mu,k}^{\prime}(x) \left(\sum_{\ell=1}^{s}B_{\mu,\ell}(x)\sum_{i=1}^{p} \frac{1}{x-x_{\ell_i}} -\sum_{i=1}^{p} \frac{1}{x-x_{k_i}} \right). \\
    \end{split}
\end{equation*}

\end{proof}
By following an analogous procedure to that of Proposition \ref{prop:derivata_prima_Bmuk} and Proposition \ref{prop:derivata_seconda_Bmuk}, we obtain the general formula
\begin{equation*}
    B_{\mu,k}^{(r)}(x) =\frac{C_{k}^{(r)}(x)} {\displaystyle\sum_{\ell=1}^{s} C_{\ell}(x)} - \sum_{i=0}^{n-1} \binom{r}{i} \left( \frac{ \displaystyle \sum_{\ell=1}^{s} C_{\ell}^{(r-i)}(x)}{\displaystyle \sum_{\ell=1}^{s} C_{\ell}(x)}B_{\mu,k}^{(i)}(x)\right).
\end{equation*}

\subsection{The Multinode Shepard operator}
Let $f: I \rightarrow \mathbb{R}$ be a function of class $C^{m}$ in $I$, $ m \ge 0$, sampled at the point $x_i \in X$, and let us denote such values by $f_i=f(x_i), \ i=1,...,n$.
Moreover, we consider the interpolation polynomials $P_{k}\left[  f\right]  \left(  x\right)  $  of degree $p-1$ on the points in $F_k$, $k=1,\dots,s$,  and we combine them with the multinode functions $B_{\mu,k}(x)$ to get the multinode Shepard operator
\begin{equation}
\mathcal{M}_{\mu}\left[  f\right]  (x)=\sum\limits_{k=1}^{s}B_{\mu,k}(x)P_{k}[f](x), \quad \mu>0 , \, x\in I.
\label{multinode_operator}%
\end{equation}
The multinode Shepard operator is an interpolation operator at the points of $X$ and reproduces polynomials up to the degree $p-1$ \cite[Proposition 4]{DellAccio:2018}.

In view of the following applications, it is convenient to rewrite $\mathcal{M}_{\mu}\left[  f \right] $ as
\begin{equation*}
	\displaystyle \mathcal{M}_{\mu}\left[  f \right](x) = \sum_{k=1}^ {s} \sum_{i=1}^{p} B_{\mu,k}(x)l_{k,i}(x)f_{k_i}, 
\end{equation*} 
where $l_{k,i}(x)$ are the fundamental Lagrange polynomials on the nodes of $F_k$, that is
\begin{equation}
	l_{k,i}(x)=\prod_{\substack{j=1\\ j\neq k_i}}^{p} \frac{x-x_j}{x_{k_i}-x_j},
\end{equation}
satisfying the well-known Kronecker delta property
\begin{align} \label{delta_kroneker}
	l_{k,i}(x_{k_j})= \delta_{ij} =\begin{cases}
		1, \qquad i=j\\
		0,  \qquad \text{otherwise}.
	\end{cases}
\end{align}
Moreover, if we identify the subsets $\mathcal{K}_i$ as in \eqref{K_i}, we can rewrite $\mathcal{M}_{\mu}\left[  f \right]  (x)$ as
\begin{equation} \label{M_riscritto}
	\mathcal{M}_{\mu}\left[  f \right]  (x) = \sum_{i=1}^{n}\sum_{j\in \mathcal{K}_i} B_{\mu,j}(x)l_{j,i}(x) f_{i}.
\end{equation}
After these rearrangements, it is easy to compute the successive derivatives of $\mathcal{M}_{\mu}$ by the formula 
\begin{equation} \label{M_mu_derivata}
	\mathcal{M}_{\mu}^{(m)}\left[  f\right]  (x) = \sum_{i=1}^{n}g_i^{(m)}(x)f_{i}
\end{equation}
where
\begin{equation}
    g_i^{(m)}(x)=\sum_{j\in \mathcal{K}_i}\left(B_{\mu,j}(x)l_{j,i}(x) \right)^{(m)}.
\end{equation}

\section{Approximation of the Caputo fractional derivative} \label{section2}
The goal of this Section is to approximate, for any $x>0$ and $ m-1<\alpha<m$, $m \in \mathbb{N}$, the fractional derivative
\begin{equation}
    (D^{\alpha} f)(x) = \frac{1}{\Gamma(m - \alpha)} \int_0^x (x - t)^{m-\alpha-1} f^{(m)}(t) \, dt, 
    \label{fractional_deriavive_f}
\end{equation}
of a given function $f\in C^{m}\left( I \right)$, by using the univariate multinode Shepard operator and the Gauss-Jacobi quadrature formula. 

To this aim we approximate $f $ with $\mathcal{M}_{\mu}\left[  f \right]$  and by substituting the expression \eqref{M_mu_derivata} in equation \eqref{fractional_deriavive_f}, we get
\begin{equation} \label{fractional_derivative_with_M}
    (D^{\alpha} f )(x) \approx(D^{\alpha} \mathcal{M}_{\mu}\left[  f \right])(x) 
    =\frac{1}{\Gamma(m - \alpha)} \sum_{i=1}^{n} f_{i} \int_0^x (x - t)^{m-\alpha-1} g_i^{(m)}(t) \, dt.
\end{equation}
To compute the integral
\begin{equation*}
    \int_0^x (x - t)^{m-\alpha-1} g_{i}^{(m)}(t) \, dt
\end{equation*}
we use the Gauss-Jacobi quadrature formula of order $N=\lceil \frac{n-m}{2} \rceil$  and we obtain
\begin{equation*}
    \begin{split}
        \int_0^x (x - t)^{m-\alpha-1} g_{i}^{(m)}(t) \, dt \approx \left(\frac{x}{2} \right)^{m-\alpha} \sum_{k=1}^{N} w_k^{(m-\alpha-1,0)} g_{i}^{(m)} \left( \frac{x}{2}\cdot  x_k^{(m-\alpha-1,0)} +\frac{x}{2}\right) 
    \end{split}
\end{equation*}
where $x_k^{(m-\alpha-1,0)}$ and $w_k^{(m-\alpha-1,0)}$ are the Jacobi nodes and weights, respectively \cite{Gautschi}.
By substituting this result into equation \eqref{fractional_derivative_with_M}, we finally get
\begin{equation}
        (D^{\alpha} \mathcal{M}_{\mu}[f])(x)  = \frac{1}{\Gamma(m - \alpha)} \left(\frac{x}{2} \right)^{m-\alpha}\sum_{i=1}^{n} \sum_{k=1}^{N} w_k^{(m-\alpha-1,0)} g_{i}^{(m)} \left( \frac{x}{2} \cdot x_k^{(m-\alpha-1,0)} +\frac{x}{2}\right) f_{i}.
    \label{fractional_derivative_approx}
\end{equation}

\section{Approximation of the solution of the Bagley-Torvik equation}
\label{sec:Bagley-Torvik}
In 1983, Bagley and Torvik formulate an equation to study the viscoelastically damped structures, and they used the equation to investigate the behaviour of real materials  \cite{Bagley1983FractionalCD,BagleyTorvik1984}.
This equation, called the \textit{Bagley–Torvik equation}, plays an important role in a really large number of applied science and engineering problems. 
It can also describe the motion of real physical systems, the modeling of the motion of a rigid plate immersed in a viscous fluid and a gas in a fluid, respectively.
The use of fractional derivatives allows modeling the anomalous behavior of viscoelastic materials and other complex systems with intrinsic damping \cite{entropy}.
A general formulation of the Bagley-Torvik equation is
\begin{equation*}
    \rho y^{\prime \prime}(x) + \lambda (D^{\alpha} y)(x) + \sigma y(x)= h(x), \quad x \geq 0,
\end{equation*}
with $m-1 < \alpha<m$, $m \in \mathbb{N}$, $\rho,\lambda,\sigma \in \mathbb{R}$ and  $h(x)$ a known continuous function.
In this Section, we apply the multinode Shepard method to solve Bagley-Torvik problems with different initial conditions.

\subsection{The Bagley-Torvik Boundary Value Problem}
\label{sec:BVP}
Let us consider the Boundary Value Problem (BVP)
\begin{equation} \label{eq:BVP}
\left\{
\begin{array}{ll}
     \rho y^{\prime \prime}(x)+\lambda (D^{\alpha}y)(x)+ \sigma y(x)=h(x), & x\in \left(0,T\right) \\
     y(0)=\gamma_1, & \\
     y(T) =\gamma_2, &
\end{array}
\right.
\end{equation}
with $m-1 < \alpha <m$, $m=1,2$, $\gamma_1,\gamma_2 \in \mathbb{R}$.

We aim to approximate the solution $y(x)$ of the BVP problem \eqref{eq:BVP} by the multinode Shepard operator $\mathcal{M}_{\mu}[y](x)$. 
Then, we choose ad hoc a set of collocation points $X=\left\{x_1,\dots,x_n\right\}$, with $0=x_1<x_2<\dots < x_{n-1}<x_n=T$, and impose differential and boundary conditions in \eqref{eq:BVP} exactly at that nodes
\begin{equation}  \label{eq:formula}
  \left\{  
  \begin{array}{ll}
     \rho \mathcal{M}^{\prime \prime}_{\mu}[y](x_j)+\lambda (D^{\alpha} \mathcal{M}_{\mu}[y])(x_j)+ \sigma \mathcal{M}_{\mu}[y](x_j)=h(x_j),    & j=2,\dots,n-1, \\
     \mathcal{M}_{\mu}[y])(x_1)=\gamma_1,  & \\
     \mathcal{M}_{\mu}[y])(x_n)=\gamma_2,  & 
  \end{array}
 \right.
\end{equation}
by assuming unknown the values $y_i=y(x_i)$, $i=1,\dots,n$. 
By using equations \eqref{M_riscritto}, \eqref{M_mu_derivata} and \eqref{fractional_derivative_with_M}, we get
\begin{equation}
\begin{split}
       \rho \sum_{i=1}^{n}g_{i}^{\prime \prime }(x_j) y_{i} &+ \frac{\lambda}{\Gamma(m-\alpha)} \left(\frac{x_j}{2} \right)^{m-\alpha}\sum_{k=1}^{N} w_k^{J}  \sum_{i=1}^{n} y_{i}  g_{i}^{(m)}\left( \frac{x_j}{2}\cdot x_k^{J} +\frac{x_j}{2}\right) \\
       &+ \sigma\sum_{i=1}^{n}g_{i}(x_j) y_{i} =h(x_j), \qquad j=2,\dots,n-1,
\end{split}
\label{differential_relation}
 \end{equation}
with $w_k^{J}$ and $x_k^{J}$ being the Jacobi weights and nodes associated to the Gauss-Jacobi quadrature formula of order $N=\lceil \frac{n-m}{2}\rceil$.
Let $j\in \left\{ 2,\dots,n-1\right\}$.
By rearranging the terms in \eqref{differential_relation}, we obtain
\begin{equation} \label{sostituzione_problema}
    \sum_{i=1}^{n} \big[ \rho g_{i}^{\prime \prime }(x_j) +  \frac{\lambda }{\Gamma(m-\alpha)} \left(\frac{x_j}{2} \right)^{m-\alpha}\sum_{k=1}^{N} w_k^{J} g_{i}^{(m)} \left( \frac{x_j}{2}\cdot x_k^{J} +\frac{x_j}{2}\right) + \sigma g_{i}(x_j) \big] y_{i} =h(x_j),
\end{equation}
and by setting 
\begin{equation*}
    A_i(x_j)= \rho g_{i}^{\prime \prime }(x_j) + \frac{\lambda}{\Gamma(m-\alpha)} \left(\frac{x_j}{2} \right)^{m-\alpha} \sum_{k=1}^{N} w_k^{J}  g_{i}^{(m)} \left( \frac{x_j}{2}\cdot x_k^{J} +\frac{x_j}{2}\right) + \sigma g_{i}(x_j), 
\end{equation*}
the equations \eqref{differential_relation} become
\begin{equation} \label{sostituzione_problema_finale}
    \sum_{i=1}^{n} A_i(x_j) y_{i} =h(x_j).
\end{equation}
By taking into account that  $y_1 = y(0)=\gamma_1$ and $y_n = y(T)=\gamma_2$, system \eqref{sostituzione_problema_finale} becomes
\begin{equation} \label{eq_semplificata}
    \sum_{i=2}^{n-1} A_i(x_j) y_{i} =h(x_j) - A_1(x_j) \gamma_1 - A_2(x_j) \gamma_2, \quad j=2,\dots,n-1
\end{equation}
which can be written in matrix form $\mathcal{A}\mathbf{y}=\mathbf{b}$ by setting 
\begin{equation} \label{matrix_BVP}
    \mathcal{A}_{ji}= A_{i+1}(x_{j+1}), \quad j,i=1,\dots,n-2,
\end{equation}
\begin{equation*}
\mathbf{y}= \left[
    \begin{array}{c}
         y_2\\
         \vdots\\
         y_{n-1} 
    \end{array}
    \right], \qquad
    \mathbf{b}= \left[
    \begin{array}{c}
         b_1\\
         \vdots\\
         b_{n-2} 
    \end{array}
    \right],
\end{equation*}
where $\mathbf{y}$ is the vector of unknowns and
\begin{equation} \label{eq:termine_noto_BVP}
b_{j}=h(x_{j+1}) - A_1(x_{j+1}) \gamma_1 - A_2(x_{j+1}) \gamma_2, \quad j=1, \dots,n-2,
\end{equation}
the components of the known vector.

Let us denote by 
\begin{equation}\label{ytilde}
\tilde{y}(x)=M_{\mu}[y](x)\approx y(x)
\end{equation} 
the approximation of the solution $y(x)$ of the problem \eqref{eq:BVP} computed through the multinode Shepard collocation method. Therefore, by equation \eqref{M_mu_derivata}, 
$$\tilde{y}(x)=\sum_{i=1}^{n}g_i(x) y(x_{i})$$
where 
$$
g_i(x)=\sum_{j\in \mathcal{K}_i}B_{\mu,j}(x)l_{j,i}(x).
$$
 Then, the polynomial reproduction property of the multinode Shepard operator $\mathcal{M}_{\mu }[\cdot]$ \cite[Proposition 4]{DellAccio:2018} implies the following result. 

\begin{theorem}
\label{Theo1}
    Let us assume that the problem
\eqref{eq:BVP} admits a unique polynomial solution $\phi$ of degree $r\leq d$ and that the matrix $\mathcal{A}$ is non-singular. Then 
    \begin{equation}
    \tilde{y}(x)=\phi(x), \, \forall x\in [0,T].
    \label{exactness}
    \end{equation}
    \label{theo_exactness}
\end{theorem}
\begin{proof}
By the polynomial reproduction property of the operator $\mathcal{M}_{\mu }[\cdot]$ we have $$\mathcal{M}_{\mu}[\phi]\equiv \phi$$
and then, 
$$   \rho  \mathcal{M}^{\prime \prime}_{\mu}[\phi](x)+\lambda (D^{\alpha} \mathcal{M}_{\mu}[\phi])(x)+ \sigma \mathcal{M}_{\mu}[\phi](x)=\rho \phi^{\prime \prime}(x)+\lambda (D^{\alpha}\phi)(x)+ \sigma \phi(x),$$
for all $x\in [0,T]$. Being, by \eqref{eq:BVP},
$$
\left\{
\begin{array}{ll}
\rho \phi^{\prime \prime}(x)+\lambda (D^{\alpha}\phi)(x)+ \sigma \phi(x)=h(x), & x\in (0,T),\\
\phi(0)=\gamma_1,&\\
\phi(T)= \gamma_2 &
\end{array}\right.
$$
the vector $\mathbf{\tilde{y}}=[\phi(x_2),\dots,\phi(x_{n-1})]$ is the unique solution of system \eqref{eq_semplificata} and the thesis follows since
\begin{equation}
\tilde{y}\left( x\right)
=\sum\limits_{i=2}^{n-1}g_{i}\left( x\right)
\phi(x_i)+g_1(x)\gamma_1+g_2(x)\gamma_2=\phi(x).
\end{equation}
\end{proof}

\subsection{The Bagley-Torvik Initial Value Problem}
Let us now consider the Initial Value Problem (IVP)
\begin{equation} \label{eq:IVP}
\left\{
\begin{array}{ll}
     \rho y^{\prime \prime}(x)+\lambda (D^{\alpha}y)(x)+ \sigma y(x)=h(x), & x\in \left(0,T\right] \\
     y(0)=\gamma_1, & \\
     y^{\prime}(0) =\gamma_2, &
\end{array}
\right.
\end{equation}
with $m-1 < \alpha <m$, $m=1,2$, $\gamma_1,\gamma_2 \in \mathbb{R}$.

By taking into account the initial condition $y_1=y(0)=\gamma_1$ in \eqref{eq:IVP} the system \eqref{sostituzione_problema_finale} now becomes 
\begin{equation} \label{eq_semplificata2}
    \sum_{i=2}^{n} A_i(x_j) y_{i} =h(x_j) - A_1(x_j) \gamma_1 \quad j=2,\dots,n,
\end{equation}
and the additional condition $y^{\prime}(0)=\gamma_2$ yields the equation
\begin{equation*}
    \sum_{i=2}^{n} g_i^{\prime}(0) y_i=\gamma_2 - g_1^{\prime} (0) \gamma_1,
\end{equation*}
which has to be included in the collocation system \eqref{eq_semplificata2}.
The linear system \eqref{eq_semplificata2} can be written in matrix form $\mathcal{A}\mathbf{y}=\mathbf{b}$ by setting 
\begin{equation} \label{matrix_IVP}
\begin{split}
    \mathcal{A}_{ji}&= A_{i+1}(x_{j+1}), \quad j,i=1,\dots,n-1, \\
    \mathcal{A}_{ni}&= g_{i+1}^{\prime}(0), \quad i=1,\dots,n-1,
\end{split}  
\end{equation} 
\begin{equation*}
\mathbf{y}= \left[
    \begin{array}{c}
         y_2\\
         \vdots\\
         y_{n} 
    \end{array}
    \right],
    \qquad
    \mathbf{b}= \left[
    \begin{array}{c}
         b_1\\
         \vdots\\
         b_{n-1}\\
         b_{n}
    \end{array}
    \right],
\end{equation*}
where
\begin{equation} \label{eq:termine_noto_IVP}
\begin{split}
b_{j}&=h(x_{j+1}) - A_1(x_{j+1}) \gamma_1 , \quad j=1, \dots,n-1, \\
    b_{n}&=\gamma_2 - g_1^{\prime} (0) \gamma_1.
\end{split}
\end{equation}
We notice that the linear system \eqref{matrix_IVP} corresponds to a least-square problem since the matrix $\mathcal{A}$ has dimension $n \times (n-1)$.

In analogy to Theorem \ref{Theo1}, the polynomial reproduction property of the multinode Shepard operator $\mathcal{M}_{\mu }[\cdot]$ \cite[Proposition 4]{DellAccio:2018} implies the following result. 
\begin{theorem}
\label{Theo2}
    Let us assume that the problem
\eqref{eq:IVP} admits a unique polynomial solution $\phi$ of degree $r\leq d$ and that the matrix $\mathcal{A}$ in \eqref{matrix_IVP} is full rank. Then 
    \begin{equation}
    \tilde{y}(x)=\phi(x), \, \forall x\in [0,T].
    \label{exactness}
    \end{equation}
    \label{theo_exactness}
\end{theorem}

\begin{proof}
By the polynomial reproduction property of the operator $\mathcal{M}_{\mu }[\cdot]$ we have $$\mathcal{M}_{\mu}[\phi]\equiv \phi$$
and then, 
$$   \rho  \mathcal{M}^{\prime \prime}_{\mu}[\phi](x)+\lambda (D^{\alpha} \mathcal{M}_{\mu}[\phi])(x)+ \sigma \mathcal{M}_{\mu}[\phi](x)=\rho \phi^{\prime \prime}(x)+\lambda (D^{\alpha}\phi)(x)+ \sigma \phi(x),$$
for all $x\in [0,T]$. Being, by \eqref{eq:IVP},
$$
\left\{
\begin{array}{ll}
\rho \phi^{\prime \prime}(x)+\lambda (D^{\alpha}\phi)(x)+ \sigma \phi(x)=h(x), & x\in (0,T],\\
\phi(0)=\gamma_1,&\\
\phi^{\prime}(0)=\gamma_2,&
\end{array}\right.
$$
the vector $\mathbf{\tilde{y}}=[\phi(x_2),\dots,\phi(x_n)]$ making null the residual $\mathbf{r}=\mathbf{b}-\mathcal{A}\mathbf{\tilde{y}}$ is the unique solution, in the least square sense, of system \eqref{eq_semplificata2}  and the thesis follows since
\begin{equation}
\tilde{y}\left( x\right)
=\sum\limits_{i=2}^{n}g_{i}\left( x\right)
\phi(x_i)+g_1(x)\gamma_1=\phi(x).
\end{equation}
\end{proof}

\section{Sets of nodes and coverings} \label{sec:nodes}
In this Section, we discuss the sets of nodes and the associated coverings of the multinode Shepard approximant, which we use in the numerical experiments.

\subsection{Mixed Equispaced-Chebyshev set of nodes}
We consider $n_e$ equispaced points $x_1<\dots<x_{n_e}$ in $\left[ 0,1 \right]$ and in each interval 
\begin{equation} \label{sotto_intervallo}
    \left[ x_{i}, x_{i+1} \right] = \left[ \left(i-1\right)h, ih \right], \quad i=1,\dots,n_e-1,\quad h= \frac{1}{n_e-1},
\end{equation}
the $d-1$ Chebyshev points of the first kind 
\begin{equation*}
    x_{{i,k}}= x_i+ \left(-\cos{\left( \frac{ k}{d} \pi \right)} +1 \right) \frac{\left(x_{i+1}-x_i\right)}{2}, \, i=1,\dots,n_e-1, \, k=1,\dots,d-1, 
\end{equation*}
obtained from the Chebyshev points of the first kind by an affine mapping of  $\left[-1,1 \right]$ onto $ \left[ x_{i}, x_{i+1} \right]$.
The set of nodes $X$, of cardinality $n=d(n_e-1)+1$, is defined by
\begin{equation} \label{X_chebyshev}
    X= \bigcup_{i=1}^{n_e-1} \left\{x_i, x_{{i,1}}, \dots x_{{i,d-1}},x_{i+1}\right\}
\end{equation}
and the associated covering is realized by the subsets 
\begin{equation*}
    F_{i}= \{ x_{i},x_{{i,1}}, \dots,x_{{i,d-1}},x_{i+1} \}, \quad i=1,\dots,n_e-1.
\end{equation*}
We call the set $X$ in \eqref{X_chebyshev} \textsl{mixed Equispaced-Chebyshev set of nodes} (mixed E-C).
As an example, in Figure \ref{fig:Chebyshev}, we show the mixed Equispaced-Chebyshev set of nodes for $n_e=6$ and $d=5$, which generates a set $X$ of $n=26$ nodes.
\begin{figure}
    \centering \includegraphics[width=0.33\textwidth]{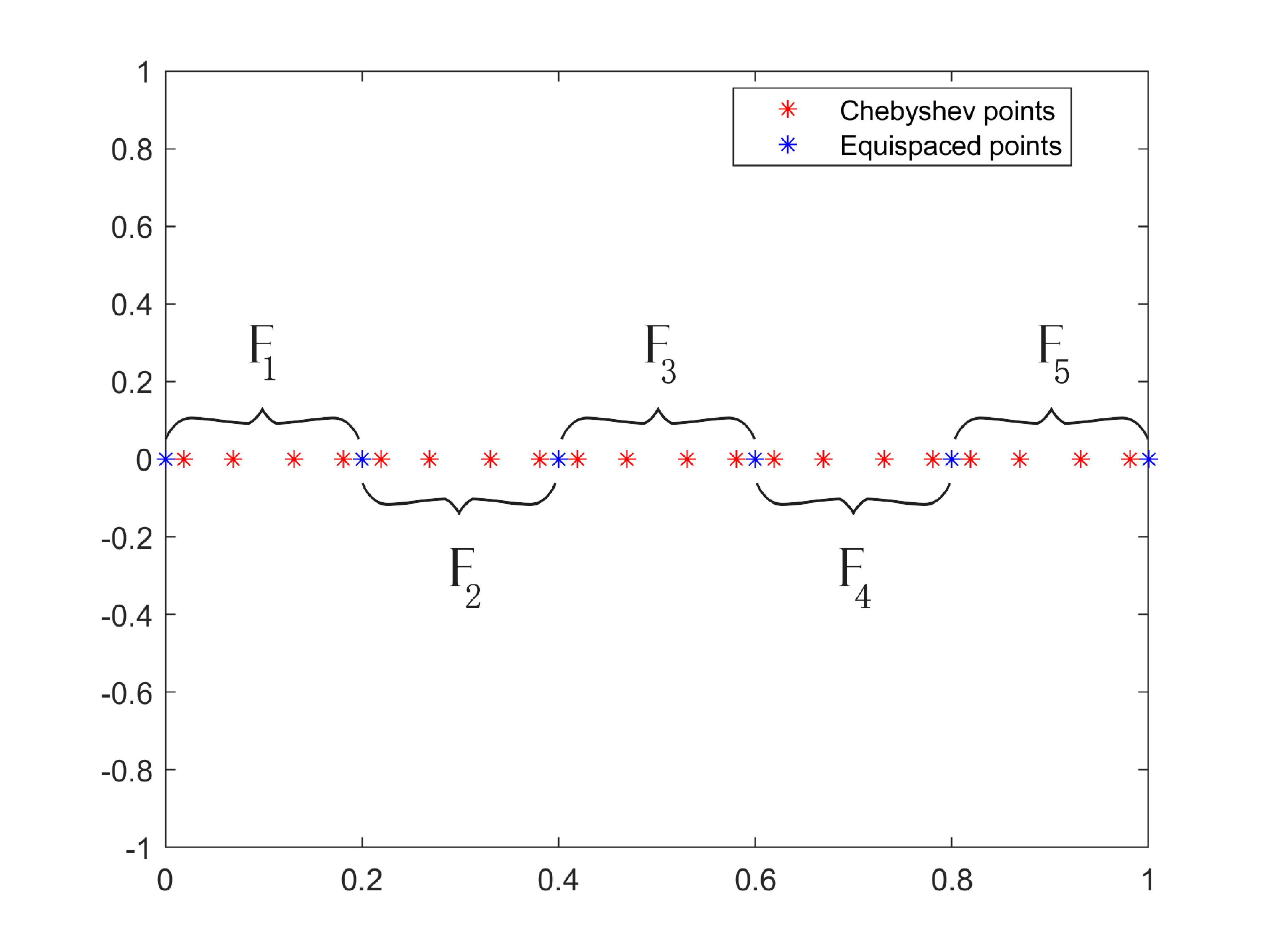}
    \caption{Mixed E-C nodes with $n_e=6$ and $d=5$}
    \label{fig:Chebyshev}
\end{figure}

\subsection{Equispaced set of nodes}
We take $n$ equispaced points in $\left[ 0,1 \right]$
\begin{equation*}
     x_{i }=(i-1)h,  \quad i=1,\dots,n,\quad h= \frac{1}{n-1},
\end{equation*}
and we consider the associated covering constituted by the subsets 
\begin{equation*}
\begin{array}{ll}
    F_{1} = \left\{ x_{1}, \dots, x_{d+1} \right\}, &  \\
     F_{k}= \left\{ x_{\left( k-1 \right)\left( d+1-q \right)}, \dots, x_{ k(d+1)-(k-1)q } \right\}, & k=2,\dots,s-1,\\
      F_{s}= \left\{ x_{n-d}, \dots, x_{ n } \right\}. &\\
\end{array}
\end{equation*}
with $q \in \mathbb{N}_0$, $0\leq q< d$, representing the position of the node in $F_k$ where the superimposition begins, by the assumption that if the last set $F_s$ does not contain $d+1$ points, additional points are taken from $F_{s-1}$, moving backward.
As examples, in Figure \ref{fig:Equispaziati}, we present two cases: $n=40$, $d=7$ and $q=0$, in which the last subset $F_6$ of the covering $\mathcal{F}$ overlaps with $F_5$ at $4$ points and  $n=40$, $d=7$ and $q=1$ in which each subset overlaps the previous one at $2$ points and the last subset $F_7$ overlaps with $F_6$ at $6$ points.

\begin{figure}
    \centering
    \parbox{.33\linewidth}{\centering
    \includegraphics[width=1\linewidth]{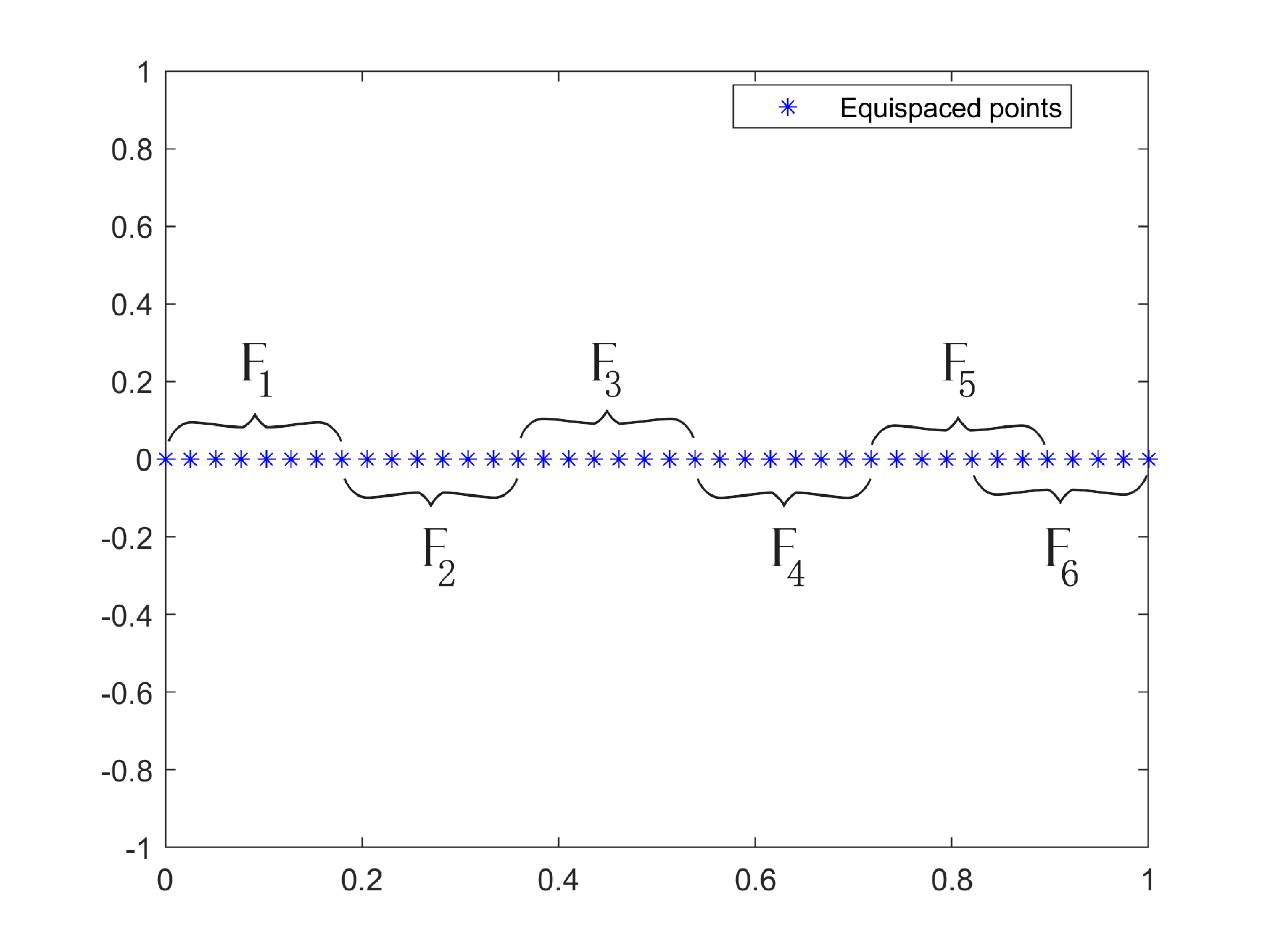}\\[-0.5ex] {$q=0$}
     }
     \parbox{.33\linewidth}{\centering
    \includegraphics[width=1\linewidth]{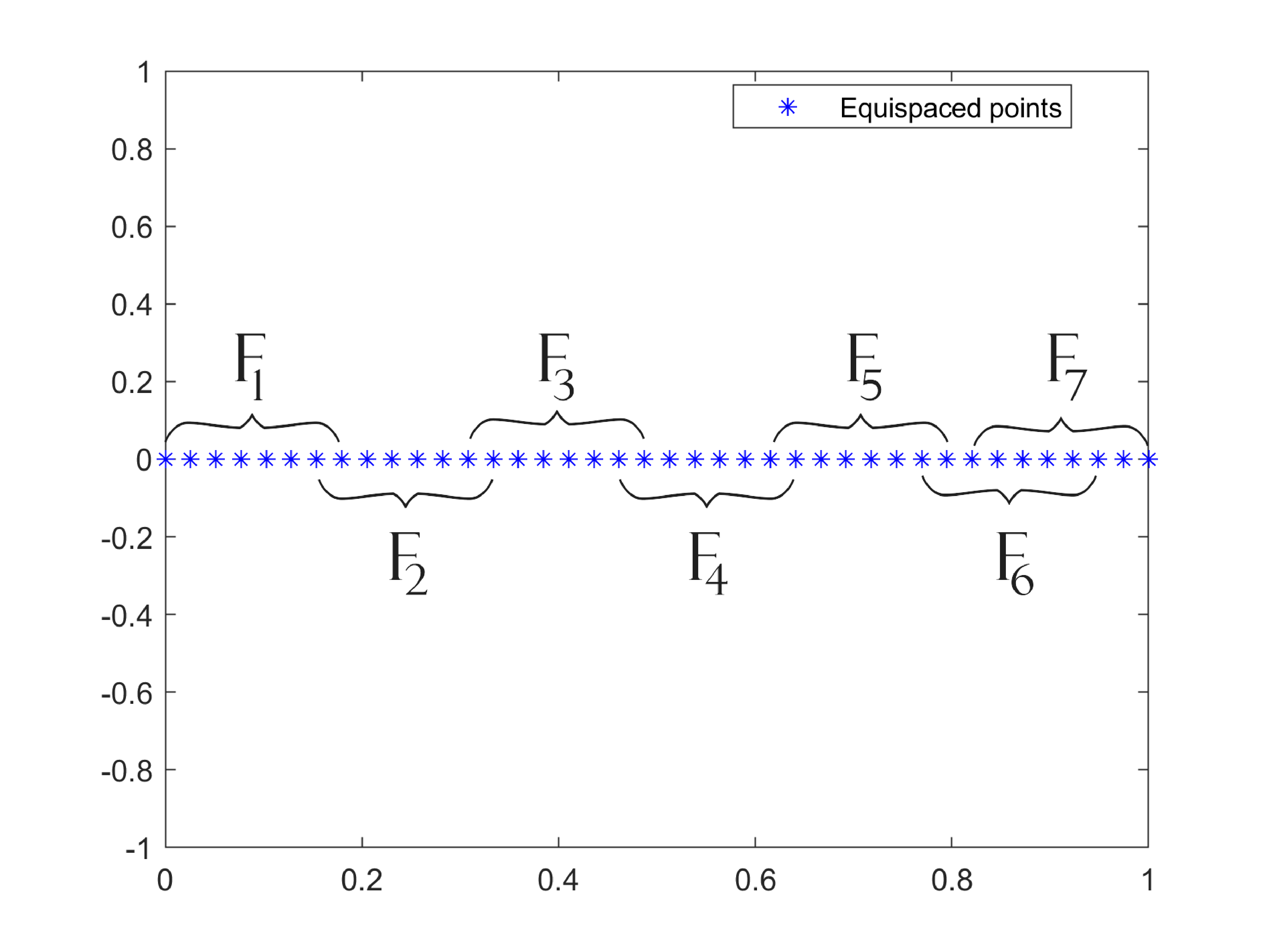}\\[-0.5ex] {$q=1$}
     }
    \caption{Equispaced nodes with $n_e=40$ and $d=7$}
    \label{fig:Equispaziati}
\end{figure}

\subsection{Mixed Equispaced-mock-Chebyshev set of nodes}
We consider \( n_e \) equispaced points in \([0,1]\) and for each subinterval \([x_i, x_{i+1}]\), with \( i=1,\dots,n_e-1 \), we construct a set of \( n_s=3(d+1) \) equispaced points, from which we extract a subset of \( d+1 \) mock-Chebyshev points. 
These points are chosen from the \( n_s +2 \) equispaced points in \([x_i, x_{i+1}]\) in such a way that their positions best mimic those of the Chebyshev-Lobatto nodes of the corresponding degree \( d \). The ratio behind this choice is that mock-Chebyshev nodes form a modified set of interpolation points inspired by Chebyshev-Lobatto nodes, designed to mitigate the Runge phenomenon and enhance the accuracy of polynomial interpolation \cite{DEMARCHI2015,DELLACCIO2022}. 
As for the mixed equispaced-Chebyshev nodes, the set $X$ has cardinality $n=d(n_e-1)+1$.
We call this set of nodes \textsl{mixed Equispaced-Mock-Chebyshev set of nodes} (mixed E-MC). This configuration of nodes is particularly useful when the total number of interpolation nodes $n$ is very large.
In Figure \ref{fig:MOCK}, as an example, we display the case $n_e=4$, $n_s=10$ and $d=6$ by showing the starting set and final set with the specification of the covering.

\begin{figure}
    \centering
    \parbox{.33\linewidth}{\centering
    \includegraphics[width=1\linewidth]{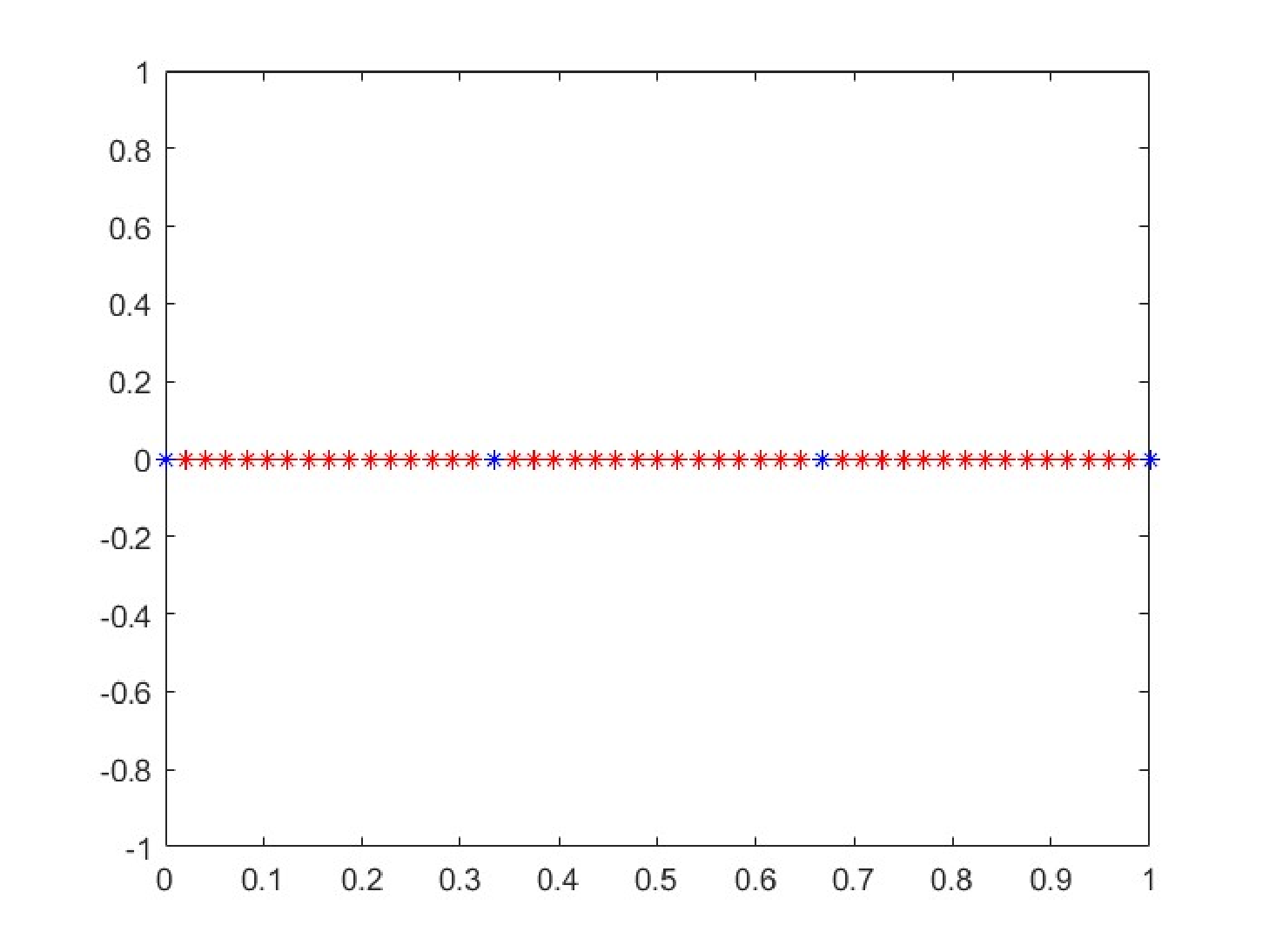}\\[-0.5ex]
     }\hfill
     \parbox{.33\linewidth}{\centering
    \includegraphics[width=1\linewidth]{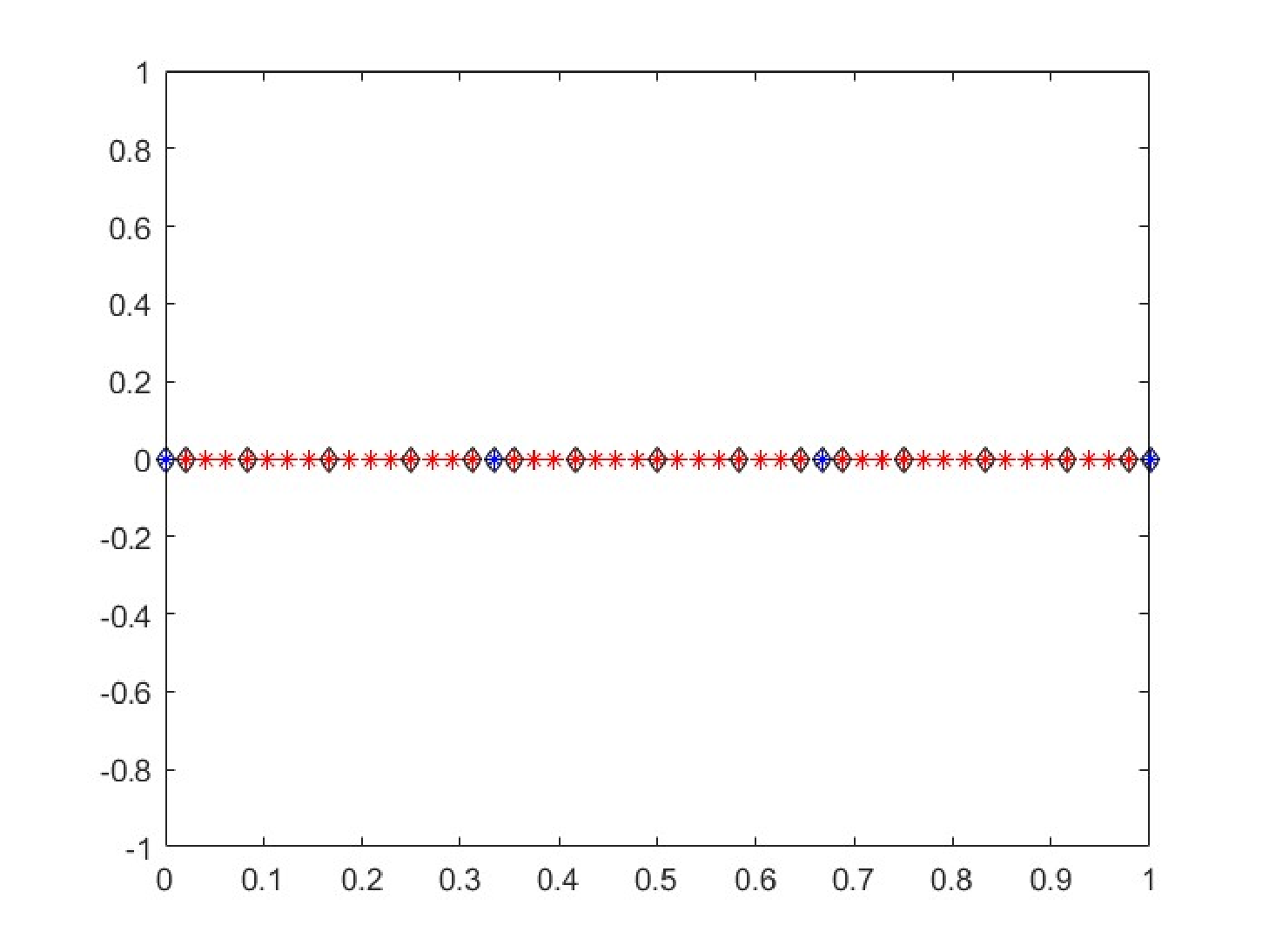}\\[-0.5ex]
     }\hfill
    \parbox{.33\linewidth}{\centering
    \includegraphics[width=1\linewidth]{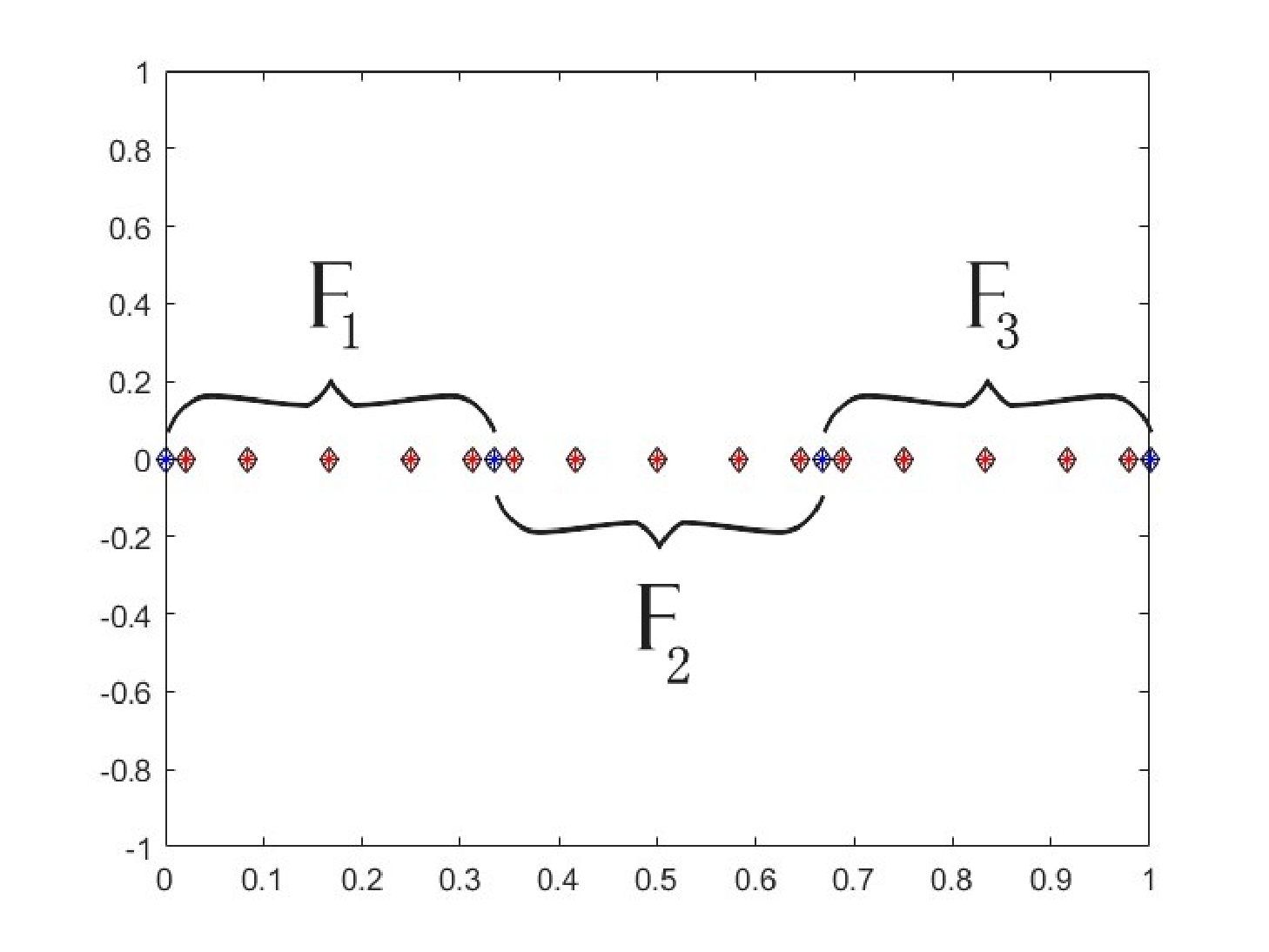}\\[-0.5ex]
    }
    \caption{Mixed E-MC nodeS with $n_e=4$, $n_s=10$ and $d=6$}
    \label{fig:MOCK}
\end{figure}

\section{Numerical experiments}
\label{sec:numerical_experiments}
In this Section, we first test the performance and show the accuracy of the proposed method in approximating the fractional derivative of a given function. We use the sets of nodes discussed in Section \ref{sec:nodes} and different degrees for the local polynomial interpolants.
Subsequently, we show the efficiency of the method in solving various Bagley-Torvik problems.

\subsection{Numerical approximation of $D^{\alpha}f$}
In the following examples \cite{esempi_numerici,Application_BTE}, we compute the pointwise absolute errors 
\begin{equation} \label{pointwise_error}
    E^{\alpha}(\xi_k)=\left\vert y(\xi_k)-(D^{\alpha}\mathcal{M}_{4}\left[y \right])(\xi_k) \right\vert, \quad k=1,\dots,N,
\end{equation}
and the mean absolute error
\begin{equation} \label{mean_error}
    E^{\alpha}_{mean}=  \frac{1}{N} \displaystyle \sum_{k=1}^{N} E^{\alpha}(\xi_k) ,
\end{equation}
at $N=100$ equispaced points $\xi_k$ in  $\left[0,1\right]$. The order $\alpha$ of the fractional derivative varies in the set $\left\{1/5, 1/2, 4/5, 6/5, 3/2, 9/5 \right\}$, the number of equispaced nodes is always $n=(n_e-1)d+1$ and the covering $\mathcal{F}$ is realized by fixing $q=0$.  
For each example we show the semilog plot of the pointwise error \eqref{pointwise_error}  when using local polynomial interpolants of degree $d=8$, and the semilog plot of the mean absolute error \eqref{mean_error} by varying $d$ from $2$ to $10$.

\begin{example} \label{example_1}
Let $f(x)=sin(x)$. 
It is known that
\begin{equation*}
    \left( D^{\alpha} f\right) (x)= \frac{i^m x^{m-\alpha}}{2i}\left[ E_{1, m-\alpha+1}\left(i  x \right)- (-1)^m E_{1, m-\alpha+1}\left(- i x \right) \right],
\end{equation*}
where $m\in \mathbb{N}$, $m-1<\alpha<m$ and $E_{\gamma,\delta}(x)$ is the generalized Mittag-Leffler function \cite{oldham2006fractional}, defined as 
\begin{equation} \label{Mittag_leffle_function}
    E_{\gamma,\delta}(x)= \sum_{k=0}^{\infty} \frac{x^k}{\Gamma(\gamma k +\delta)}, \quad x \in \mathbb{C}, \quad  \gamma, \delta \in \mathbb{C}, \quad R(\gamma),  R(\delta)>0.
\end{equation}
We fix $n_e=9$. The numerical results are reported in Figure \ref{fig:pointwise_error_example1} and Figure \ref{fig:mean_error_example1}.
\end{example}

\begin{example}\label{example_2}
Let $f(x)=x^{\frac{9}{2}}$.
The analytical form of the Caputo fractional derivative for $f(x)$ is given by
\begin{equation*}
    (D^{\alpha} f)(x) = \frac{\Gamma(\frac{9}{2}+1)}{\Gamma(\frac{9}{2}+1-\alpha)}x^{\frac{9}{2}-\alpha}.
\end{equation*}
We fix $n_e=20$.
The numerical results are reported in Figure \ref{fig:pointwise_error_example2} and Figure \ref{fig:mean_error_example2}.
\end{example}

\begin{example} \label{example_3}
    Let $f(x)=e^{2 x }$.
The analytical form of the Caputo fractional derivative for $f(x)$ is given by \cite{some_fractional_derivatives}
\begin{equation*}
    \left( D^{\alpha} f \right) (x) =\displaystyle \sum_{k=0}^{\infty} \frac{2^{k+m}}{{\Gamma(k+1+m-\alpha)}} x^{k+m-\alpha}= 2^{m} x^{m-\alpha} E_{1, m-\alpha+1}\left(2 x \right),
\end{equation*}
where $m-1<\alpha<m$ and $E_{\gamma,\delta}(x)$ is the generalized Mittag-Leffler function \eqref{Mittag_leffle_function}.
We fix $n_e=10$.
The numerical results are reported in Figure \ref{fig:pointwise_error_example3} and Figure \ref{fig:mean_error_example3}

\end{example}

For all three examples, the semilog plots of the pointwise absolute error show excellent performance of the method, with errors typically ranging between 
$10^{-20}$ and $10^{-6}$. The results confirm that the method is robust and accurate across a broad range of fractional orders. Among the three node configurations, the mixed E-C and E-MC sets tend to offer slightly better or comparable performance to the equispaced nodes, particularly in handling endpoint behavior.

The semilog plots of the mean absolute error as a function of the polynomial degree show a clear trend: increasing the degree of the local interpolants improves accuracy, with errors decreasing by several orders of magnitude. This trend is consistent across all three test functions. The mixed node sets, especially the mixed E-C configuration, often achieve lower mean errors at moderate degrees.

\begin{figure}
     \centering
    \parbox{.33\linewidth}{\centering
    \includegraphics[width=1\linewidth]{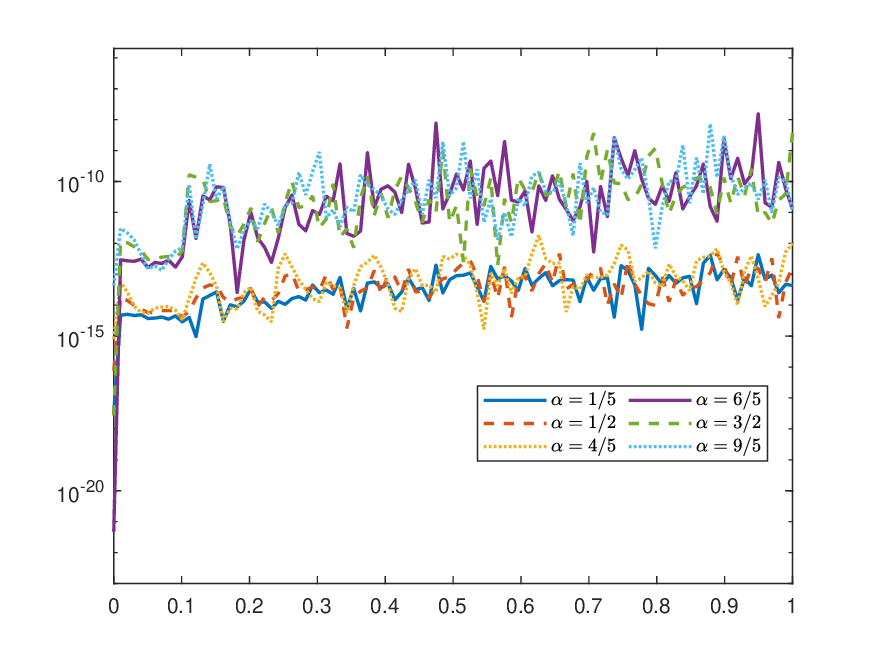}\\[-0.5ex] \scriptsize{Equispaced}}\hfill
    \parbox{.33\linewidth}{\centering
    \includegraphics[width=1\linewidth]{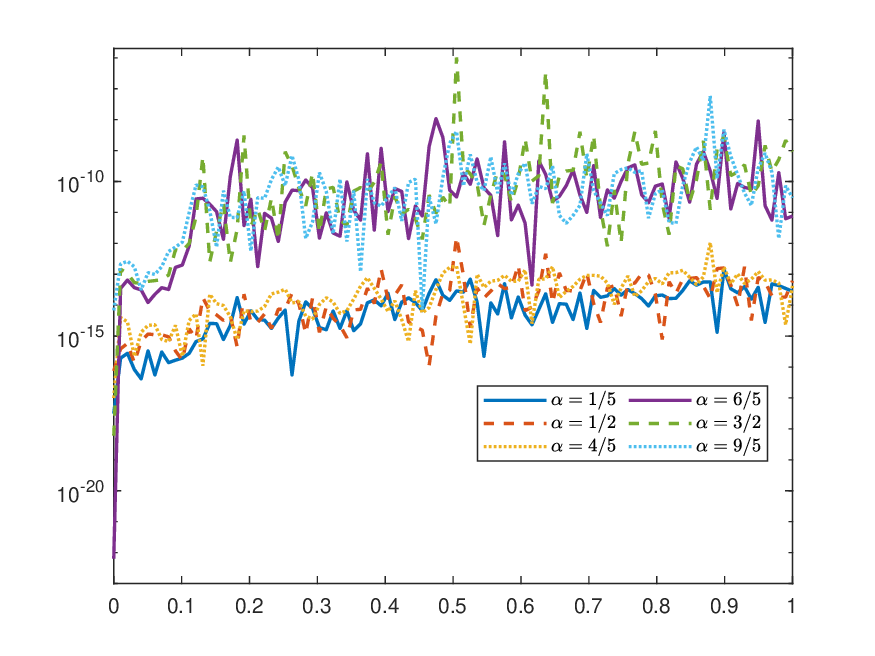}\\[-0.5ex]
     \scriptsize{Mixed E-C} }\hfill
     \parbox{.33\linewidth}{\centering
    \includegraphics[width=1\linewidth]{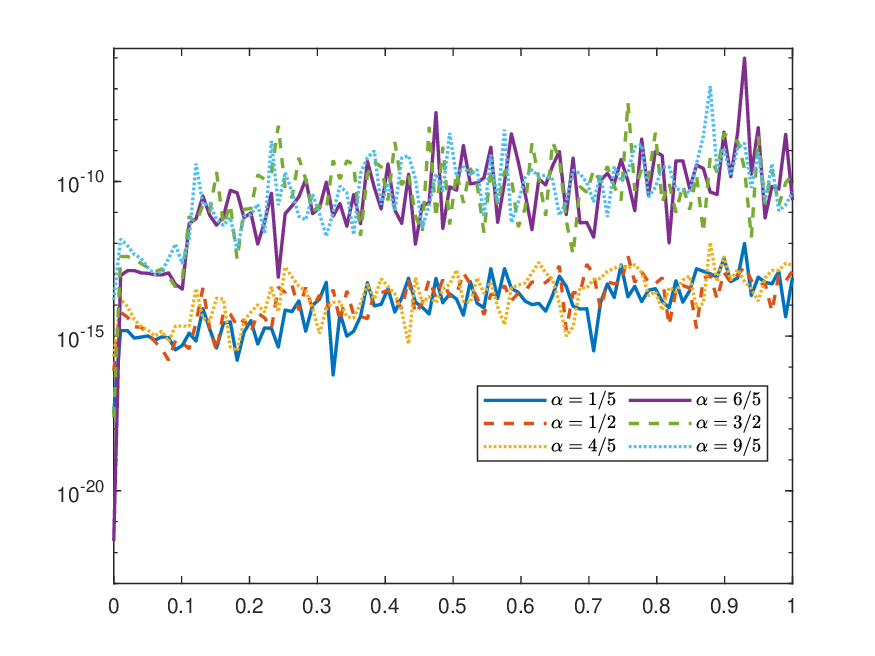}\\[-0.5ex]
    \scriptsize{Mixed E-MC} }
    \caption{Semilog plots of the pointwise absolute error for the Example \ref{example_1} for the various values of $\alpha$ with $d=8$.}
    \label{fig:pointwise_error_example1}
\end{figure}

\begin{figure}
    \centering
    \parbox{.33\linewidth}{\centering
    \includegraphics[width=1\linewidth]{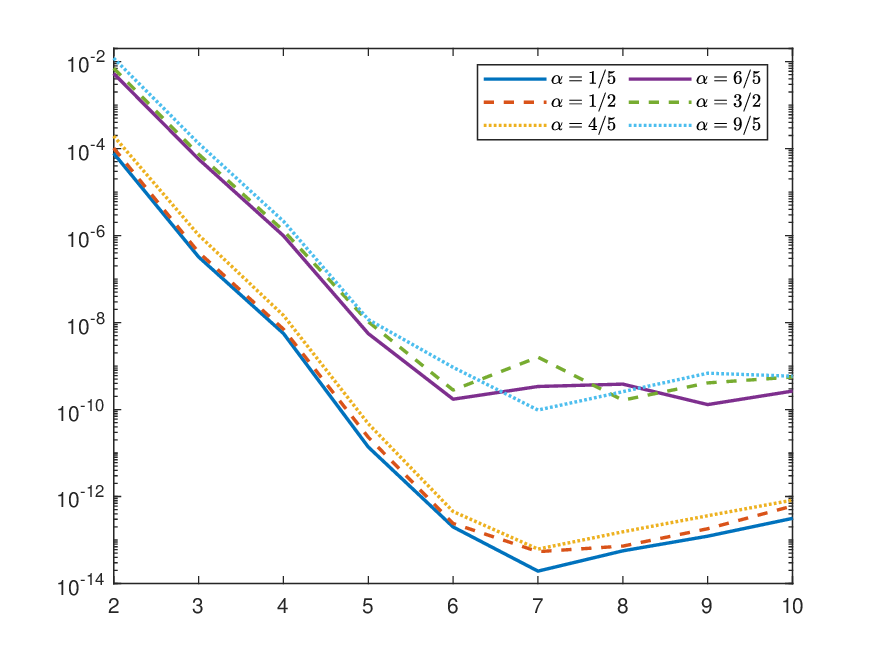}\\[-0.5ex] \scriptsize{Equispaced}}\hfill
    \parbox{.33\linewidth}{\centering
    \includegraphics[width=1\linewidth]{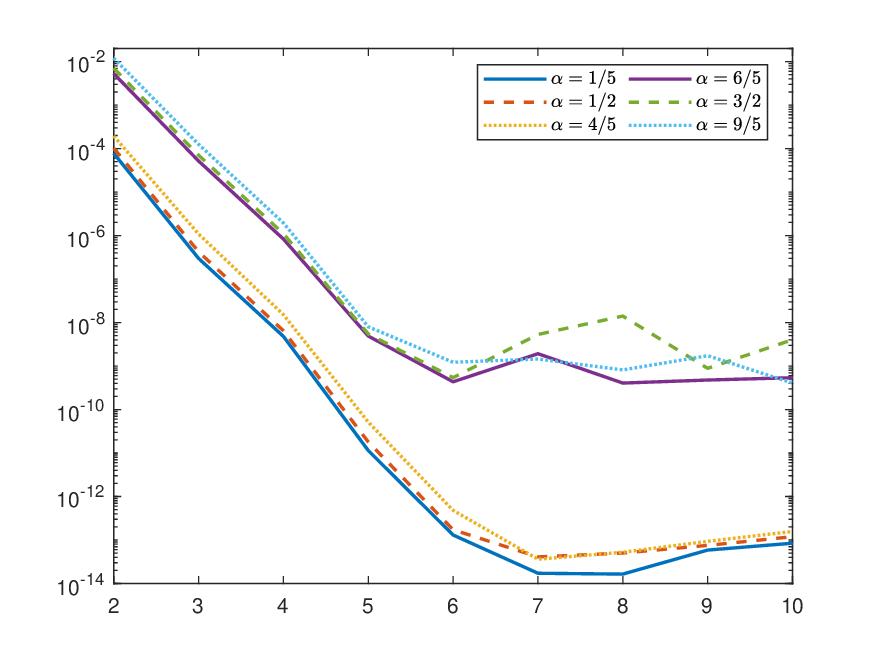}\\[-0.5ex]
     \scriptsize{Mixed E-C} }\hfill
     \parbox{.33\linewidth}{\centering
    \includegraphics[width=1\linewidth]{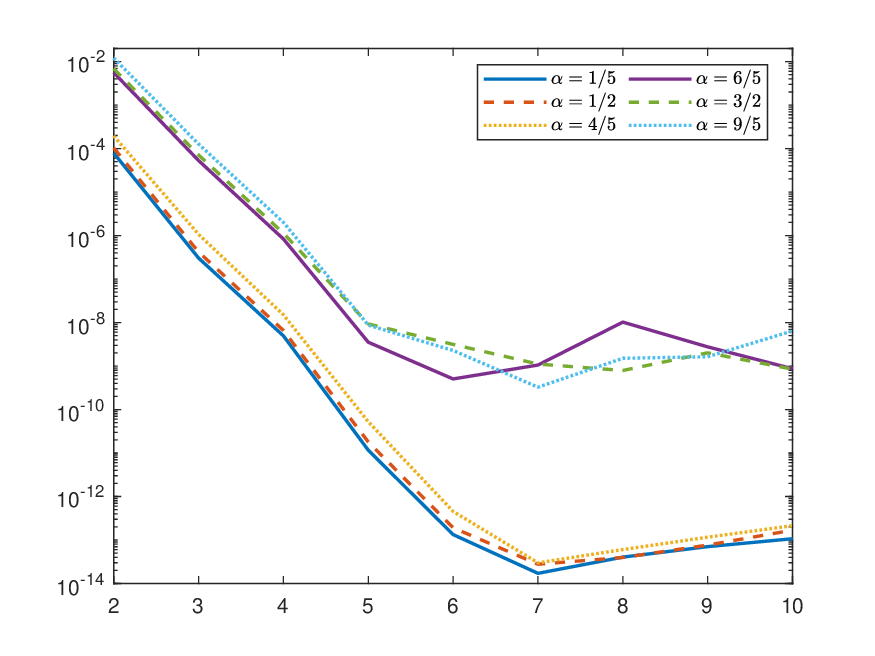}\\[-0.5ex]
     \scriptsize{Mixed E-MC} }    
    \caption{Semilog plots of the mean absolute error for the Example \ref{example_1} for the various values of $\alpha$ with $d$ varying from $2$ to $10$.}
    \label{fig:mean_error_example1}
\end{figure}

\begin{figure}
    \centering
    \parbox{.33\linewidth}{\centering
    \includegraphics[width=1\linewidth]{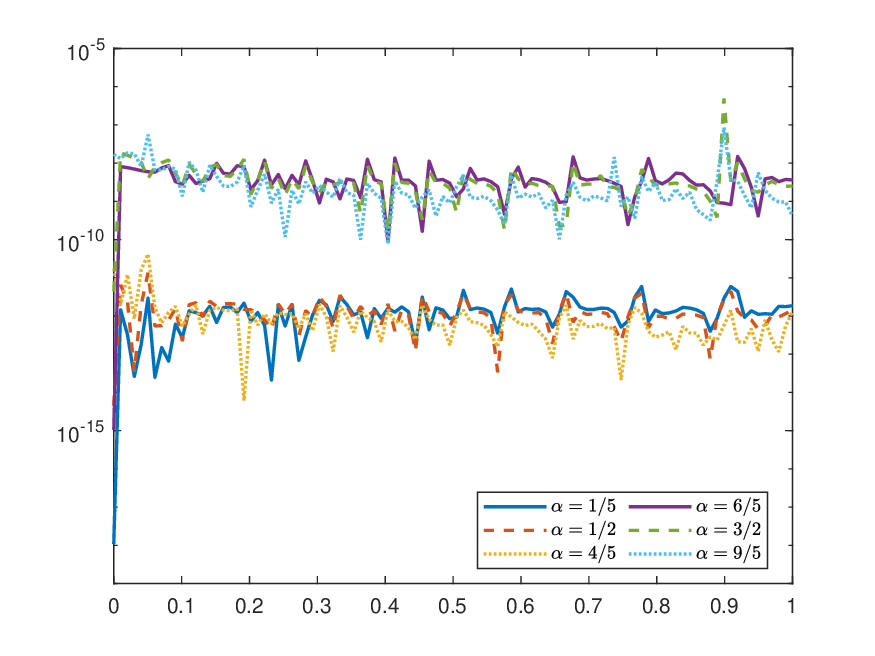}\\[-0.5ex] \scriptsize{Equispaced}}\hfill
    \parbox{.33\linewidth}{\centering
    \includegraphics[width=1\linewidth]{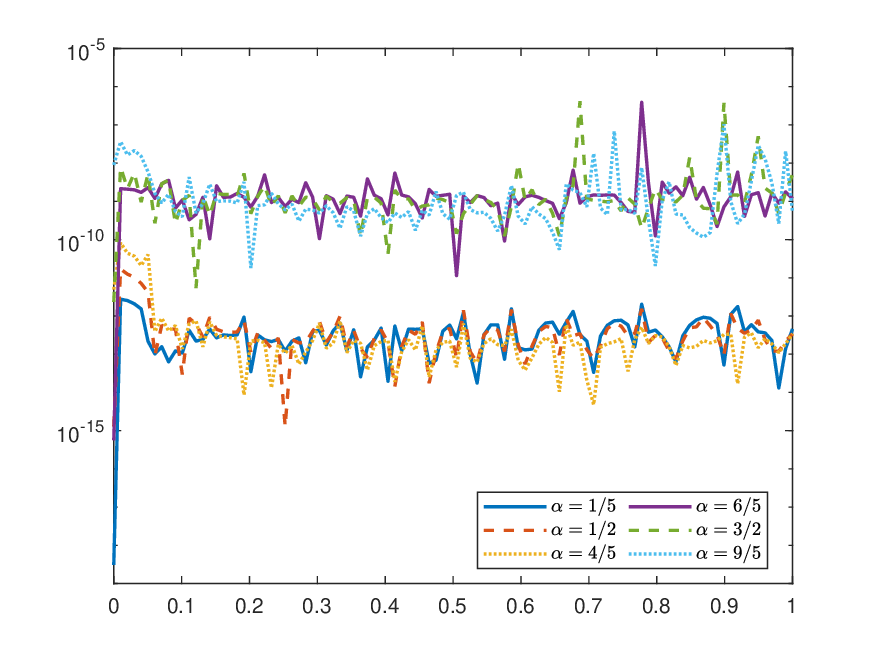}\\[-0.5ex]
     \scriptsize{Mixed E-C} }\hfill
     \parbox{.33\linewidth}{\centering
    \includegraphics[width=1\linewidth]{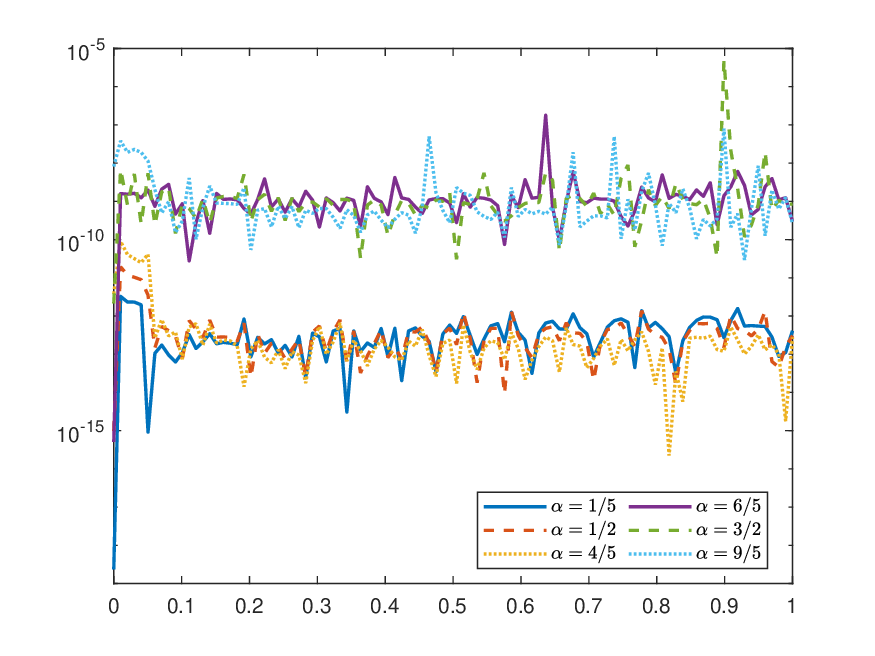}\\[-0.5ex]
    \scriptsize{Mixed E-MC}}
    \caption{Semilog plots of the pointwise absolute error for the Example \ref{example_2} for the various values of $\alpha$ with $d=8$.}
    \label{fig:pointwise_error_example2}
\end{figure}

\begin{figure}
    \centering
    \parbox{.33\linewidth}{\centering
    \includegraphics[width=1\linewidth]{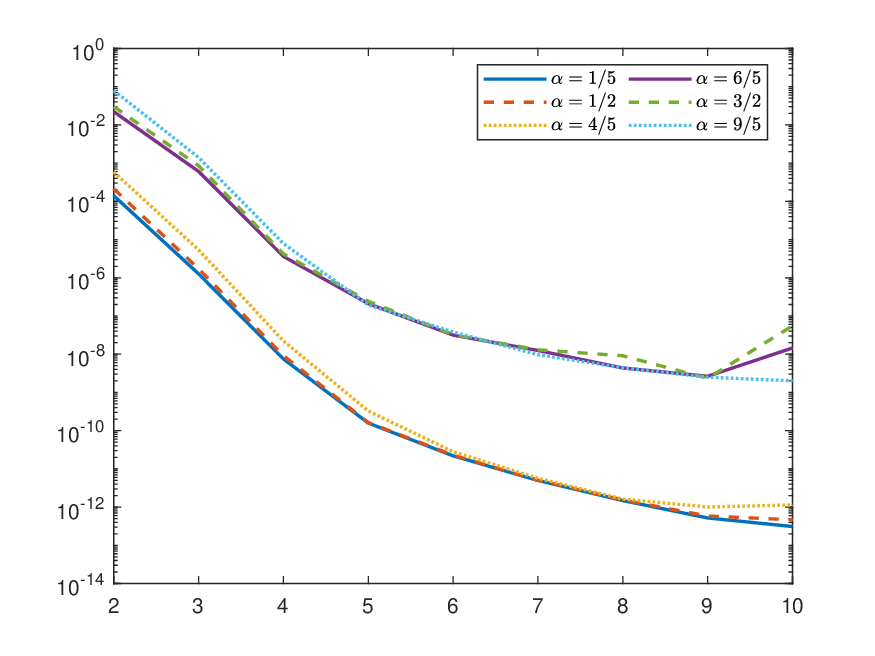}\\[-0.5ex] \scriptsize{Equispaced}}\hfill
    \parbox{.33\linewidth}{\centering
    \includegraphics[width=1\linewidth]{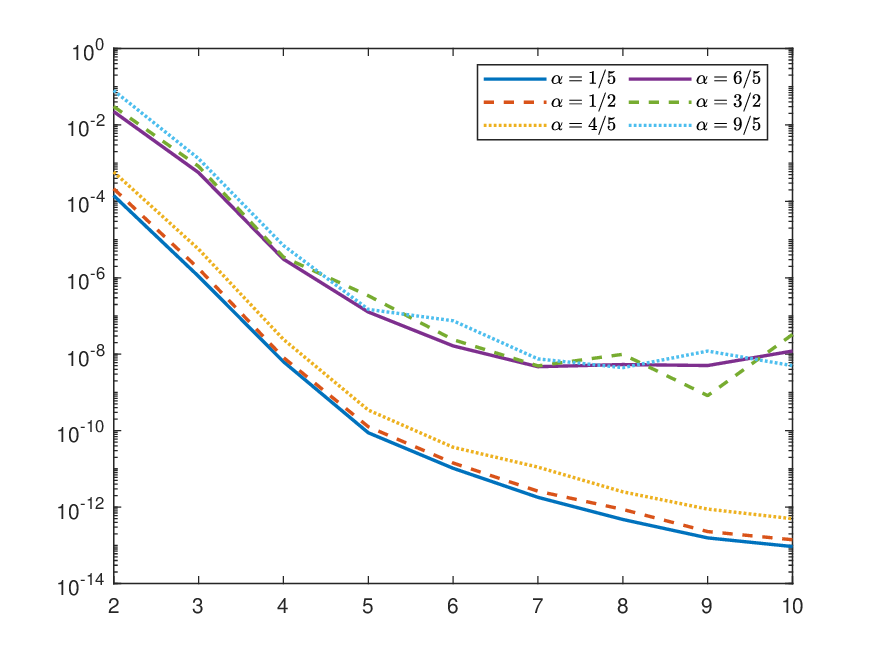}\\[-0.5ex]
     \scriptsize{Mixed E-C} }\hfill
     \parbox{.33\linewidth}{\centering
    \includegraphics[width=1\linewidth]{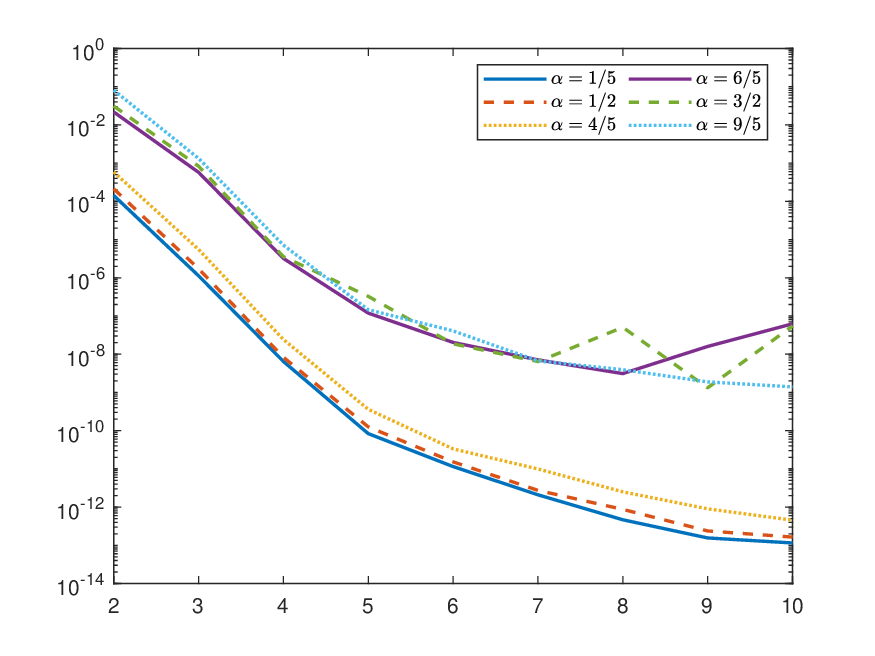}\\[-0.5ex]
     \scriptsize{Mixed E-MC} }  
    \caption{Semilog plots of the mean absolute error for the Example \ref{example_2} for the various values of $\alpha$ with $d$ varying from $2$ to $11$.}
    \label{fig:mean_error_example2}
\end{figure}

\begin{figure}
    \centering
    \parbox{.33\linewidth}{\centering
    \includegraphics[width=1\linewidth]{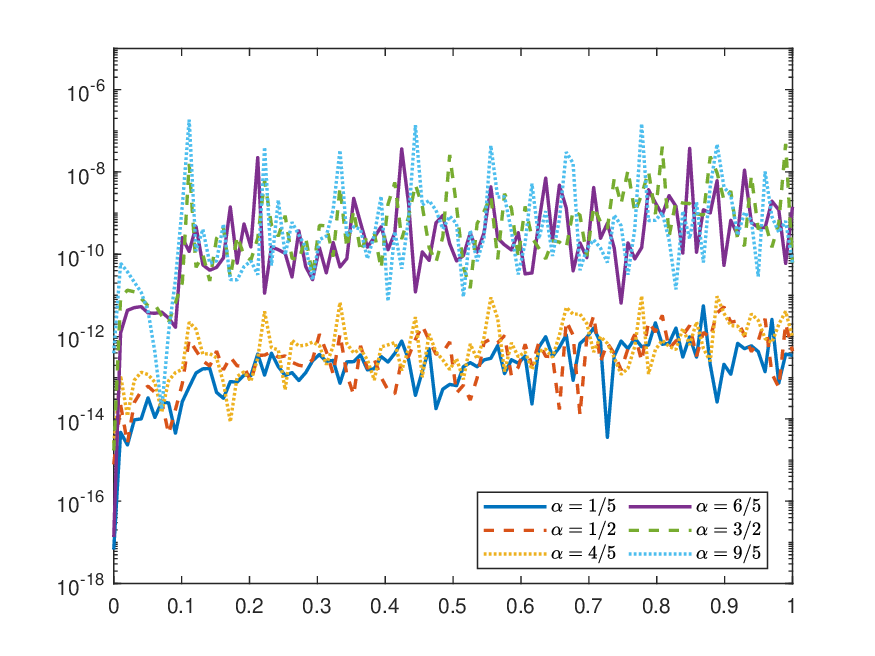}\\[-0.5ex] \scriptsize{Equispaced}}\hfill
    \parbox{.33\linewidth}{\centering
    \includegraphics[width=1\linewidth]{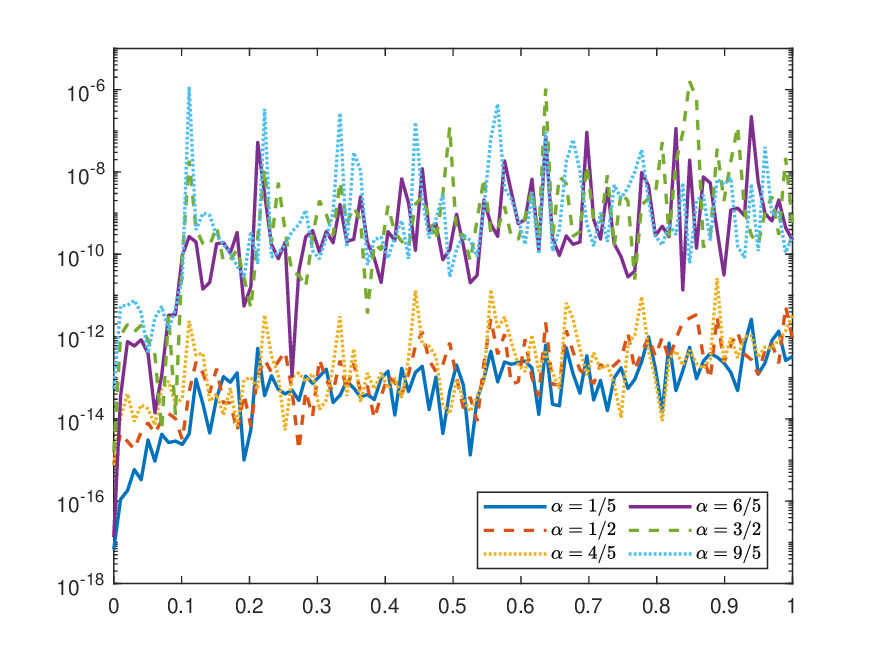}\\[-0.5ex]
     \scriptsize{Mixed E-C} }\hfill
     \parbox{.33\linewidth}{\centering
    \includegraphics[width=1\linewidth]{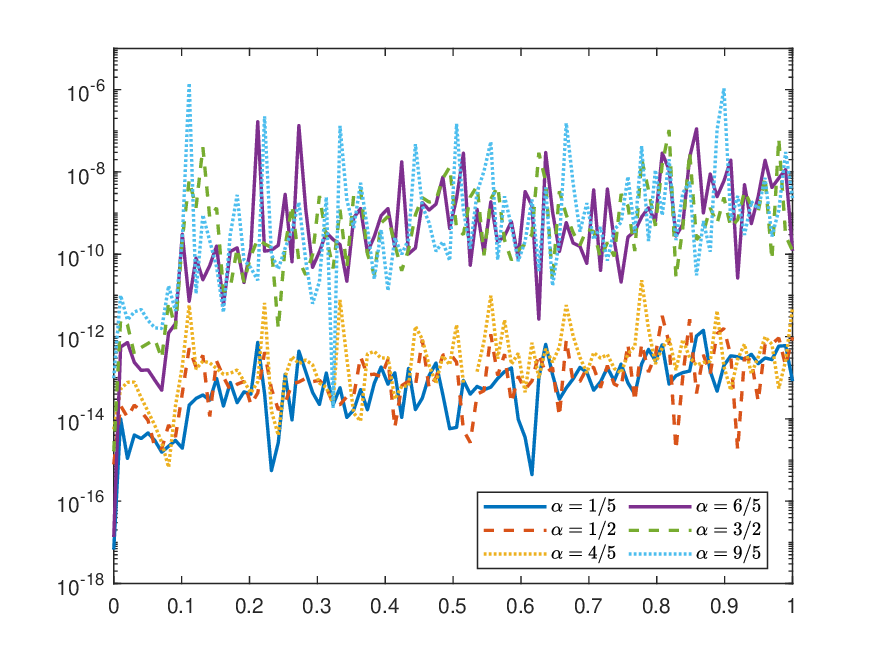}\\[-0.5ex]
    \scriptsize{Mixed E-MC}}
    \caption{Semilog plots of the pointwise absolute error for the Example \ref{example_3}.}
    \label{fig:pointwise_error_example3}
\end{figure}

\begin{figure}
    \centering
    \parbox{.33\linewidth}{\centering
    \includegraphics[width=1\linewidth]{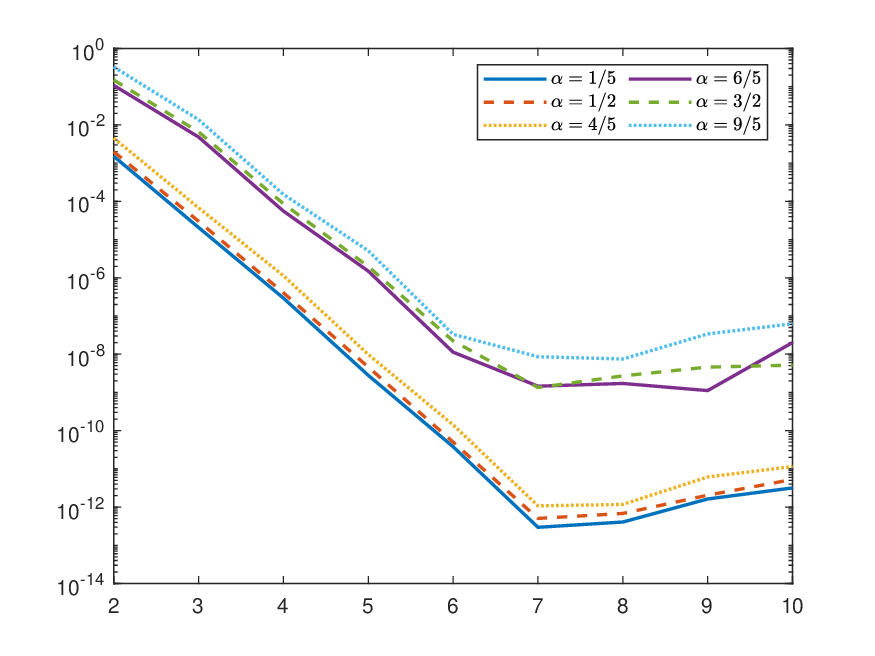}\\[-0.5ex] \scriptsize{Equispaced}}\hfill
    \parbox{.33\linewidth}{\centering
    \includegraphics[width=1\linewidth]{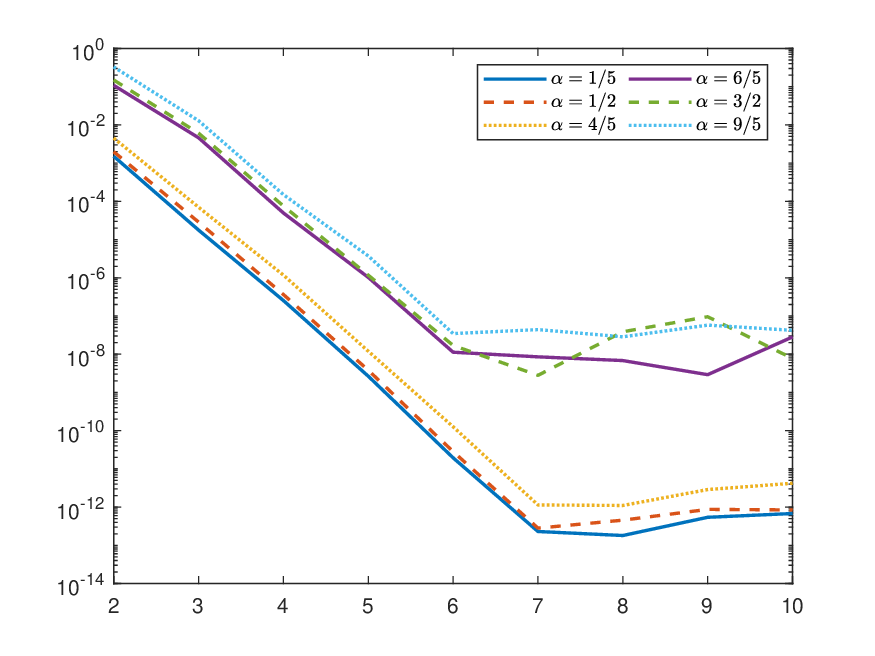}\\[-0.5ex]
     \scriptsize{Mixed E-C} }\hfill
     \parbox{.33\linewidth}{\centering
    \includegraphics[width=1\linewidth]{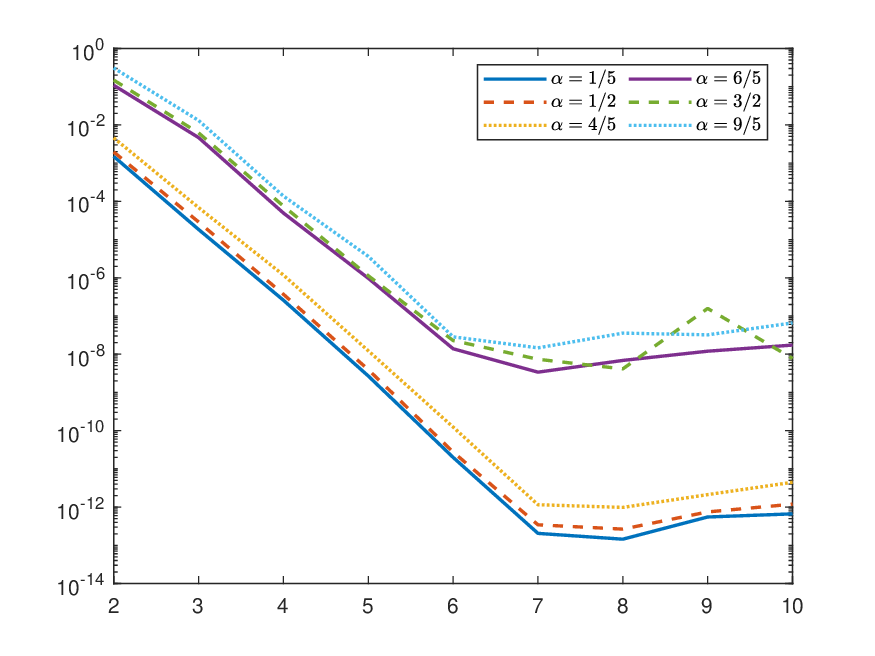}\\[-0.5ex]
     \scriptsize{Mixed E-MC} }  
    \caption{Semilog plots of the mean absolute error for the Example \ref{example_3}.}
    \label{fig:mean_error_example3}
\end{figure}

\subsection{Applications to the Bagley-Torvik BVPs} \label{section_BVP}
We test the performance of the proposed method in numerically solving Bagley-Torvik BVPs on well-known examples \cite{Fractional_derivative_BTE,BTE_Bessel}. 
The degree $d$ of the local polynomial interpolant and the cardinality of the node sets are specified in each example.
In particular, for the equispaced set of nodes of Section \ref{sec:nodes}, the values of $q=q(d,\alpha)$, for each fixed $d$ and $\alpha$, are computed by a trial and error procedure that allows us to obtain near the optimal approximation. These values are not reported for space reasons.

The approximations are computed at $N=100$ equispaced points $\xi_k$ in  $\left[0,1\right]$. In the examples, we display the semilog plot of the pointwise absolute error
\begin{equation}
    E(\xi_k)=\left\vert y(\xi_k)-\mathcal{M}_{4}[y](\xi_k) \right\vert, \quad k=1,\dots,100,
\end{equation}
or the semilog plot of the mean absolute error 
\begin{equation} \label{mean_error}
    E_{mean}=  \frac{1}{N} \displaystyle \sum_{k=1}^{N} E(\xi_k) ,
\end{equation}
by varying the degree $d$ of the local polynomial interpolant.
Moreover, for each example, we compute the condition number of the collocation matrix $\mathcal{A}$ and we refer to it as $\cond(\mathcal{A})$.

\begin{example}
\label{example_BVP_1}
   Let us consider the boundary value problem 
    \begin{equation*}
\left\{
\begin{array}{ll}
     y^{\prime \prime}(x)+(D^{\frac{3}{2}}y)(x)+ y(x)=h(x), & x\in \left(0,1\right), \\
     y(0)=1, & \\
     y(1) =2, &
\end{array}
\right.
\end{equation*}
with $h(x)=1+x$, whose exact solution is $y(x)=1+x$.

We fix $d=3$.
For the mixed E-C and mixed E-MC set of nodes, we set $n_e=3$.
For the equispaced nodes, we set $n=8$ and we realize the covering $\mathcal{F}$ by setting $q=2$.
In Figure \ref{BVP_1.1}, we display the semilog plot of the pointwise absolute error. Moreover, we have a mean absolute error of $9.85e-15$ by using the equispaced set of nodes, $4.19e-14$ by using the  mixed E-C set of nodes and $5.69e-14$ by using the  mixed E-MC set of nodes. The condition number of the collocation matrix $\mathcal{A}$ in \eqref{matrix_BVP} is $3.72e+1$, $1.28e+2$ and $1.17e+2$, respectively.

\end{example}

\begin{figure}
    \centering
    \parbox{.33\linewidth}{\centering
    \includegraphics[width=\linewidth]{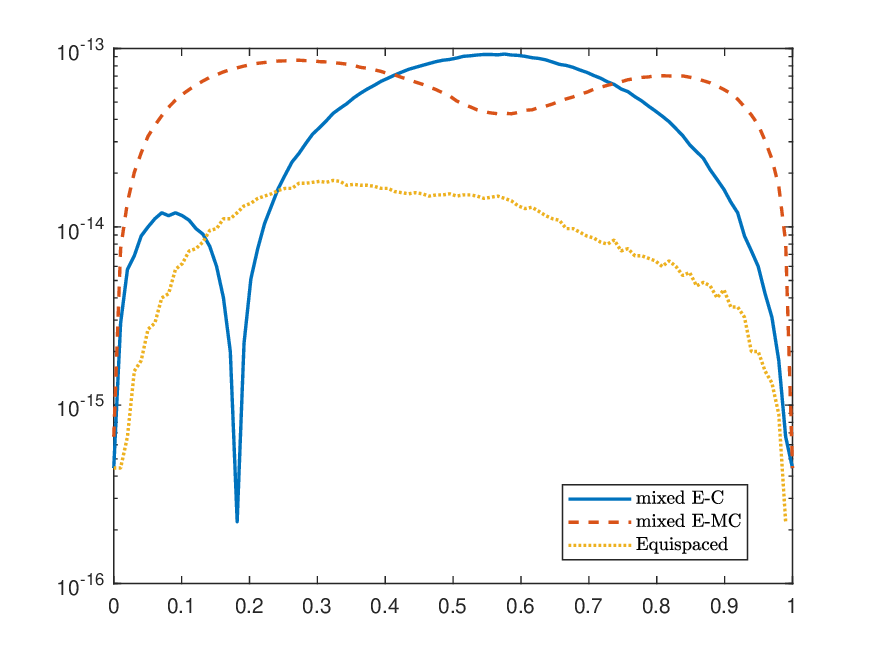}\\[-0.5ex]
    }
    \caption{Semilog plots of the pointwise errors for the Example \ref{example_BVP_1}.}
    \label{BVP_1.1}
\end{figure}

\begin{example}\label{example_BVP_2}
    Let us consider the boundary value problem
    \begin{equation*}
\left\{
\begin{array}{ll}
     y^{\prime \prime}(x)+\frac{8}{17}(D^{\frac{3}{2}}y)(x)+ \frac{13}{51}y(x)=h(x), & x\in \left(0,1\right), \\
     y(0)=0, & \\
     y(1) =0, &
\end{array}
\right.
\end{equation*}
where $h(x)=\frac{1}{89250 \sqrt{\pi}}\left(96 \sqrt{x} \, p(x)+7 \sqrt{\pi} \, q(x)\right)$, with $p(x)=-2373+10640x-16240x^2+8000x^3$ and $q(x)=-34578+233262x-448107x^2+264880x^3-9425x^4+3250x^5$.
The exact solution is $y(x)=\frac{27}{125}x-\frac{339}{250}x^2+\frac{76}{25}x^3-\frac{29}{10}x^4+x^5$.

We fix $d=6$.
For the mixed E-C and mixed E-MC set of nodes, we set $n_e=3$.
For the equispaced nodes, we set $n=13$ and we realize the covering $\mathcal{F}$ by setting $q=7$.
In Figure \ref{BVP_2.1}, we display the semilog plot of the pointwise absolute error. Moreover, we have a mean absolute error of $4.72e-17$ by using the equispaced set of nodes, $5.82e-17$ by using the  mixed E-C set of nodes and $3.93e-16$ by using the  mixed E-MC set of nodes. The condition number of the collocation matrix $\mathcal{A}$ in \eqref{matrix_BVP} is $4.22e+3$, $1.98e+3$ and $1.71e+3$, respectively.
\end{example}

\begin{figure}
	\centering
	\parbox{.33\linewidth}{\centering
		\includegraphics[width=\linewidth]{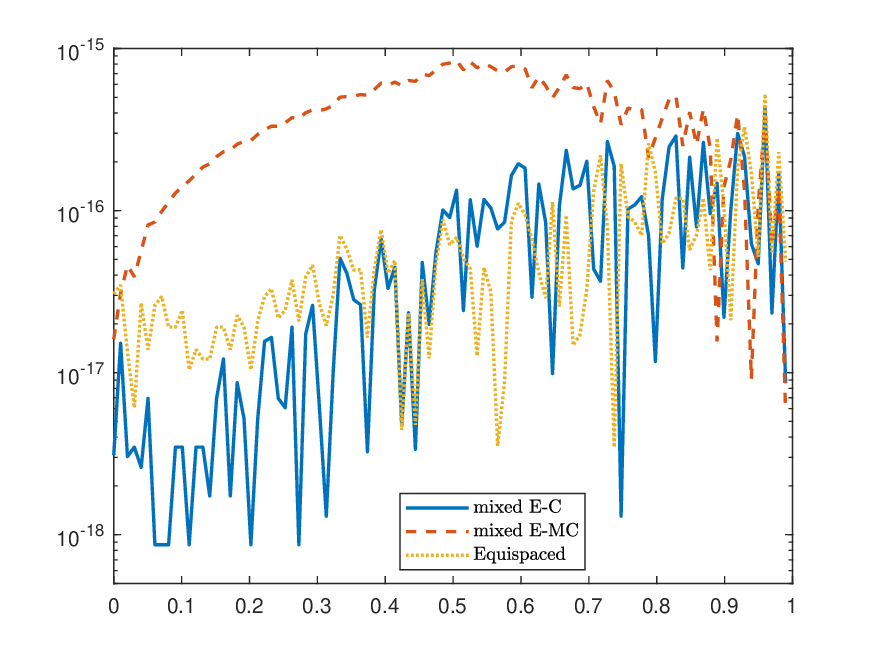}\\[-0.5ex]
		} \hfill
	
	\caption{Semilog plots of the pointwise error for the Example \ref{example_BVP_2}.}
	\label{BVP_2.1}
\end{figure}

\begin{example} \label{example_BVP_3}
   Let us consider the boundary value problem
    \begin{equation*}
\left\{
\begin{array}{ll}
     (D^{\frac{3}{2}}y)(x)+ y(x)=h(x), & x\in \left(0,1\right), \\
     y(0)=0, & \\
     y(1)=0, &
\end{array}
\right.
\end{equation*}
with $h(x)=2\frac{ \displaystyle x^{\frac{1}{2}}}{ \Gamma{ \left(\frac{3}{2}\right)}}+x^2-x$.
The exact solution of the BVP is $y(x)=x^2-x$.

We fix $d=3$.
For the mixed E-C and mixed E-MC set of nodes we set $n_e=3$.
For the equispaced nodes, we set $n=7$ and the covering $\mathcal{F}$ is realized with $q=2$.
In Figure \ref{BVP_3.1}, we display the semilog plot of the pointwise absolute error. Moreover, we have a mean absolute error of $8.30e-15$ with the equispaced set of nodes, $1.21e-13$ with the  mixed E-C set of nodes and $5.06e-14$ with the  mixed E-MC set of nodes. The condition number of the collocation matrix $\mathcal{A}$ in \eqref{matrix_BVP} is $2.35e+1$, $9.68e+2$ and $4.18e+2$, respectively.
\end{example}

\begin{figure}
	\centering
	\parbox{.33\linewidth}{\centering
		\includegraphics[width=\linewidth]{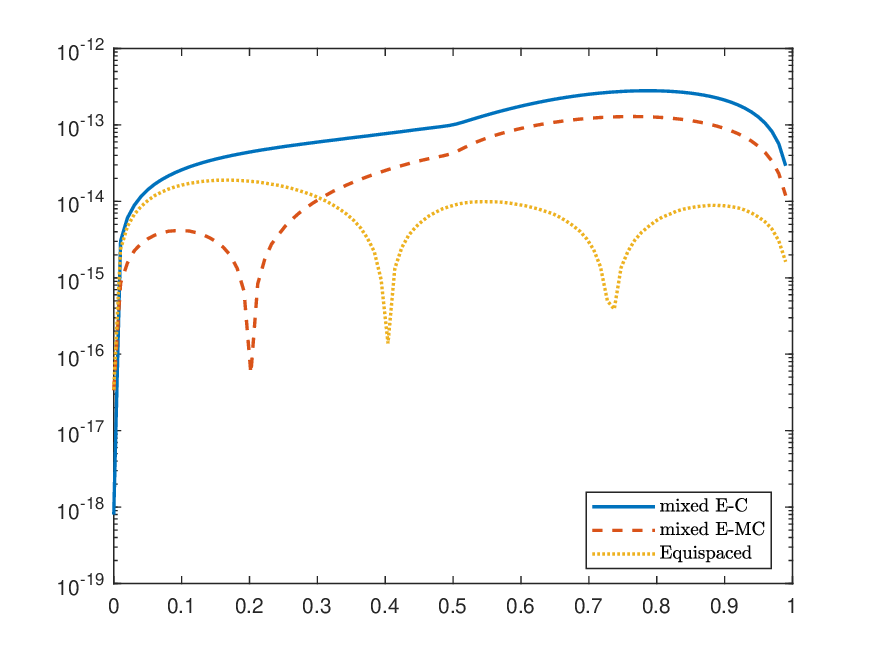}\\[-0.5ex]}
	\caption{Semilog plots of the pointwise error for the Example \ref{example_BVP_3}.}
	\label{BVP_3.1}
\end{figure}

\begin{example} \label{example_BVP_5}
    Consider the boundary value problem
    \begin{equation*}
\left\{
\begin{array}{ll}
     (D^{\frac{3}{2}}y)(x)+\frac{e^{-3\pi}}{\sqrt{\pi}}y(x)=h(x), & x\in \left(0,1\right), \\
     y(0)=0, & \\
     y(1)=-\frac{1}{40}, &
\end{array}
\right.
\end{equation*}
with $h(x)= \frac{e^{-3\pi}}{40\sqrt{\pi}}x^2(40x^3-74x+33)+\frac{\displaystyle \sqrt{x}}{70\sqrt{\pi}}(1280x^3-1036x+231)$.
The exact solution is $y(x)=\left(x^3-\frac{37}{20}x+ \frac{33}{40}\right)x^2$.

We fix $d=6$.
For the mixed E-C and mixed E-MC set of nodes we set $n_e=3$.
For the equispaced nodes, we set $n=8$ and realize the cover $\mathcal{F}$ with $q=2$.
In Figure \ref{BVP_5.1}, we display the semilog plot of the pointwise absolute error.
Moreover, we have a mean absolute error of $8.43e-16$ with the equispaced set of nodes, $4.05-14$ with the  mixed E-C set of nodes and $1.52e-14$ with the  mixed E-MC set of nodes. The condition number of the collocation matrix $\mathcal{A}$ in \eqref{matrix_BVP} is $8.87e+1$, $1.19e+3$ and $1.22e+3$, respectively.

\end{example}

\begin{figure}
	\centering
	\parbox{.33\linewidth}{\centering
		\includegraphics[width=\linewidth]{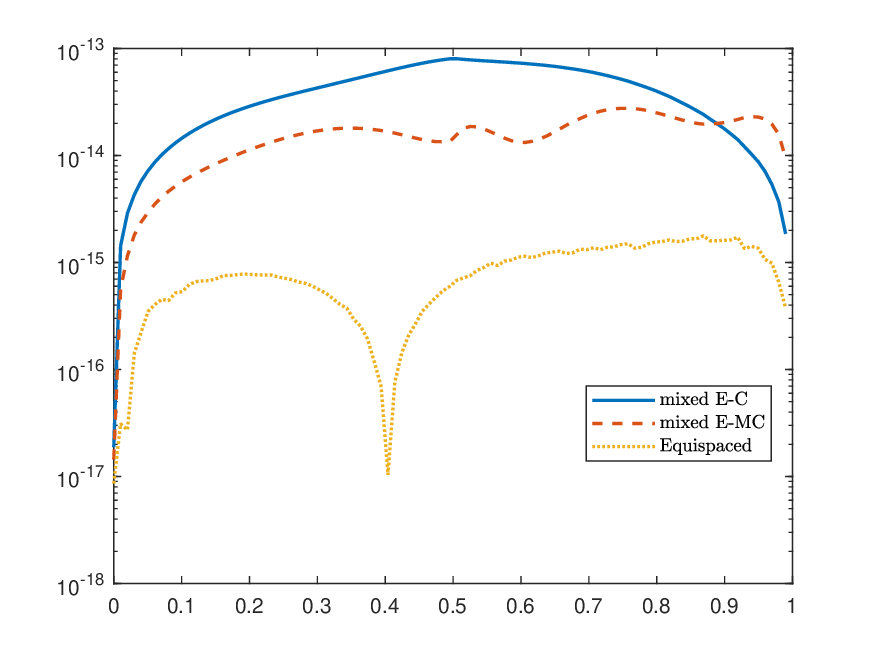}\\[-0.5ex]} 
	\caption{Semilog plots of the pointwise error for the Example \ref{example_BVP_5}. }
	\label{BVP_5.1}
\end{figure}

We highlight that the BVPs considered in Example \ref{example_BVP_1}-\ref{example_BVP_5} have an exact polynomial solution.
The obtained numerical results confirm the theoretical results of Theorem \ref{Theo1} and \ref{Theo2} on the reproduction property of the multinode Shepard operator when applied to polynomial functions.
In fact, in these cases, the error is on the order of machine error independently of the set of nodes.

The following two examples concern BVP problems with no exact polynomial solution.

\begin{example}  \label{example_BVP_4}
    Consider the boundary value problem
    \begin{equation*}
\left\{
\begin{array}{ll}
     (D^{\alpha}y)(x)+ \frac{3}{57}y(x)=h(x), & x\in \left(0,1\right), \\
     y(0)=0, & \\
     y(1)=\frac{1}{\Gamma {\left( \alpha+1\right)}}, &
\end{array}
\right.
\end{equation*}
where $1<\alpha<2$ and $h(x)= (\alpha+1)x+ \frac{3}{57}\frac{ \displaystyle x^{\alpha+1}}{\Gamma{(\alpha+1)}}$.
The exact solution is $y(x)=\frac{x^{\alpha+1}}{\Gamma{\left(\alpha +1\right)}}$.

We set $\alpha \in \left\{1.2, 1.4, 1.6,1.8\right\}$ and we vary $d$ from $3$ to $12$.
We set $n_e=3$ for the mixed E-C and mixed E-MC sets of nodes and $n=100$ for the equispaced set of nodes. 
In Figure \ref{BVP_4.3}, we display the semilog plot of the mean absolute error, and in Figure \ref{fig:cond_BVP_4} the condition number of the matrix $\mathcal{A}$ in \eqref{matrix_BVP}.
\end{example}

\begin{figure}
	\centering
         \parbox{.32\linewidth}{\centering
    \includegraphics[width=\linewidth]{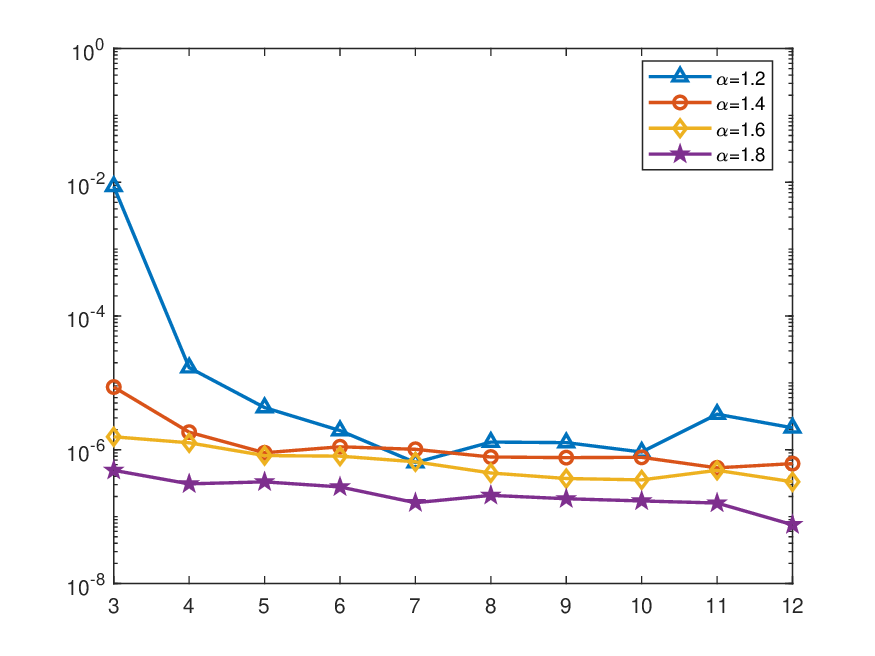}\\[-0.5ex] \scriptsize{Equispaced }}
			\parbox{.32\linewidth}{\centering
			\includegraphics[width=\linewidth]{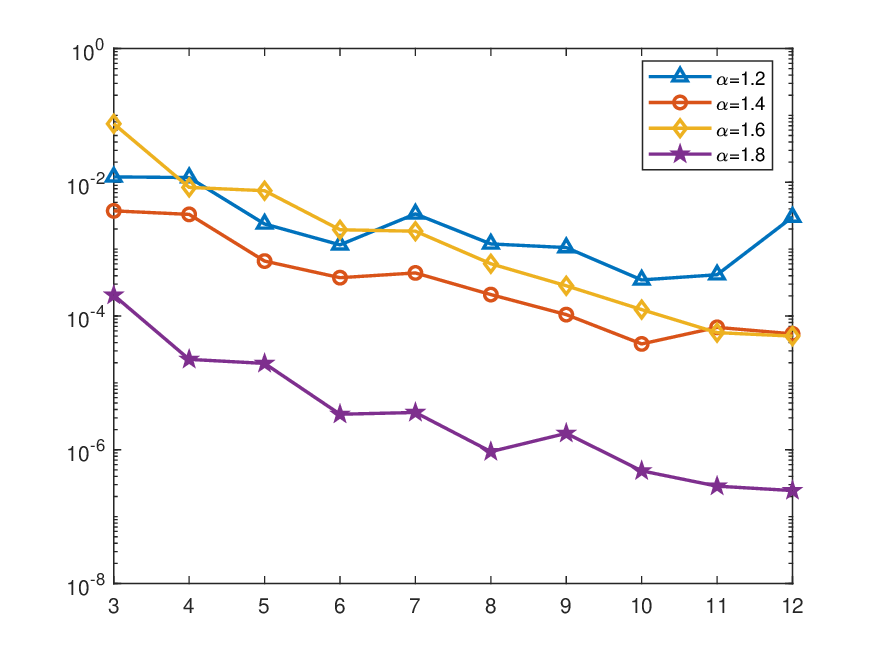}\\[-0.5ex]
			\scriptsize{Mixed E-C}}
		\parbox{.32\linewidth}{\centering
			\includegraphics[width=\linewidth]{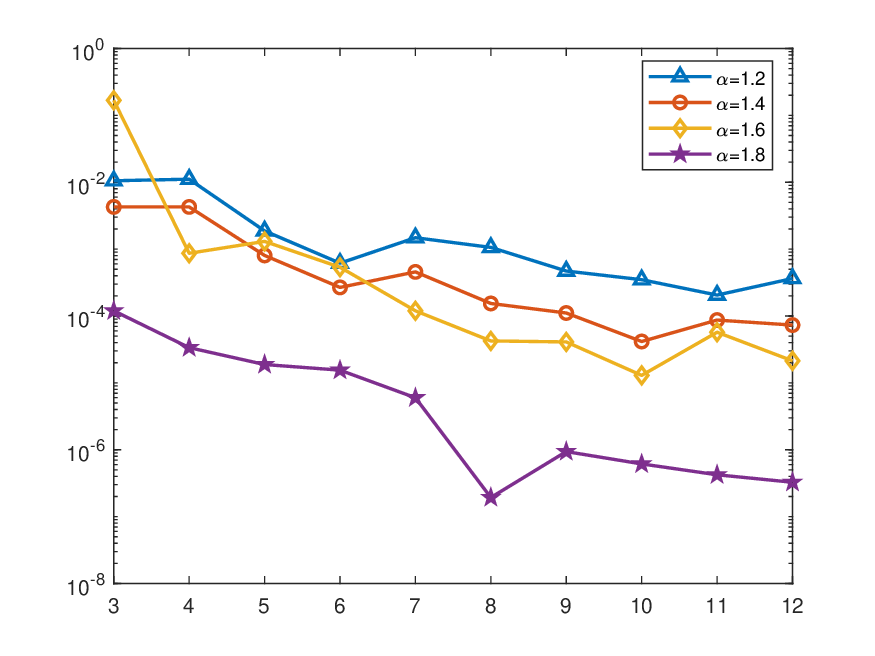}\\[-0.5ex]
			\scriptsize{Mixed E-MC}}
		\caption{Semilog plots of the mean absolute error for the Example \ref{example_BVP_4}.}
		\label{BVP_4.3}
	\end{figure}

\begin{figure}
	\centering
	\parbox{.33\linewidth}{\centering
		\includegraphics[width=\linewidth]{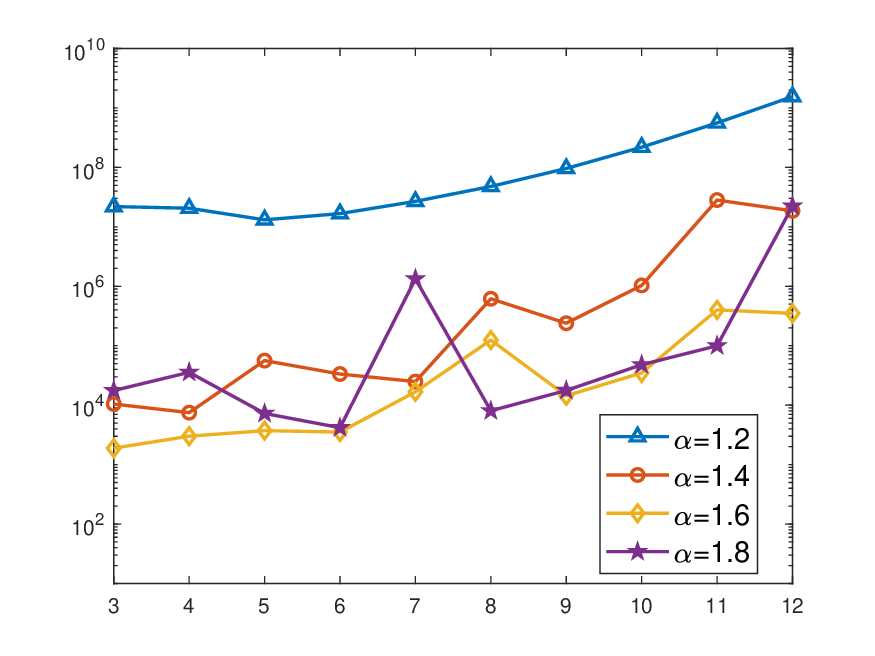}\\[-0.5ex] \scriptsize{Equispaced}}\hfill		
	\parbox{.33\linewidth}{\centering
		\includegraphics[width=\linewidth]{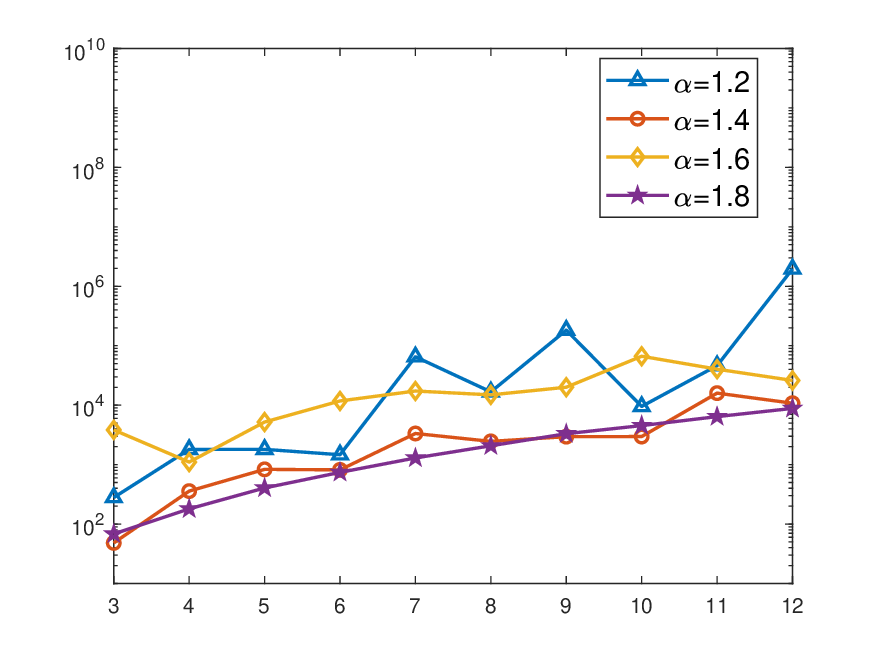}\\[-0.5ex] \scriptsize{Mixed E-C}}\hfill
	\parbox{.33\linewidth}{\centering
		\includegraphics[width=\linewidth]{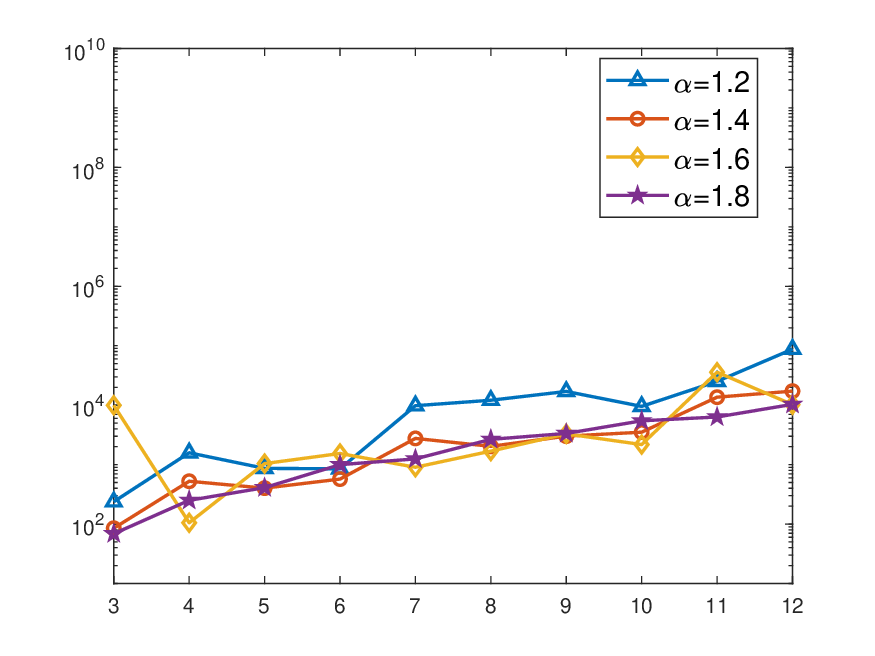}\\[-0.5ex] \scriptsize{Mixed E-MC} }
	\caption{Condition number of the matrix $\mathcal{A}$ in \eqref{matrix_BVP}  for the Example \ref{example_BVP_4}.}
	\label{fig:cond_BVP_4}
\end{figure}

\begin{example} \label{example_BVP_6}
    Consider the boundary value problem
    \begin{equation*}
\left\{
\begin{array}{ll}
     (D^{\frac{3}{2}}y)(x)+\sigma(x)y(x)=h(x), & x\in \left(0,1\right), \\
     y(0)=0, & \\
     y(1)=\frac{1}{4}, &
\end{array}
\right.
\end{equation*}
where $\sigma(x)=-\frac{(1-x)}{(1+x)^2}$ and $h(x)= \frac{x-1}{(1+x)^4 }+\frac{\sqrt{x}\sqrt{1+x}(33+26x+8x^2)+15\log{\left(\sqrt{x}+\sqrt{x+1}\right)}}{4\sqrt{\pi}(1+x)^\frac{7}{2}}$.
The exact solution is $y(x)=\frac{1}{(1+x)^2}$.

We set $\alpha=1.5$ and we vary $d$ from $3$ to $11$.
We set $n_e=6,12,18$ for the mixed E-C and mixed E-MC sets of nodes and $n=40,80,120$ for the equispaced set of nodes.
In Figure \ref{fig:BVP_6}, we show the semilog plot of the mean absolute error, and in Figure \ref{fig:cond_BVP_6} the condition number of the matrix $\mathcal{A}$ in \eqref{matrix_BVP}.

\end{example}

\begin{figure}
    \centering
    \parbox{.33\linewidth}{\centering
\includegraphics[width=\linewidth]{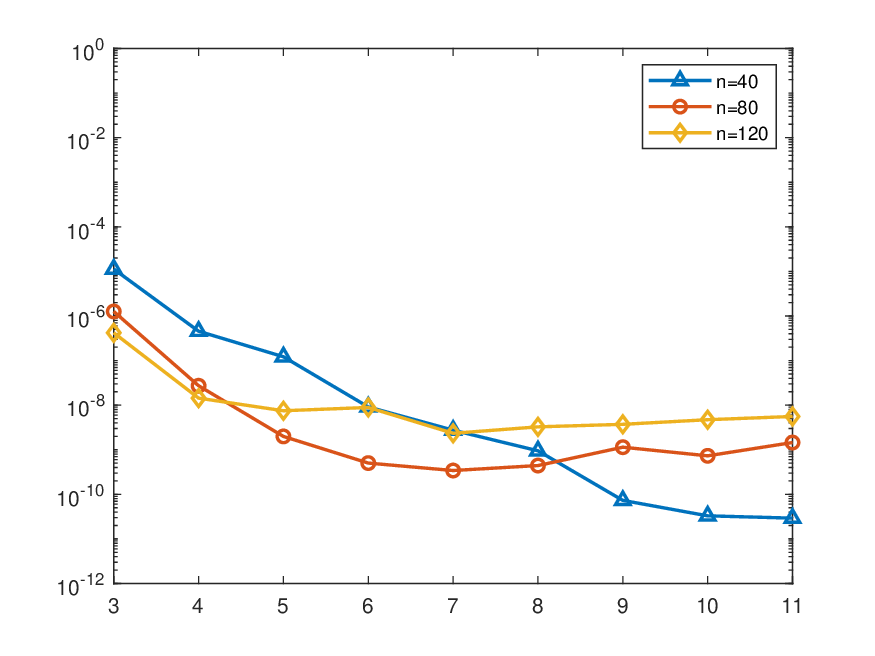}\\[-0.5ex]
    \scriptsize{Equispaced}} \hfill
    \parbox{.33\linewidth}{\centering
    \includegraphics[width=\linewidth]{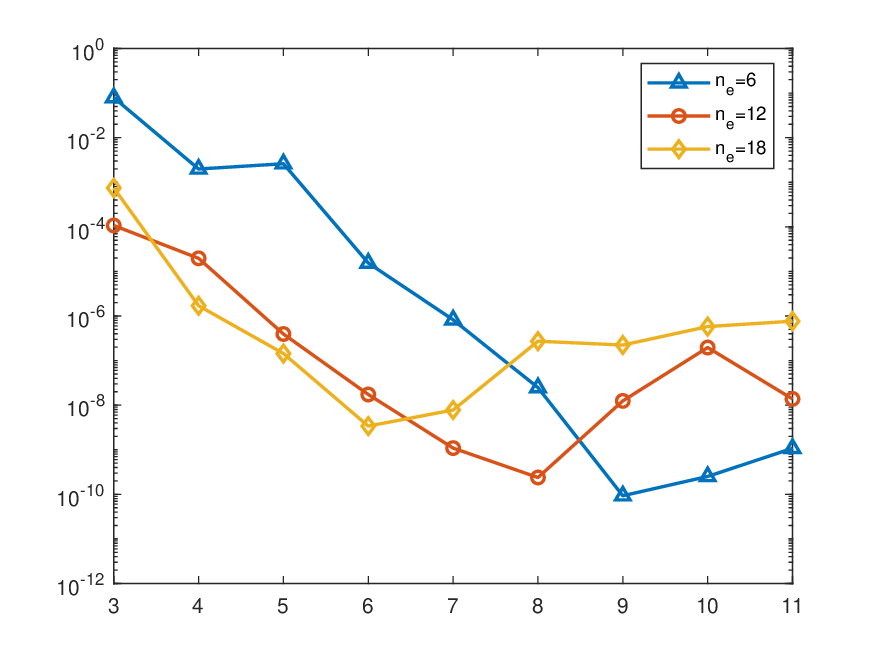}\\[-0.5ex] \scriptsize{Mixed E-C}}\hfill
    \parbox{.33\linewidth}{\centering
    \includegraphics[width=\linewidth]{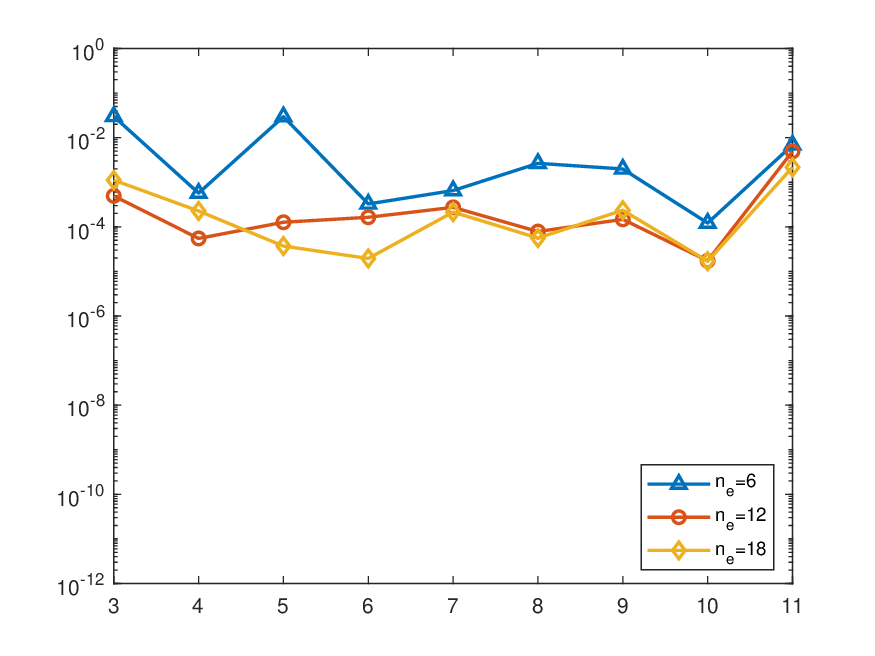}\\[-0.5ex] \scriptsize{Mixed E-MC} }
    \caption{Semilog plots of the mean absolute error for the Example \ref{example_BVP_6}.}
    \label{fig:BVP_6}
\end{figure}

\begin{figure}
	\centering
	\parbox{.33\linewidth}{\centering
		\includegraphics[width=\linewidth]{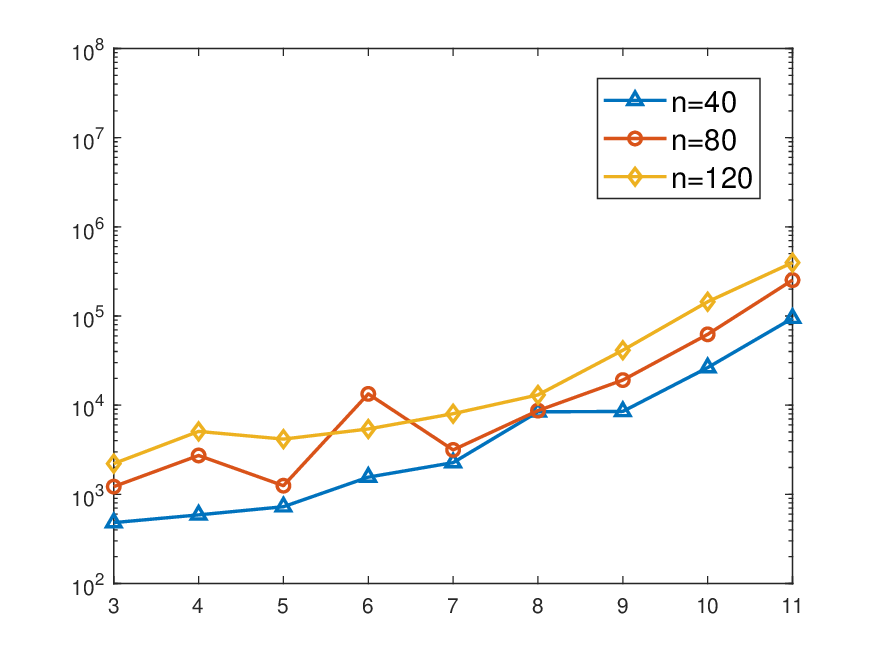}\\[-0.5ex] \scriptsize{Equispaced}}\hfill		
	\parbox{.33\linewidth}{\centering
		\includegraphics[width=\linewidth]{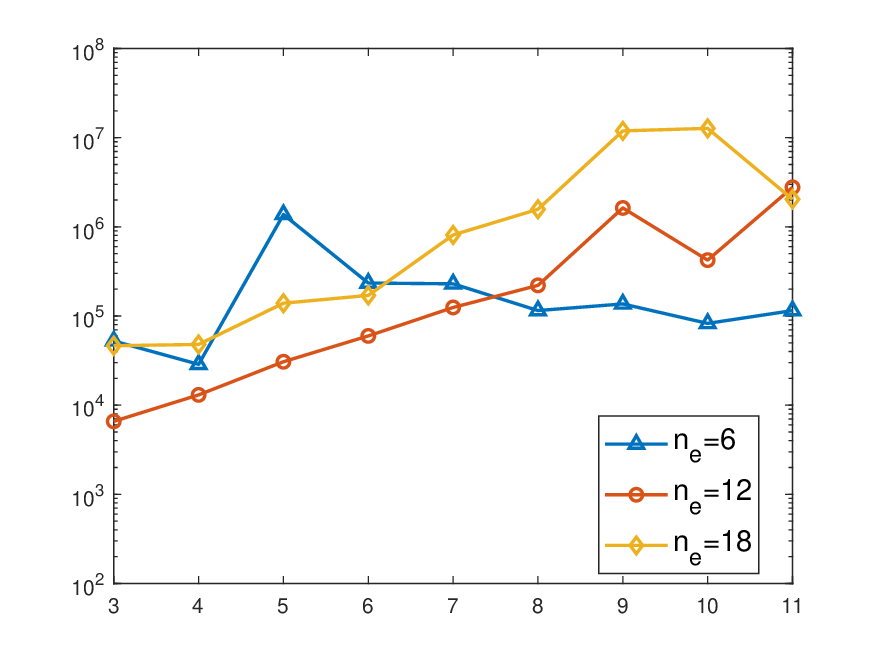}\\[-0.5ex] \scriptsize{Mixed E-C}}\hfill
	\parbox{.33\linewidth}{\centering
		\includegraphics[width=\linewidth]{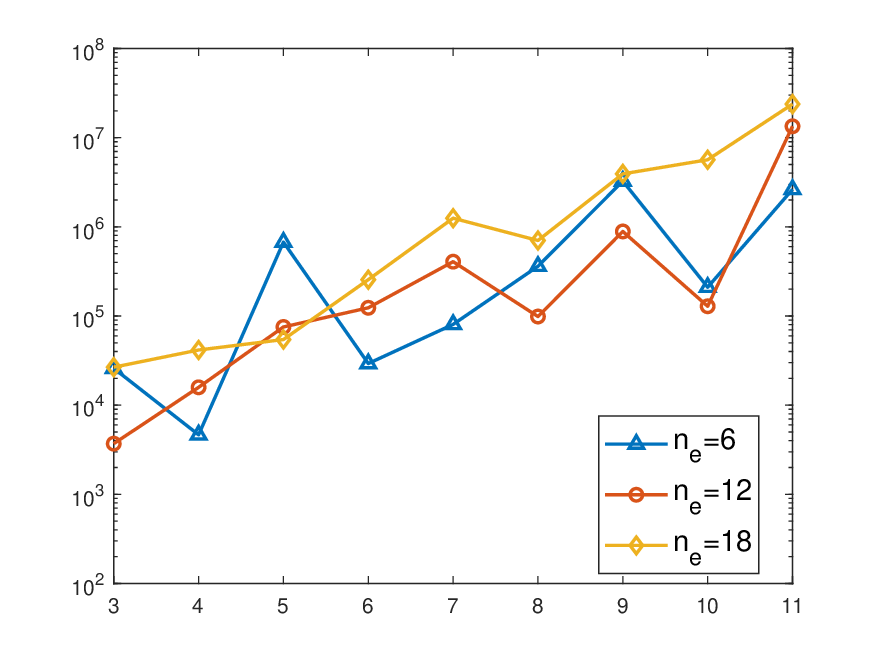}\\[-0.5ex] \scriptsize{Mixed E-MC} }
	\caption{Condition number of the matrix $\mathcal{A}$ in \eqref{matrix_BVP} for the Example \ref{example_BVP_6}.}
	\label{fig:cond_BVP_6}
\end{figure}

We notice that using sets of equispaced nodes, in Examples \ref{example_BVP_4}-\ref{example_BVP_6},  results in a lower approximation error for lower-degree polynomials. The approximation error remains comparable across all sets of nodes for higher-degree polynomials.

\subsection{Applications to the Bagley-Torvik IVPs}
We test the performance of the proposed method in numerically solving Bagley-Torvik IVPs on well-known examples \cite{Application_BTE,Fractional_derivative_legendre}.
The degree of the local polynomial interpolant and the cardinality of the node sets are specified in each example. In all cases, we compute the mean absolute error by varying both the degree of the local polynomial interpolant and the node set.
Moreover, we analyze the condition number of the collocation matrix.

\begin{example}
\label{example_IVP1}
    Consider the initial value problem
\begin{equation*}
\left\{
\begin{array}{ll}
     y^{\prime \prime}(x)+\frac{1}{100}(D^{\frac{3}{2}}y)(x)+\frac{1}{200}y(x)=h(x), & x\in \left(0,1\right], \\
     y(0)=0, & \\
     y^{\prime}(0)=0, &
\end{array}
\right.
\end{equation*}
with $h(x)=  \frac{1}{200}x^{\frac{5}{2}}+ \frac{3}{40\sqrt{\pi}}x + \frac{15}{4} \sqrt{x}$.
The exact solution is $y(x)=x^{\frac{5}{2}}$.

We vary $d$ from $3$ to $20$ and we fix $n=40,80,160$ for the equispaced set of nodes 
and  $n_e=14,21,28$ for the mixed E-C and mixed E-MC sets of nodes. In this example, we choose $n_s=3(d+1)+5$ for the mixed E-MC set of nodes. The numerical results are reported in Figure \ref{IVP_1.2} and in Figure \ref{fig:cond_IVP_1}.
\end{example}

\begin{example}
\label{example_IVP2}
     Consider the initial value problem
    \begin{equation*}
\left\{
\begin{array}{ll}
     y^{\prime \prime}(x)+\frac{1}{2}(D^{\frac{3}{2}}y)(x)+\frac{1}{2}y(x)=h(x), & x\in \left(0,1\right], \\
     y(0)=1, & \\
     y^{\prime}(0)=1, &
\end{array}
\right.
\end{equation*}
with $h(x)= \frac{3}{2}e^{x}+\frac{D(\sqrt{-x})}{\sqrt{-\pi}} $ and $D(x)$ being the Dawson function defined as
\begin{equation*}
    D(x)= e^{-x^2} \int_{0}^{x}e^{t^2} \, dt.
\end{equation*}
The exact solution is $y(x)=e^{x}$.

We vary $d$ from $3$ to $20$ 
and we fix $n=20,40,80$ for the equispaced set of nodes 
and  $n_e=3,13,23$ for the mixed E-C and mixed E-MC sets of nodes. 
The numerical results are reported in Figure 
\ref{IVP_2.2} and in Figure \ref{fig:cond_IVP_2}.
\end{example}

\begin{example}
\label{Example_IVP3}
     Consider the initial value problem
    \begin{equation*}
\left\{
\begin{array}{ll}
     y^{\prime \prime}(x)+(D^{\frac{3}{2}}y)(x)+y(x)=h(x), & x\in \left(0,1\right], \\
     y(0)=0, & \\
     y^{\prime}(0)=\omega, & \omega \in \mathbb{R},\\
\end{array}
\right.
\end{equation*}
where
\begin{align*}
h(x) &= \sin(\omega x) - \omega^2 \sin(\omega x)  + \sqrt{2} \omega^{3/2} 
\left(
\cos(\omega x) S\left(\sqrt{\frac{2}{\pi}} (\omega x)^{1/2} \right) 
\right. \\
&\quad \left.
- \sin(\omega x) C\left(\sqrt{\frac{2}{\pi}} (\omega x)^{1/2} \right)
\right)
\end{align*}
with $S(x)$ and $C(x)$ being the Fresnel functions defined as
\begin{equation*}    
S(x) = \int_{0}^{x} \sin{\left(t^2 \right)} \, dt, \quad  \quad C(x) = \int_{0}^{x} \cos{\left(t^2 \right)} \, dt.   
\end{equation*}
The exact solution is $y(x)=\sin{(\omega x})$.

We fix $\omega\in \left\{1, 2\pi, 4\pi\right\}$, we vary $d$ from $3$ to $11$, and we perform two experiments. In the first experiment, we fix $n=20,40,80,160$ for the equispaced set of nodes and we report the numerical results in Figure \ref{IVP_3.2} and Figure \ref{fig:cond_IVP_3_eq}.
In the second experiment, we consider $n_e=3$ for $\omega=1$, $n_e=5$ for $\omega=2\pi$, and $n_e=9$ for $\omega=4\pi$ for the mixed E-C and mixed E-MC sets of nodes and $n=d(n_e-1)+1$ for the equispaced set of nodes.
The numerical results are reported in Figure  \ref{IVP_3.5} and in Figure \ref{fig:cond_IVP_3}.
\end{example}

\begin{figure}
    \centering
    \parbox{.33\linewidth}{\centering
    \includegraphics[width=\linewidth]{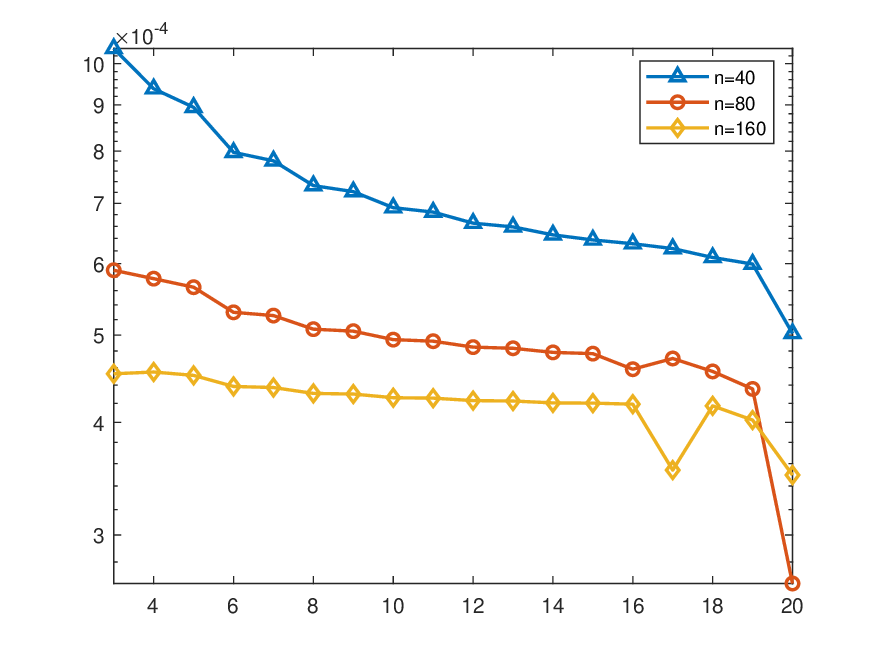}\\[-0.5ex]
    \scriptsize{Equispaced}}\hfill
    \parbox{.33\linewidth}{\centering
    \includegraphics[width=\linewidth]{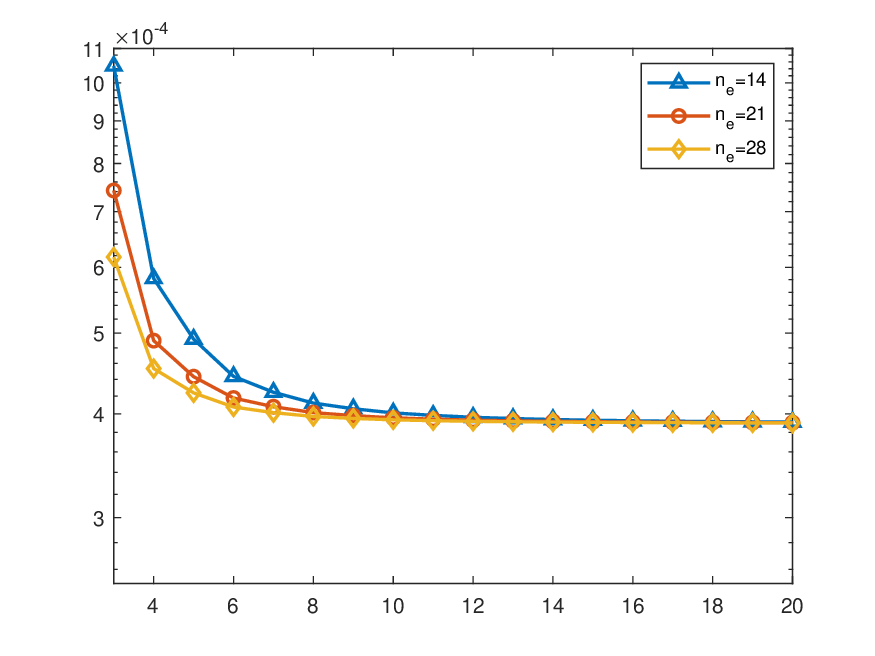}\\[-0.5ex]
    	\scriptsize{Mixed E-C}} \hfill
    \parbox{.33\linewidth}{\centering
    \includegraphics[width=\linewidth]{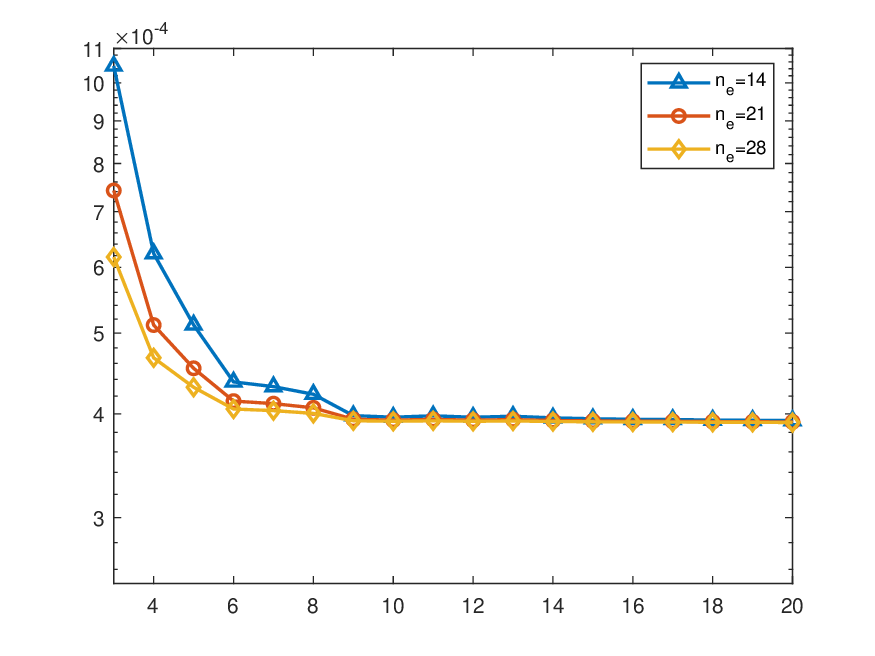}\\[-0.5ex]
    		\scriptsize{Mixed E-MC}}
    \caption{Semilog plots of the mean absolute error for the Example \ref{example_IVP1}.}
    \label{IVP_1.2}
\end{figure}

\begin{figure}
	\centering
	\parbox{.33\linewidth}{\centering
		\includegraphics[width=\linewidth]{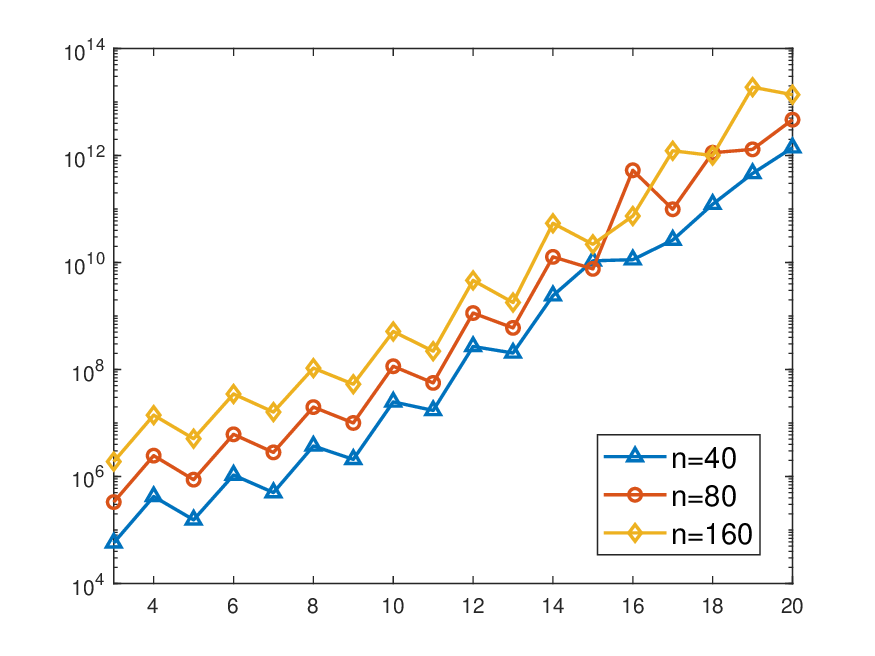}\\[-0.5ex] \scriptsize{Equispaced}}\hfill		
	\parbox{.33\linewidth}{\centering
		\includegraphics[width=\linewidth]{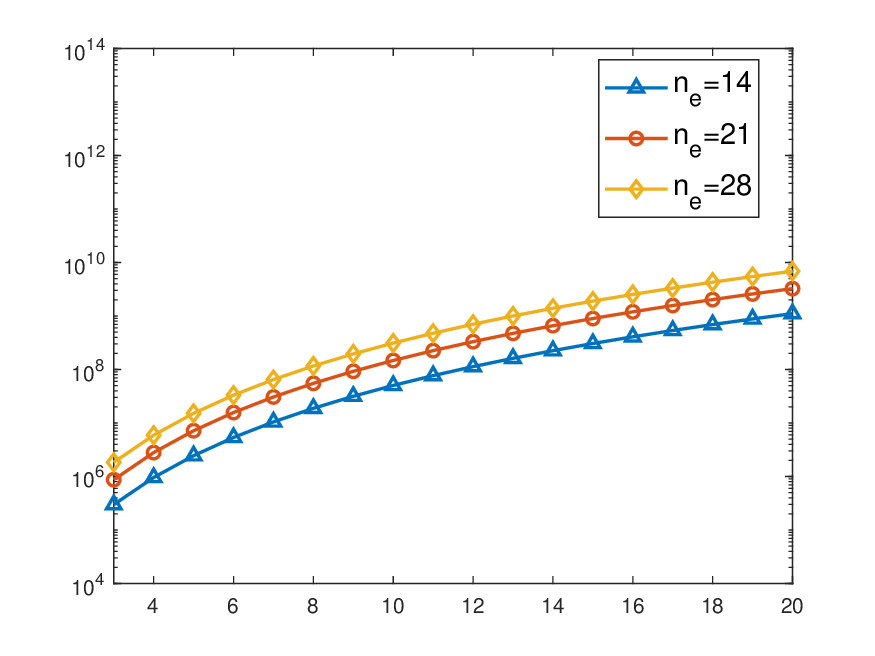}\\[-0.5ex] \scriptsize{Mixed E-C}}\hfill
	\parbox{.33\linewidth}{\centering
		\includegraphics[width=\linewidth]{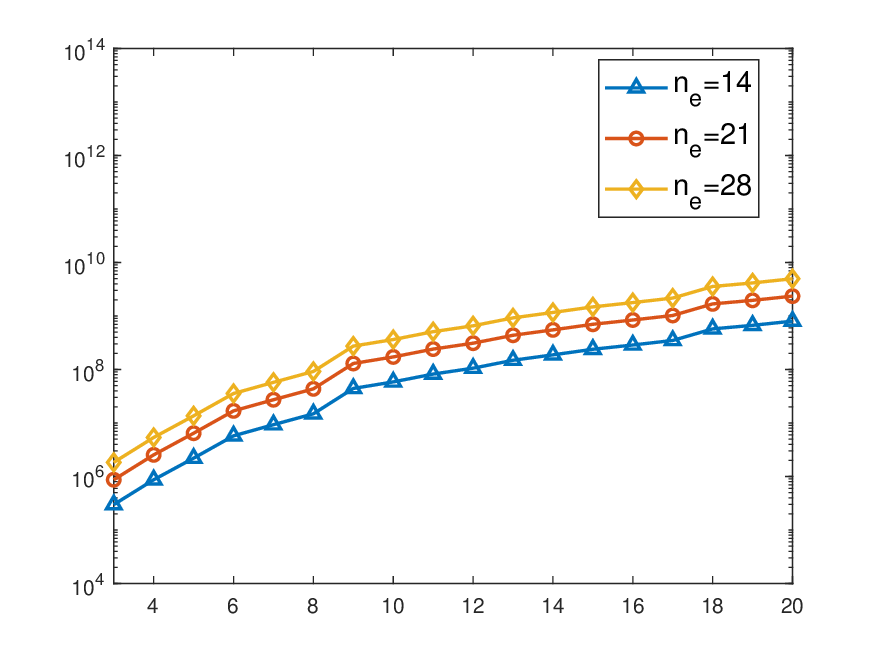}\\[-0.5ex] \scriptsize{Mixed E-MC} }
	\caption{Condition number of the matrix $\mathcal{A}$ in \eqref{matrix_IVP} for the Example \ref{example_IVP1}.}
	\label{fig:cond_IVP_1}
\end{figure}

\begin{figure}
    \centering
    \parbox{.33\linewidth}{\centering
    \includegraphics[width=\linewidth]{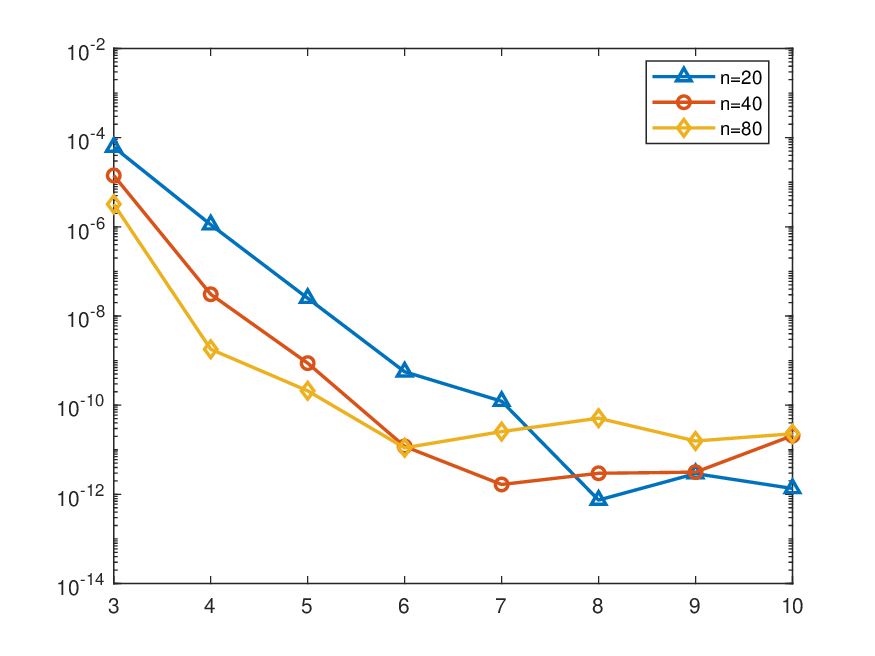}\\[-0.5ex]
    \scriptsize{Equispaced}}\hfill
    \parbox{.33\linewidth}{\centering
    \includegraphics[width=\linewidth]{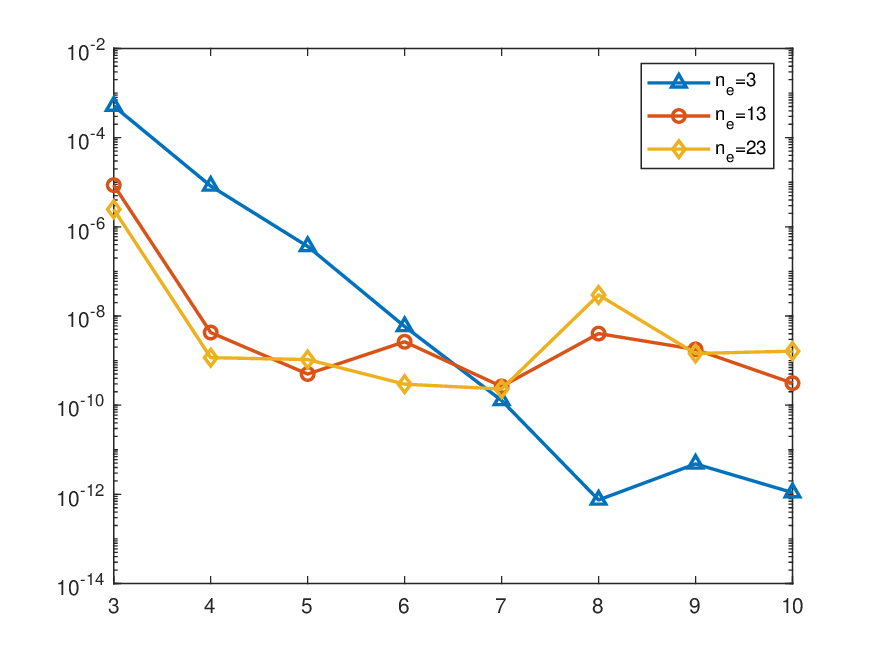}\\[-0.5ex]
    	\scriptsize{Mixed E-C}} \hfill
    \parbox{.33\linewidth}{\centering
    \includegraphics[width=\linewidth]{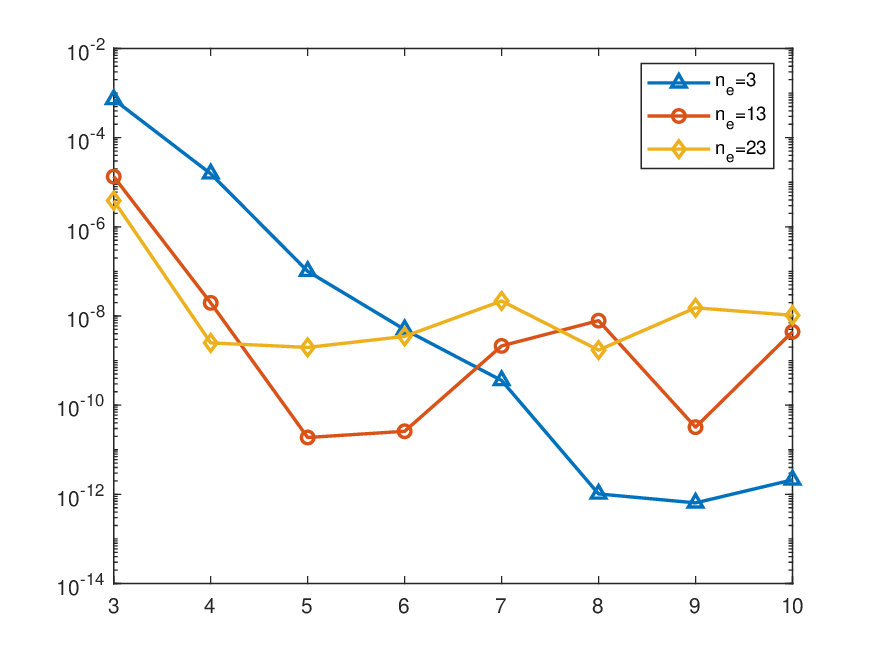}\\[-0.5ex]
    	\scriptsize{Mixed E-MC}}
    \caption{Semilog plots of the mean absolute error for the Example \ref{example_IVP2}.}
    \label{IVP_2.2}
\end{figure}

\begin{figure}
	\centering
	\parbox{.33\linewidth}{\centering
		\includegraphics[width=1\linewidth]{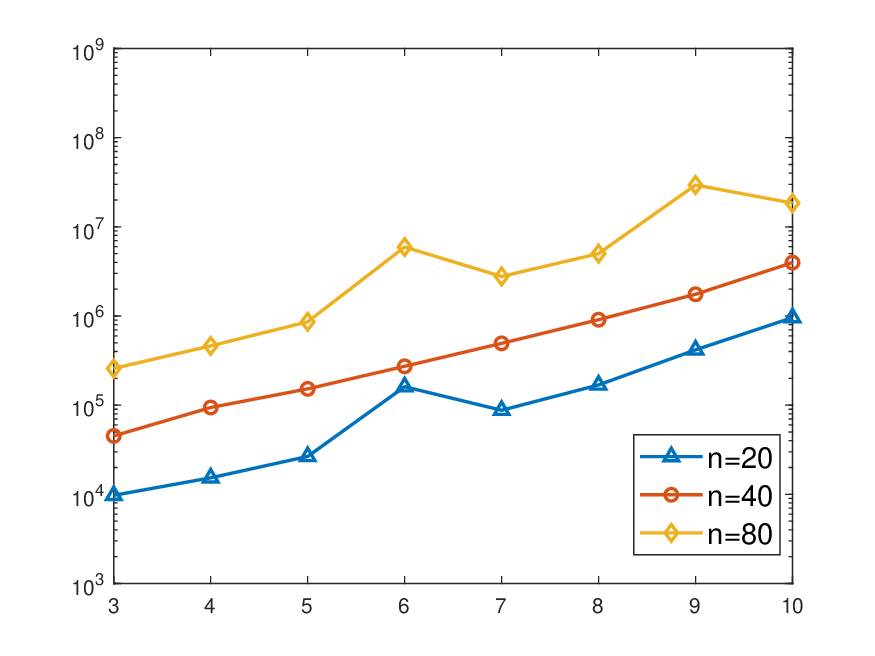}\\[-0.5ex] \scriptsize{Equispaced}}\hfill		
	\parbox{.33\linewidth}{\centering
		\includegraphics[width=1\linewidth]{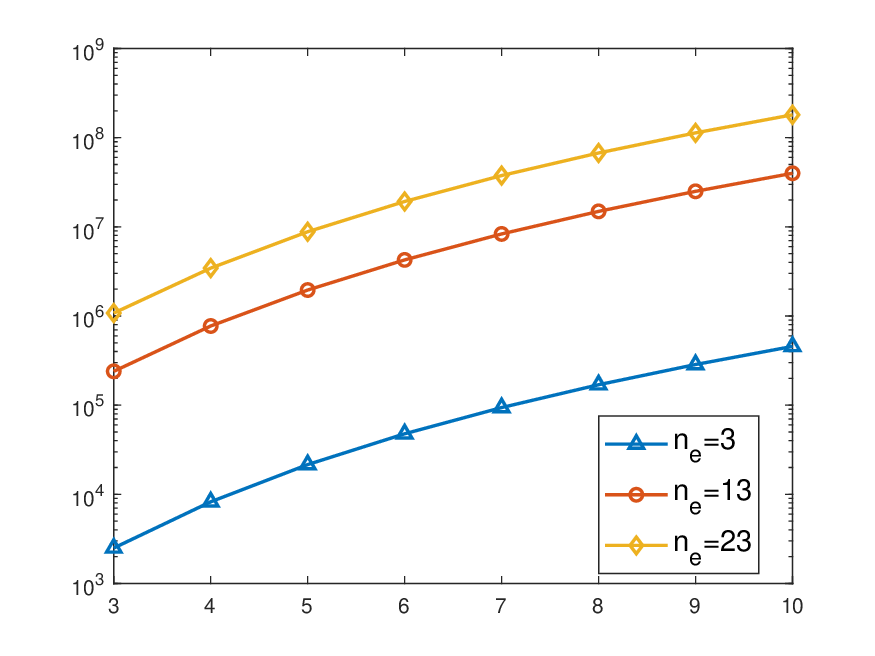}\\[-0.5ex] \scriptsize{Mixed E-C}}\hfill
	\parbox{.33\linewidth}{\centering
		\includegraphics[width=1\linewidth]{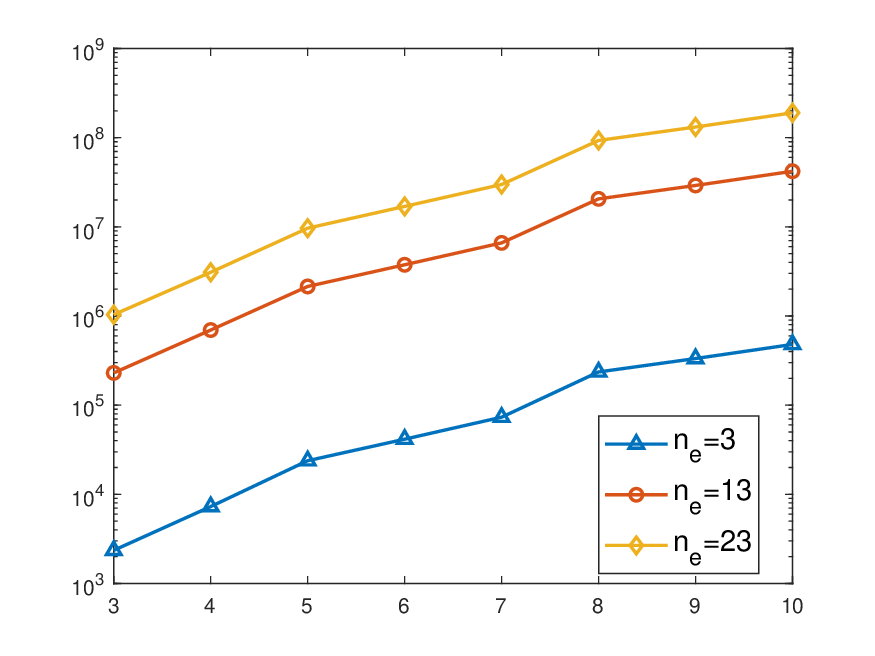}\\[-0.5ex] \scriptsize{Mixed E-MC} }
	\caption{Condition number of the matrix $\mathcal{A}$ in \eqref{matrix_IVP} for the Example \ref{example_IVP2}.}
	\label{fig:cond_IVP_2}
\end{figure}

\begin{figure}
\begin{center}
	\begin{minipage}{0.32\textwidth}
	\centering
			\includegraphics[width=\linewidth]{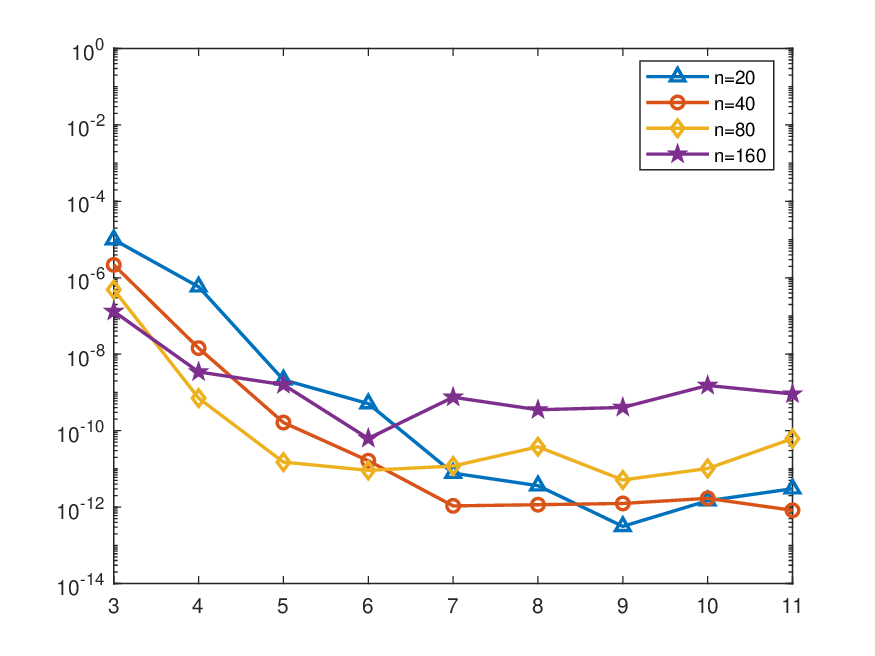}\\[-0.5ex] \scriptsize{$\omega=1$} 
		
	\end{minipage}
	\hfill
	\begin{minipage}{0.32\textwidth}
		\centering
			\includegraphics[width=\linewidth]{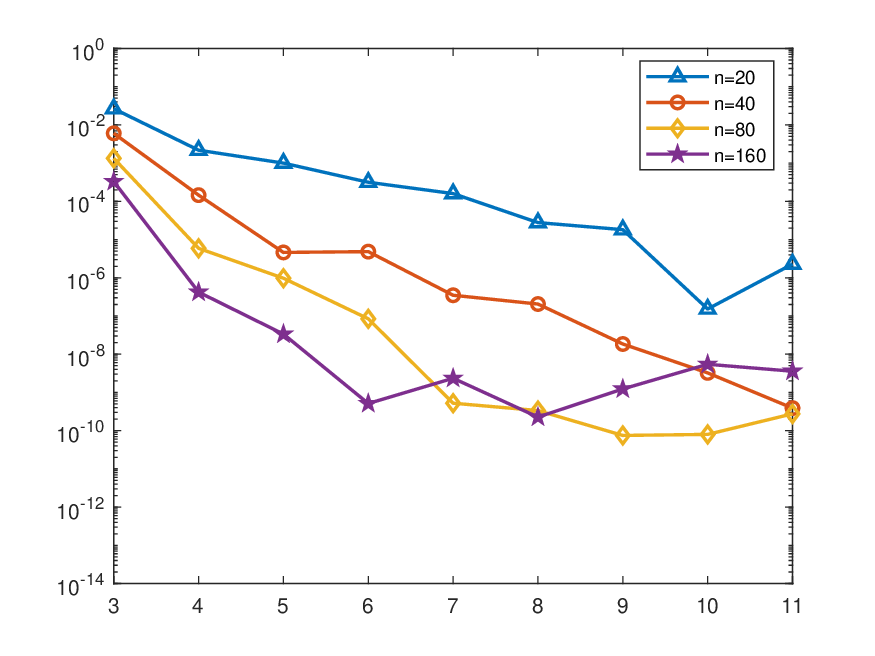}\\[-0.5ex] \scriptsize{$\omega=2\pi$}
	\end{minipage}
	\hfill
	\begin{minipage}{0.32\textwidth}
		\centering
			\includegraphics[width=\linewidth]{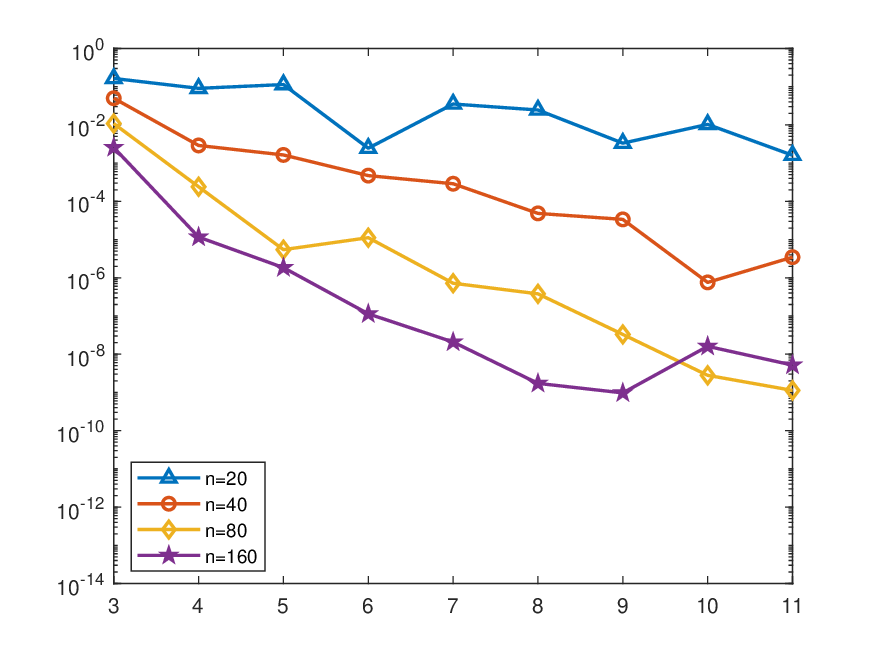}\\[-0.5ex] \scriptsize{$\omega=4\pi$}
		
	\end{minipage}
\end{center}
    \caption{Semilog plots of the mean absolute error for the Example \ref{Example_IVP3} by using equispaced sets of nodes.}
    \label{IVP_3.2}
\end{figure}

\begin{figure}
	\centering
	\parbox{.33\linewidth}{\centering
		\includegraphics[width=\linewidth]{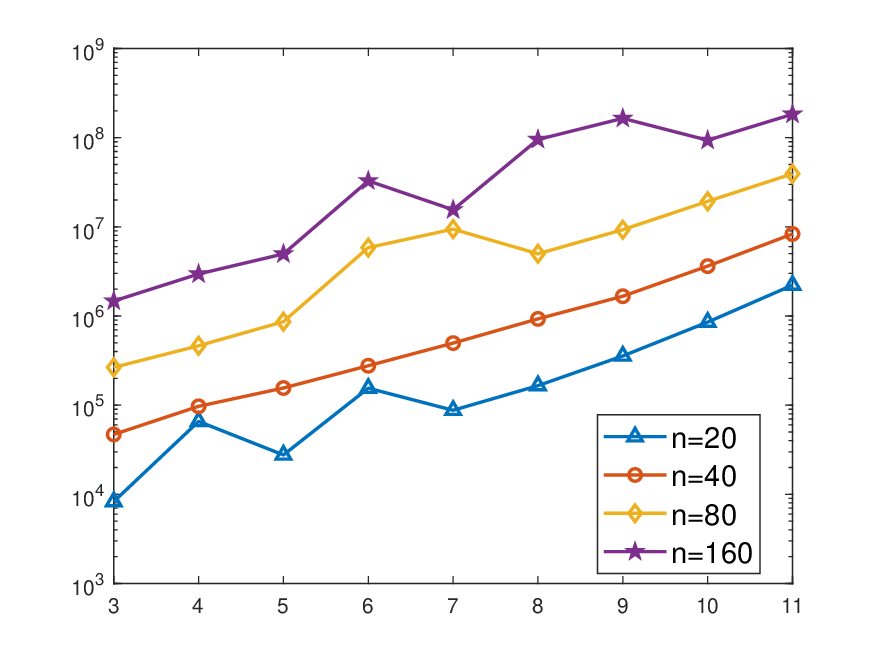}\\[-0.5ex] \scriptsize{$\omega=1$}}\hfill		
	\parbox{.33\linewidth}{\centering
		\includegraphics[width=\linewidth]{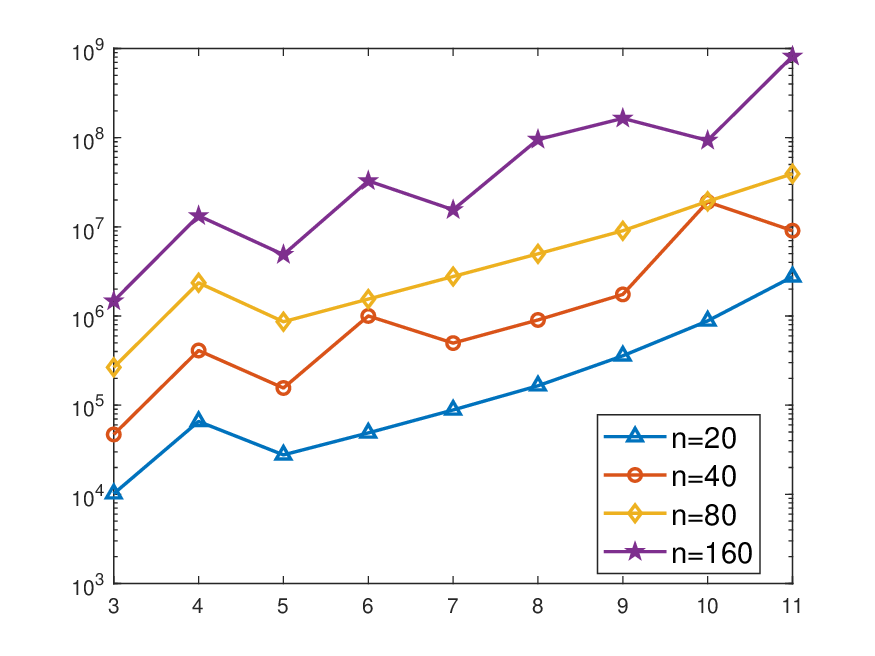}\\[-0.5ex] \scriptsize{$\omega=2\pi$}}\hfill
	\parbox{.33\linewidth}{\centering
		\includegraphics[width=\linewidth]{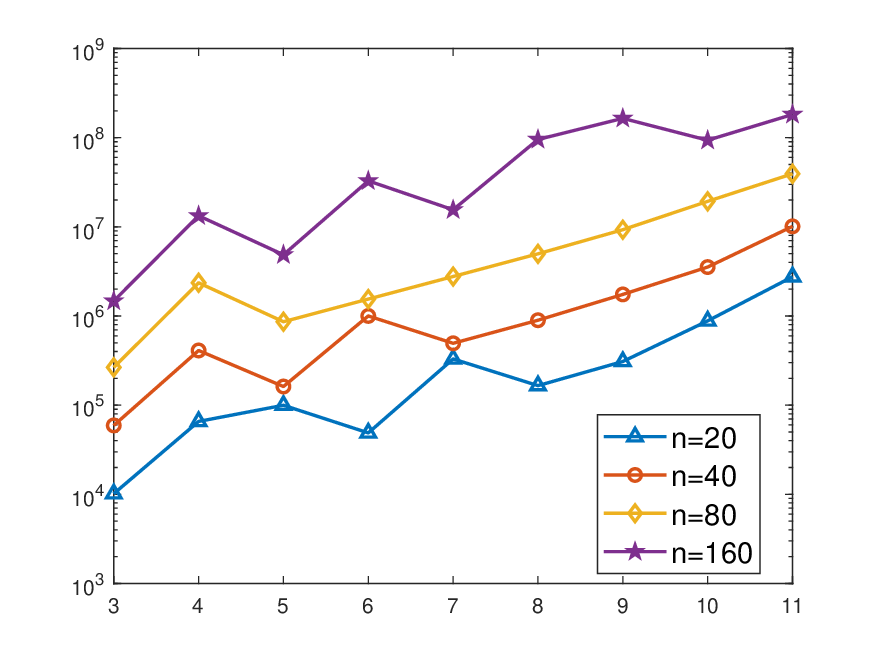}\\[-0.5ex] \scriptsize{$\omega=4\pi$} }
	\caption{Condition number of the matrix $\mathcal{A}$ in \eqref{matrix_IVP} for the Example \ref{Example_IVP3} by using equispaced sets of nodes. }
	\label{fig:cond_IVP_3_eq}
\end{figure}

\begin{figure}
	\centering
	\parbox{.32\linewidth}{\centering
	\includegraphics[width=1\linewidth]{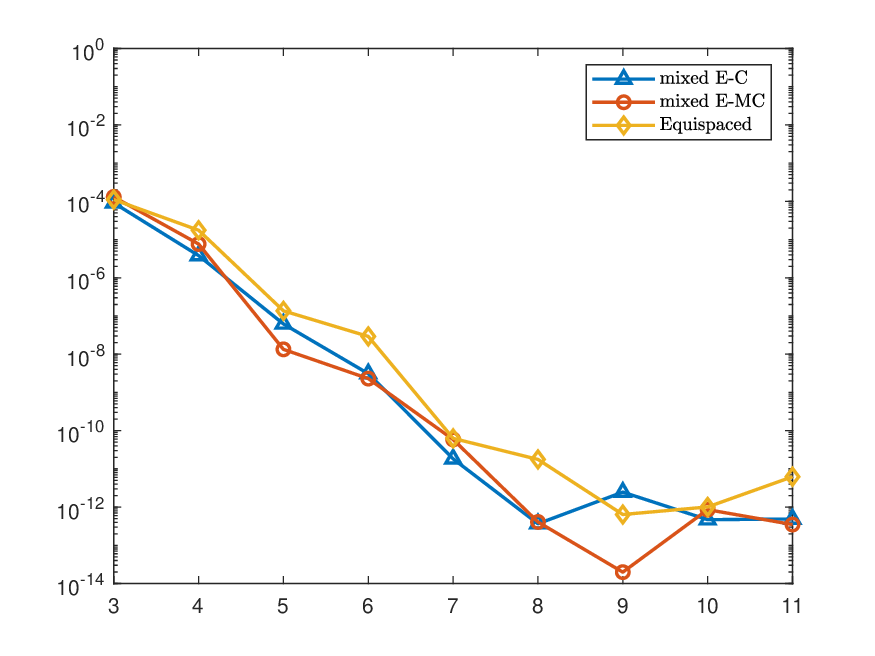}\\[-0.5ex] \scriptsize{$\omega=1$,$n_e=3$}}
	\hfill
	\parbox{.32\linewidth}{\centering
	\includegraphics[width=1\linewidth]{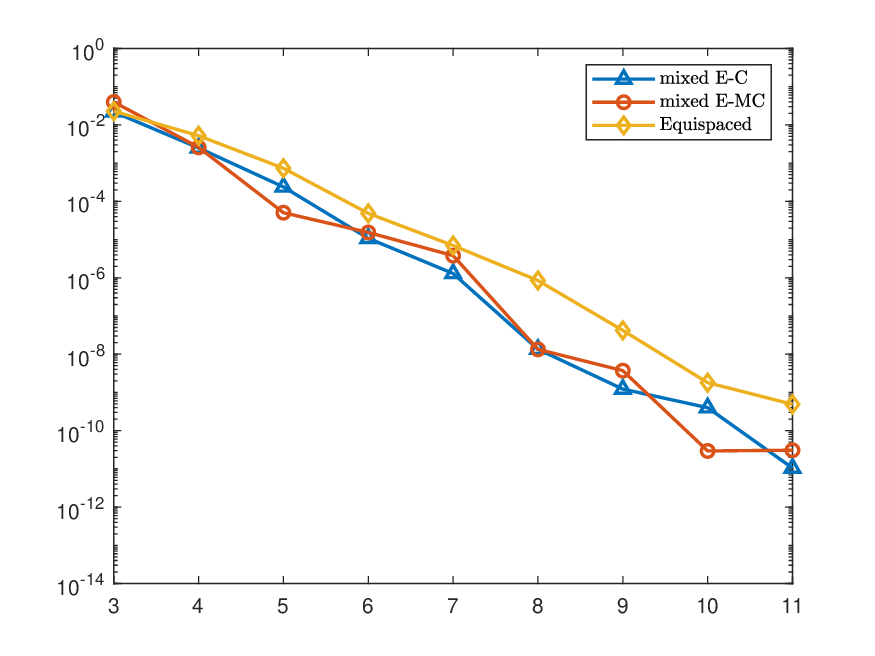}\\[-0.5ex]\scriptsize{$\omega=2\pi$,$n_e=5$}}
	\hfill
	\parbox{.32\linewidth}{\centering
	\includegraphics[width=1\linewidth]{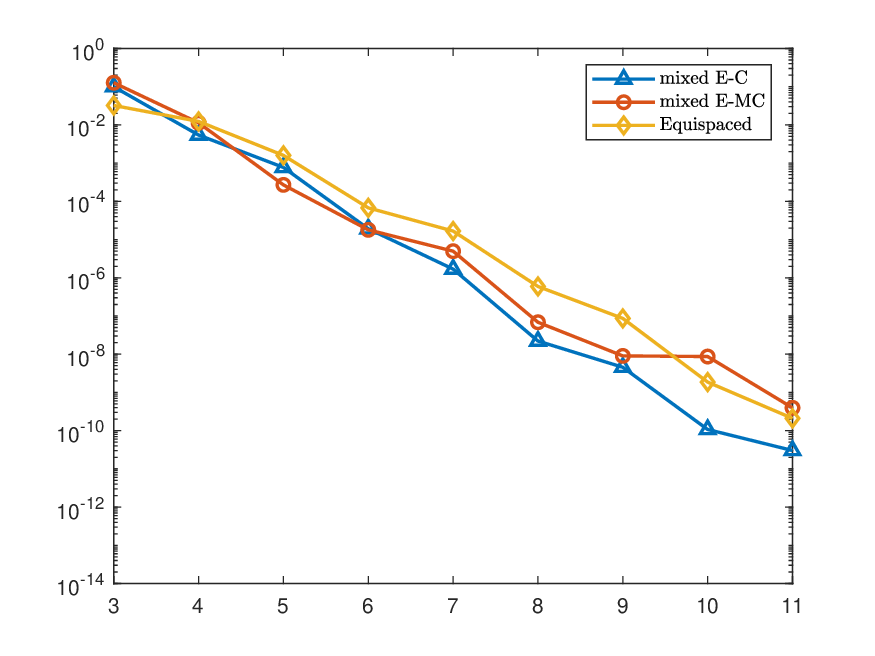}\\[-0.5ex] \scriptsize{$\omega=4\pi$,$n_e=9$}}
	 
	\caption{Semilog plots of the mean absolute error, for the Example \ref{Example_IVP3} by using the different sets of nodes.}
	\label{IVP_3.5}
\end{figure}

\begin{figure}
	\centering
	\parbox{.33\linewidth}{\centering
		\includegraphics[width=\linewidth]{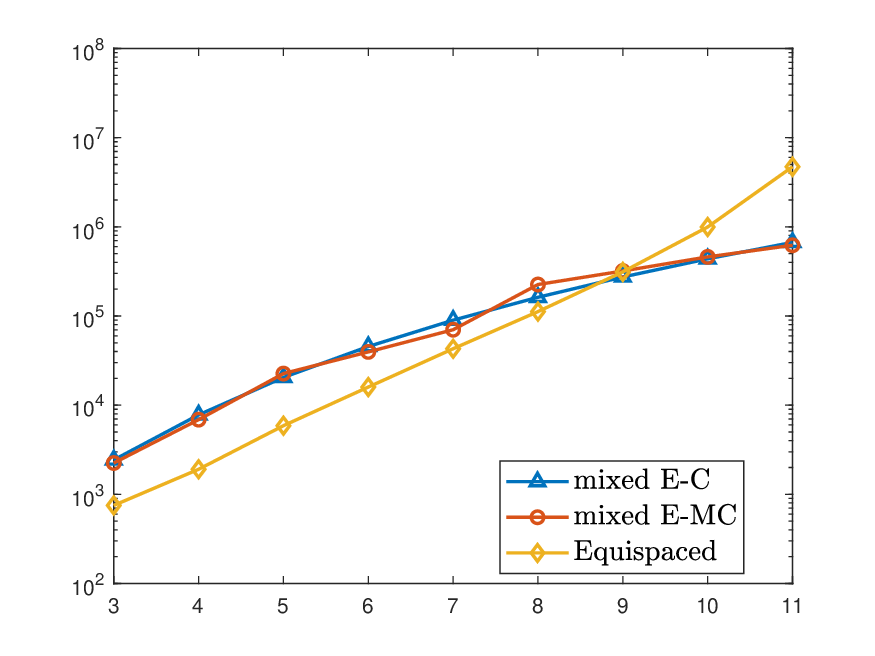}\\[-0.5ex] \scriptsize{$\omega=1,n_e=3$}}\hfill		
	\parbox{.33\linewidth}{\centering
		\includegraphics[width=\linewidth]{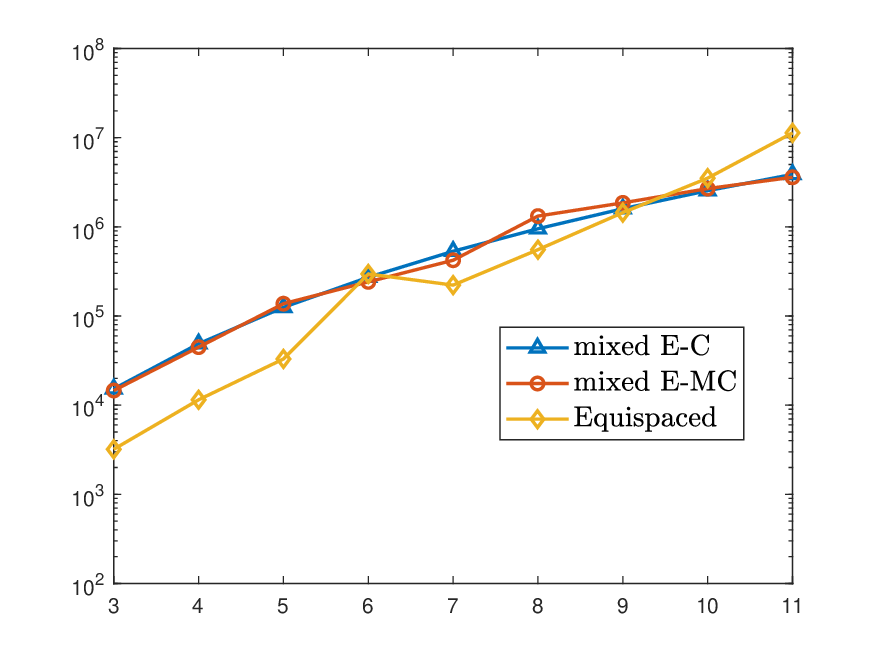}\\[-0.5ex] \scriptsize{$\omega=2\pi,n_e=5$}}\hfill
	\parbox{.33\linewidth}{\centering
		\includegraphics[width=\linewidth]{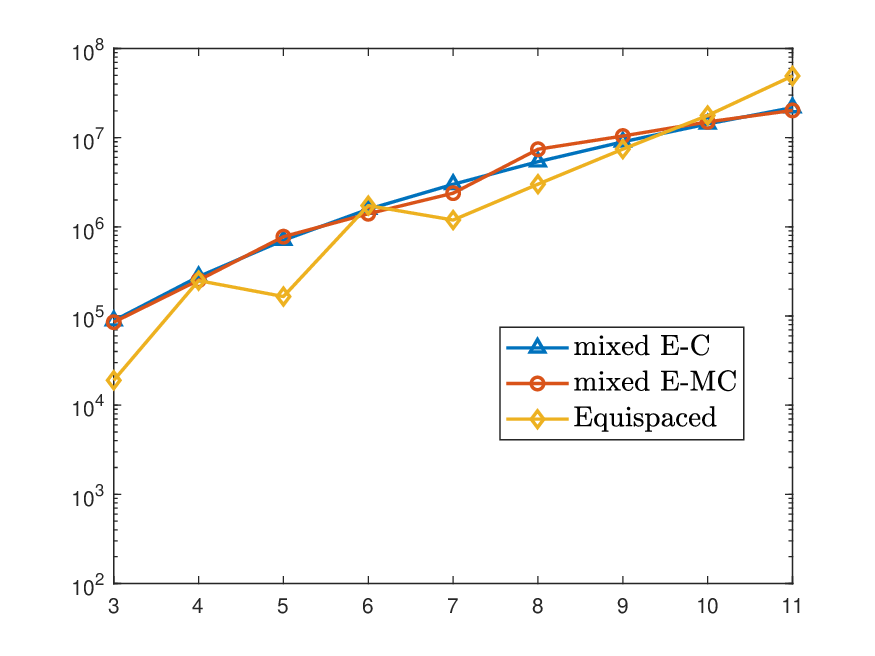}\\[-0.5ex] \scriptsize{$\omega=4\pi,n_e=9$} }
	\caption{Condition number of the matrix $\mathcal{A}$ in \eqref{matrix_IVP} for the Example \ref{Example_IVP3} y using the different sets of nodes.}
	\label{fig:cond_IVP_3}
\end{figure}

As in the case of the BVPs of Examples \ref{example_BVP_4}-\ref{example_BVP_6}, using sets of equispaced nodes results in a lower approximation error for lower-degree polynomials. For higher-degree polynomials, the approximation error remains comparable across all sets of nodes.

\begin{acknowledgements}
 This research has been achieved as part of RITA “Research ITalian network on Approximation”, as part of the UMI group “Teoria dell’Approssimazione e Applicazioni” and was supported by INDAM-GNCS project 2025. The first two authors are members of the INdAM Research group GNCS.
 \end{acknowledgements}

 \section*{\small
 Conflict of interest} 

 {\small
 The authors declare that they have no conflict of interest.}


\bibliographystyle{spmpsci}
\bibliography{bibliografia}

\begin{thebibliography}{10}
\providecommand{\url}[1]{{#1}}
\providecommand{\urlprefix}{URL }
\expandafter\ifx\csname urlstyle\endcsname\relax
  \providecommand{\doi}[1]{DOI~\discretionary{}{}{}#1}\else
  \providecommand{\doi}{DOI~\discretionary{}{}{}\begingroup
  \urlstyle{rm}\Url}\fi

\bibitem{Fractional_derivative_BTE}
ur~Rehman, M., Khan, R.A.: A numerical method for solving boundary value
  problems for fractional differential equations.
\newblock Applied Mathematical Modelling \textbf{36}(3), 894--907 (2012)

\bibitem{math7050407}
Garrappa, R., Kaslik, E., Popolizio, M.: Evaluation of {F}ractional {I}ntegrals
  and {D}erivatives of {E}lementary {F}unctions: {O}verview and {T}utorial.
\newblock Mathematics \textbf{7}(5) (2019)

\bibitem{Fractional_derivative_book}
Li, X., Chen, W., Sun, H.: Fractional {D}erivative {M}odeling in {M}echanics
  and {E}ngineering.
\newblock Springer Nature Singapore (2022)

\bibitem{esempi_numerici}
Pulido, M.A.P., Sousa, J.V.C., de~Oliveira, E.C.: New discretization of
  {$\psi$-C}aputo fractional derivative and applications.
\newblock Mathematics and Computer in Simulation \textbf{221}(C), 135–158
  (2024)

\bibitem{Application_BTE}
De~Bonis, M.C., Occorsio, D.: A {G}lobal {M}ethod for {A}pproximating {C}aputo
  {F}ractional {D}erivatives-{A}n {A}pplication to the {B}agley-{T}orvik
  {E}quation.
\newblock Axioms \textbf{13}(11) (2024)

\bibitem{article}
Mahmudov, N., Huseynov, I., Aliyev, F., Aliyev, N.: Analytical approach to a
  class of {B}agley-{T}orvik equations.
\newblock TWMS Journal of Pure \& Applied Mathematics pp. 238--258 (2020)

\bibitem{article2}
Garrappa, R.: Numerical {S}olution of {F}ractional {D}ifferential {E}quations:
  {A} {S}urvey and a {S}oftware {T}utorial.
\newblock Mathematics \textbf{6}, 16 (2018)

\bibitem{article3}
Sun, Y., Zeng, Z., Song, J.: Existence and {U}niqueness for the {B}oundary
  {V}alue {P}roblems of {N}onlinear {F}ractional {D}ifferential {E}quation.
\newblock Applied Mathematics \textbf{08}, 312--323 (2017)

\bibitem{Fractional_derivative_legendre}
Jafari, H., Yousefi, S., Firoozjaee, M., Momani, S., Khalique, C.: Application
  of {L}egendre wavelets for solving fractional differential equations.
\newblock Computers and Mathematics with Applications \textbf{62}(3),
  1038--1045 (2011)

\bibitem{entropy}
Zafar, A.A., Kudra, G., Awrejcewicz, J.: An {I}nvestigation of {F}ractional
  {B}agley–{T}orvik {E}quation.
\newblock Entropy \textbf{22}(1) (2020)

\bibitem{article4}
Diethelm, K., Ford, N.: Numerical {S}olution of the {B}agley-{T}orvik equation.
\newblock BIT Numerical Mathematics \textbf{42}, 490--507 (2002)

\bibitem{DEMIRCI20122754}
Demirci, E., Ozalp, N.: A method for solving differential equations of
  fractional order.
\newblock Journal of Computational and Applied Mathematics \textbf{236}(11),
  2754--2762 (2012)

\bibitem{DellAccio2016}
Dell'Accio, F., Di~Tommaso, F., Hormann, K.: Multinode rational operators for
  univariate interpolation.
\newblock In: Y.D. Sergeyev, D.E. Kvasov, F.~Dell'Accio, M.S. Mukhametzhanov
  (eds.) Numerical Computations: Theory and Algorithms (NUMTA-2016), \emph{AIP
  Conference Proceedings}, vol. 1776, pp. 070010:1--4. AIP Publishing, Pizzo
  Calabro, Italy (2016)

\bibitem{DellAccio:2018}
Dell’Accio, F., {Di Tommaso}, F., Hormann, K.: Reconstruction of a function
  from {H}ermite-{B}irkhoff data.
\newblock Applied Mathematics and Computation \textbf{318}, 51--69 (2018)

\bibitem{BagleyTorvik1984}
Torvik, P.J., Bagley, R.L.: On the {A}ppearance of the {F}ractional
  {D}erivative in the {B}ehavior of {R}eal {M}aterials.
\newblock Journal of Applied Mechanics \textbf{51}(2), 294--298 (1984)

\bibitem{Gautschi}
Gautschi, W.: Numerical analysis: an introduction.
\newblock Birkhauser Boston Inc., USA (1997)

\bibitem{Bagley1983FractionalCD}
Bagley, R.L., Torvik, P.J.: Fractional calculus-{A} different approach to the
  analysis of viscoelastically damped structures.
\newblock AIAA Journal  (1983)

\bibitem{DEMARCHI2015}
{De Marchi}, S., Dell’Accio, F., Mazza, M.: On the constrained
  mock-{C}hebyshev least-squares.
\newblock Journal of Computational and Applied Mathematics \textbf{280},
  94--109 (2015)

\bibitem{DELLACCIO2022}
Dell’Accio, F., {Di Tommaso}, F., Nudo, F.: Generalizations of the
  constrained mock-chebyshev least squares in two variables: Tensor product vs
  total degree polynomial interpolation.
\newblock Applied Mathematics Letters \textbf{125}, 107732 (2022)

\bibitem{oldham2006fractional}
Oldham, K., Spanier, J.: The {F}ractional {C}alculus: {T}heory and
  {A}pplications of {D}ifferentiation and {I}ntegration to {A}rbitrary {O}rder.
\newblock Academic Press (1974)

\bibitem{some_fractional_derivatives}
Sikora, B.: Remarks on the {C}aputo fractional derivative.
\newblock Matematyka I Informatyka Na Uczelniach Technicznych \textbf{5},
  78--84 (2023)

\bibitem{BTE_Bessel}
\c{S}. Y\"{u}zba\c{s}\i: Numerical solution of the {B}agley–{T}orvik equation
  by the {B}essel collocation method.
\newblock Mathematical Methods in the Applied Sciences \textbf{36} (2013)

\end{thebibliography}

\bigskip  

\small 
\noindent
{\bf Publisher's Note}
Springer Nature remains neutral with regard to jurisdictional claims in published maps and institutional affiliations.

\end{document}